\documentclass[a4paper,12pt]{article}
\usepackage[utf8]{inputenc}
\usepackage[english]{babel}
\usepackage{amsmath, amsfonts, amsthm, amssymb}
\usepackage{caption,subcaption}
\usepackage{color,soul}
\usepackage{enumerate}
\usepackage{float}
\usepackage[margin=2.5cm]{geometry}
\usepackage{graphicx}
\usepackage{hyperref}
\usepackage[nameinlink]{cleveref}
\usepackage{mathrsfs}
\usepackage{pdfpages}
\usepackage{tikz}
\usetikzlibrary{calc,patterns,angles,quotes,arrows.meta,decorations.pathreplacing}
\usepackage{tkz-euclide}
\usepackage{longtable}

\title{Gap asymptotics of the directions in an Ammann--Beenker-like quasicrystal}
\author{Gustav Hammarhjelm}

\newcommand{\bb}[1]{\mathbb{#1}}
\newcommand{\bs}{\backslash}

\newcommand{\inv}[1]{#1^{-1}}

\newcommand{\mc}[1]{\mathcal{#1}}
\newcommand{\mf}[1]{\mathfrak{#1}}
\newcommand{\norm}[1]{\left\| #1\right\|}

\newcommand{\SL}[2]{\mathrm{SL}_#1(#2)}
\newcommand{\SLK}{\mathrm{SL}_2(\mc{O}_K)}
\newcommand{\SLR}[1]{\mathrm{SL}_{#1}(\R)}
\newcommand{\SLZ}[1]{\mathrm{SL}_{#1}(\Z)}

\newcommand{\smpt}[1]{\setminus\{#1\}}

\newcommand{\vol}{\mathrm{vol}}

\renewcommand{\a}{\alpha}
\renewcommand{\b}{\beta}
\renewcommand{\d}{\delta}
\newcommand{\e}{\epsilon}
\newcommand{\g}{\gamma}

\renewcommand{\l}{\lambda}

\newcommand{\s}{\sigma}

\newcommand{\z}{\zeta}

\newcommand{\Q}{\bb{Q}}
\newcommand{\R}{\bb{R}}

\newcommand{\Z}{\bb{Z}}

\newtheorem{thm}{Theorem}[section]
\newtheorem*{thm*}{Theorem}

\newtheorem{lem}[thm]{Lemma}
\newtheorem{cor}[thm]{Corollary}
\newtheorem{prop}[thm]{Proposition}

\theoremstyle{definition}
\newtheorem{defn}[thm]{Definition}

\theoremstyle{remark}

\theoremstyle{remark}

\theoremstyle{remark}

\begin{document}
	
	\maketitle
	
	\begin{abstract}
		\noindent It is known that the limiting gap distribution of the directions to visible points in planar quasicrystals of cut-and-project type exists as a continuous function $F(s)$. In this article we study the asymptotic behaviour of said limiting gap distribution in the particular case of an Ammann--Beenker-like quasicrystal $\mc{P}$; more precisely we show that in this case $F(s)=C_{\mc{P}}s^{-2}+\mc{O}(s^{-17/8})$ as $s\to \infty$ with an explicit constant $C_{\mc{P}}>0$.
	\end{abstract}
	
	\section{Introduction}
	
	To each locally finite point set $\mc{P}\subset \R^d$ we associate its subset  \[\widehat{\mathcal{P}}:=\{x\in \mathcal{P}\mid tx\notin \mathcal{P}, \forall t\in (0,1)\}\label{Def-vis-points}\] of points \textit{visible} from the origin. Given $T>0$ we let $\widehat{\mc{P}}_T:=\widehat{\mc{P}}\cap B_T(0)$, where \[B_T(0):=\{x\in \R^d: \norm{x}<T\}.\] In dimension $d=2$, we can reduce each $\widehat{\mc{P}}_T$ to a set of real numbers containing the \textit{directions} in which points of $\mc{P}$ within $B_T(0)$ can be seen. More precisely: For $T\in \bb{R}_{>0}$, let $\widehat{N}(T)$ denote the number of visible points of $\mc{P}$ within $B_T(0)$. Arrange the normalised arguments of the points of $\widehat{\mc{P}}_T$, which are numbers in $(-\frac{1}{2},\frac{1}{2}]$, in an increasing list as
	\begin{equation*}
	\label{eqnXiHat}
	-\tfrac{1}{2}<\widehat{\xi}_{T,1}<\widehat{\xi}_{T,2}< \cdots<\widehat{\xi}_{T,\widehat{N}(T)}\leq \tfrac{1}{2}.\end{equation*}	
	In addition, let $\widehat{\xi}_{T,0}:=\widehat{\xi}_{T,\widehat{N}(T)}-1$. For each $1\leq i\leq \widehat{N}_T$, we set $\widehat{d}_{T,i}:=\widehat{N}(T)(\widehat{\xi}_{T,i}-\widehat{\xi}_{T,i-1})$, which is the $i$:th  \textit{normalised gap} (between the angles of visible points) in $\mc{P}$. 
	
	We will be interested in the complementary distribution function $F_T$ of the probability measure 
	\[\mu_T:=\frac{1}{\widehat{N}_T}\sum_{i=1}^{\widehat{N}_T}\delta_{\widehat{d}_{T,i}}\]
	comprised of the Dirac measures at the normalised gaps $\widehat{d}_{T,i}$, that is,
	\begin{equation*}
	\label{eqnLimMinGap}
	F_T(s):=\frac{\#\{1\leq i\leq \widehat{N}_T\mid \widehat{d}_{T,i}\geq s\}}{\widehat{N}_T}=\mu_T([s,\infty)).
	\end{equation*} 
	We say that \textit{the limiting distribution of normalised gaps in} $\mc{P}$ exists if $\mu_T$ converges weakly to a probability measure $\mu$, that is, if $F_T(s)$ converges to $F(s):=\mu([s,\infty))$ at every continuity point of $F$. In this case, we call $F$ the limiting distribution of normalised gaps in $\mc{P}$, or simply the limiting gap distribution (of $\mc{P}$).
	
	The limiting distribution of normalised gaps is known to exist as a continuous function in several interesting cases. In \cite{boca2000distribution}, Boca, Cobeli, and Zaharescu determined the limiting distribution of normalised gaps explicitly in the case of $\mc{P}=\Z^2$. In \cite{marklof2010distribution}, Marklof and Strömbergsson investigated the fine-scale distribution of the directions of points in affine lattices and described those distributions in terms of probability measures on associated homogeneous spaces. In \cite{baake2014radial}, Baake, Götze, Huck and Jakobi numerically computed the distribution of $\widehat{d}_{T,i}$ for large $T$ for several prominent examples of quasicrystals, giving indications of what to expect regarding a minimal gap in the distributions as well as asymptotic tail behaviour of the distributions. In \cite[Corollary 3]{marklof2014visibility}, Marklof and Strömbergsson proved that the limiting distribution of normalised gaps exists as a continuous function for every planar, regular cut-and-project set. Again, these distributions are described explicitly in terms of a probability measure on an associated homogeneous space of cut-and-project sets. In this article it was also shown that there are quasicrystals with a \textit{minimal gap}, i.e.\ quasicrystals $\mc{P}$ such that the limiting distribution function is equal to $0$ for all $s<m_{\mc{P}}$ for some $m_{\mc{P}}>0$. 
	
	In \cite{hammarhjelm2019density} the limiting gap distributions of several families of quasicrystals, including the Ammann--Beenker point set and Penrose quasicrystals, were studied. In particular, the minimal gap of the corresponding limiting gap distribution functions were determined.
	In the present article, we continue the study of the limiting gap distribution of Ammann--Beenker-like point sets initiated in \cite{hammarhjelm2019density} by determining the asymptotic behaviour of the limiting gap distribution $F(s)$ as $s\to\infty$ of a particular cut-and-project set. The main result of the present article is the following; it is proved in \Cref{thm-F(s)-asymp-decay} below.
	
	\begin{thm*}
		Let $\mc{L}$ be the Minkowski embedding of $\Z[\sqrt{2}]$ in $\R^4$. Let $\mc{W}\subset \R^2$ be the open, regular octagon of sidelength $\sqrt{2}$ centered at the origin, and with sides perpendicularly bisected by the coordinate axes. Let $\mc{P}\subset \R^2$ be the cut-and-project set obtained from $\mc{L}$ and $\mc{W}$ and let $F(s)$ be its associated limiting gap distribution function. Then
		\[F(s)=C_\mc{P}s^{-2}+\mc{O}(s^{-17/8})\]
		as $s\to\infty$, where 
		\[C_{\mc{P}}:=\frac{(\sqrt{2}-1)}{\sqrt{2}\z_K(2)^2}\]
		and $\z_K$ is the Dedekind zeta function of $K:=\Q(\sqrt{2})$.
	\end{thm*}

	The set $\mc{P}$ is \textit{Ammann--Beenker-like}; it is closely related to the Ammann--Beenker point set. The relation between these two sets will be made precise in \Cref{sec-ABlike} below. The significant difference is that the Ammann--Beenker point set is a cut-and-project set with a window that differs from $\mc{W}$ by an application of an invertible linear map. By working with $\mc{W}$ we get a pleasant window exhibiting a high degree of symmetry, which will simplify some of the calculations in the paper, allowing us to make the asymptotic decay of the limiting gap distribution completely explicit. The main result of this paper should be seen as a proof of concept. The novelty is that we are able to explicitly determine the asymptotic behaviour of the limiting gap distribution of a planar quasicrystal of cut-and-project type. 

	The outline of the paper is as follows. In \Cref{sec-background}, we recall some necessary background, such as the definitions of cut-and-project sets as well as their associated homogeneous spaces. In section \Cref{sec-G(s)} we determine the asymptotics of a function $G(s)$ that satisfies $-G'(s)=F(s)$. Determining the asymptotics of this function will then help us to determine the asymptotics of $F(s)$ in \Cref{sec-F(s)}.
	
	\section{Background and set-up} 
	\label{sec-background}
	
	We begin by recalling the object of study, the Ammann--Beenker like point set, and the result of \cite{marklof2014visibility} describing the limiting gap distribution of a regular cut-and-project set.
	
	\subsection{Cut-and-project sets and associated homogeneous spaces}
	
	\label{secCPS}
	We briefly recall cut-and-project sets, following the notation of \cite[Sec. 1.2]{marklof2014free}. 
	
	If $\bb{R}^n=\bb{R}^d\times \bb{R}^m$, let 
	\begin{alignat*}{2}
	\pi:& ~\bb{R}^n\longrightarrow \bb{R}^d  & \pi_{\mathrm{int}}: & ~ \bb{R}^n\longrightarrow \bb{R}^m \\
	& (x_1,\ldots,x_n)\longmapsto (x_1,\ldots,x_d)\hspace{1cm}& & (x_1,\ldots,x_n)\longmapsto (x_{d+1},\ldots, x_n)
	\end{alignat*}
	denote the projections onto the first $d$ and last $m$ coordinates, respectively.
	
	\begin{defn}
		\label{defnCPS}
		Given a lattice $\mathcal{L}\subset \bb{R}^n=\bb{R}^d\times \bb{R}^m$ and a bounded set $\mathcal{W}\subset \overline{\pi_{\mathrm{int}}(\mathcal{L})}$ with $\mathcal{W}^\circ\neq\emptyset$, the set
		\[\mathcal{P}(\mathcal{W},\mathcal{L}):=\{\pi(y)\mid y\in \mathcal{L},\pi_{\mathrm{int}}(y)\in\mathcal{W}\}\label{DefP(W,L)}\]
		is the \textit{cut-and-project} obtained from $\mathcal{L}$ and $\mathcal{W}$. If the boundary of $\mc{W}$ has measure $0$ with respect to Haar measure on $\overline{\pi_{\mathrm{int}}(\mathcal{L})}$ we say that $\mc{P}(\mc{W},\mc{L})$ is \textit{regular}.
	\end{defn}
	
	With these assumptions on $\mc{L}$ and $\mc{W}$, the resulting cut-and-project set is always Delone, see \cite[Proposition 3.1]{marklof2014free}. Next, we recall the homogeneous space associated to a cut-and-project set, from \cite{marklof2014free}.
	
	Let $n,d>0$ be given as above and let $G:=\SL{n}{\R}$\label{Def-G}, $\Gamma:=\SL{n}{\Z}$\label{Def-Gamma}. Fix a lattice $\mc{L}\subset \R^n$ and $\delta>0$, $g\in G$ so that $\mc{L}=\delta^{1/n}\Z^ng$. Let $\varphi_g:\SL{d}{\R}\longrightarrow G$ be the embedding given by
	\[A\mapsto g\,\mathrm{diag}(A,I_m)\inv{g},\]
	where $\mathrm{diag}(A,I_m)$ is the block-diagonal matrix with blocks $A$ and $I_m$. By Ratner's theorems, there exists a unique, closed, connected subgroup $H_g\subset G$ such that $H_g\cap \Gamma$ is a lattice in $H_g$, $\varphi_g(\SL{d}{\R})\subset H_g$ and such that the closure of $\Gamma \backslash \Gamma \varphi_g(\SL{d}{\R})$ in $\Gamma\backslash G$ is equal to $\Gamma \backslash \Gamma H_g$. Let $X$ denote the homogeneous space $(H_g\cap \Gamma)\backslash H_g$ and $\mu$\label{Def-mu} its unique, right $H_g$-invariant probability measure. This space can be identified with $\Gamma\backslash \Gamma H_g$; we also let $\mu$ denote the right $H_g$-invariant probability measure on this space. For a fixed $\mc{W}\subset \overline{\pi_{\mathrm{int}}}(\mc{L})$ we get, for each $x=\Gamma h\in X$, a cut-and-project set by
	\[\mc{P}^x:=\mc{P}(\mc{W},\delta^{1/n}\Z^nhg).\]
	The space of cut-and-project sets obtained as $x$ varies in $X$ is the homogeneous space of cut-and-project sets associated to $\mc{P}(\mc{W},\mc{L})$. We note that for random $x$ in $(X,\mu)$, the point process $\mc{P}^x$ is $\SL{d}{\R}$-invariant, in view of $\varphi_g(\SL{d}{\R})\subset H_g$.
	
	\subsection{The limiting gap distribution of a planar, regular cut-and-project set}
	
	We will now recall the characterisation of the limiting gap distribution of a regular, planar cut-and-project set from \cite{marklof2014visibility}. A point set $\mc{P}\subset \R^d$ is said to have \textit{asymptotic density} if there is a number $\theta(\mc{P})$, the \textit{density} of $\mc{P}$\label{Def-density}, such that \[\lim_{T\to\infty} T^{-d}\#(\mc{P}\cap TD)=\theta(\mc{P})\vol(D)\] holds for all Jordan measurable $D\subset \R^d$. It is well-known that regular cut-and-project sets have an asymptotic density and in \cite[Theorem 1]{marklof2014visibility} it was shown that the density of the subset of visible points of such a set also exists. 
	
	We now fix $d=2$ and a regular cut-and-project set $\mc{P}=\mc{P}(\mc{W},\mc{L})$. Fix $g\in G$ and $\delta>0$ so that $\mc{L}=\delta^{1/n}\Z^ng$. Let $X$ be the homogeneous space associated to $\mc{P}(\mc{W},\mc{L})$. Given $s>0$, let \[\mathfrak{C}(s):=\{(x_1,x_2)\in\R^2: 0<x_1<1, |x_2|<s\theta(\widehat{\mc{P}})^{-1}\}\label{Def-mfC}\] (this is well defined, since $0<\theta(\widehat{\mc{P}})\leq \theta(\mc{P})$ for regular cut-and-project sets). Combining Corollary 3, Theorem 4 and Section 11 of \cite{marklof2014visibility} we find that the limiting gap distribution of $\mc{P}$ exists as a continuous function given by
	\begin{equation}
	\label{eqn-F(s)}
	F(s):=-\frac{d}{ds}\mu(\{x\in X\mid \widehat{\mc{P}^x}\cap \mf{C}(s)=\emptyset\}).
	\end{equation}
	In the present article we study the asymptotics of $F(s)$ as $s\to\infty$ in the particular case of an Ammann--Beenker like point set, which we will now proceed to define.
	
	\subsection{The Ammann--Beenker-like point set $\mc{P}$}
	\label{sec-ABlike}
	
	Let $K=\bb{Q}(\sqrt{2})$\label{Def-K} and let $\mathcal{O}_K=\Z[\sqrt{2}]$ be its subset of algebraic integers\label{Def-O_K}. Let $\sigma$\label{Def-sigma} be the non-trivial automorphism of $K$ and let $N(\a)=\a\sigma(\a)$ denote the norm of $\a\in \mc{O}_K$. Given a vector or matrix $A$ with entries in $K$, let $\sigma(A)$ be the object obtained by applying $\sigma$ to each entry of $A$. Let $\mc{L}$ be the Minkowski embedding of $\mc{O}_K^2$ in $\R^4$, i.e.\ the lattice
	$\mc{L}=\{(x,\sigma(x))\mid x\in \mc{O}_K^2\}$. With $\delta=8$ and 
	\begin{equation}g:=8^{-1/4}\begin{pmatrix} 1 & 0 & 1 & 0\\ \sqrt{2} & 0 & -\sqrt{2} & 0\\ 0 & 1 & 0 & 1\\ 0 & \sqrt{2} & 0 & -\sqrt{2}
	\end{pmatrix}\label{Def-g}\end{equation}
	we have $\mc{L}=\delta^{1/4}\Z^4g$. Let also $\mc{L}'$ denote the Minkowski embedding of $\mc{O}_K$ in $\R^2$\label{Def-L'-L}. Let $\mc{W}\subset \R^2$\label{Def-W} be the open, regular octagon of side length $\sqrt{2}$ centered at the origin, oriented so that its sides are perpendicularly bisected by the coordinate axes. The main object of study in the present article is \[\mc{P}:=\mc{P}(\mc{W},\mc{L})\label{def-P}\]
	which is what we call an \textit{Ammann--Beenker-like} point set. The Ammann--Beenker point set $\mc{A}$ (obtained from the famous substitution tiling, see e.g. \cite{baake2013aperiodic}) is of the form $\mc{P}(\mc{W}A,\mc{L})B$ where 
	\[A:=\frac{1}{\sqrt{2}}\begin{pmatrix}\frac{1}{\sqrt{2}} & 0 \\ \frac{1}{\sqrt{2}} & 1\end{pmatrix}\text{ and }B:=\begin{pmatrix}1 & 0 \\ \frac{1}{\sqrt{2}} & \frac{1}{\sqrt{2}}\end{pmatrix}\]
	(see e.g.\ \cite[(5)]{hammarhjelm2019density}). The main difference between $\mc{P}$ and $\mc{A}$ is that $\mc{P}$ is a cut-and-project whose window shows a higher degree of symmetry than that of $\mc{A}$.
	
	In \cite{marklof2014free} it was shown that the Ratner group $H_g$\label{Def-H_g} associated to $g$ is given by $gH \inv{g}$, where $H:=\SL{2}{\R}^2\subset G$\label{Def-H} denotes the set of block-diagonal matrices with blocks from $\SL{2}{\R}$. One verifies that $\Gamma_g:=\Gamma\cap H_g$\label{Def-Gamma_g} is equal to $g\Gamma_K\inv{g}$, where $\Gamma_K\subset G$\label{Def-Gamma_K} denotes the image of the Minkowski embedding $\SL{2}{\mc{O}_K}\longrightarrow G$,  $A\mapsto\mathrm{diag}(A,\sigma(A))$. Recall that $X=\Gamma\backslash\Gamma H_g$\label{Def-X}. Given $x=\Gamma gh\inv{g}\in X$ ($h\in H$), we will use the notation $\mc{L}(x)$\label{Def-L(x)} for the lattice $\delta^{1/4}\Z^4(gh\inv{g})g=\mc{L}h$. Thus, $\mc{P}^x=\mc{P}(\mc{W},\mc{L}(x))$\label{Def-P^x}.
	
	\section{Asymptotics of $\mu(\{x\in X\mid \widehat{\mc{P}^x}\cap \mf{C}(s)=\emptyset\})$}
	\label{sec-G(s)}
	
	As a first step, we determine the asymptotics of 
	\[G(s):=\mu(\{x\in X\mid \widehat{\mc{P}^x}\cap \mf{C}(s)=\emptyset\})\label{Def-G(s)}\] as $s\to\infty$. In view of \eqref{eqn-F(s)}, our final goal is to understand the derivative of $-G(s)$. Calculating the asymptotics of $G(s)$ is, as we will see, simpler than calculating the asymptotics of $F(s)$. In fact, we will be able to give the asymptotics of $G(s)$ in a very explicit form in \Cref{prop-G(s)-asympt} below. This in turn allows us to give the asymptotics of $F(s)$ in an equally explicit form in \Cref{thm-F(s)-asymp-decay}.
	
	\begin{lem}
		\label{lem-rewrite-G(s)}
		Let $\lambda:=1+\sqrt{2}$\label{Def-lambda} be the fundamental unit of $\mc{O}_K$. Given $s>0$, let $r\in \Z$ be the integer so that $\l^{2r}\leq s^{1/4}<\l^{2(r+1)}$. Let also $T(s)\subset \R^2$\label{Def-T(s)} be the open triangle with vertices $(0,0)$ and $(s/\theta(\widehat{\mc{P}}))^{1/2}(1,\pm 1)$ and set \[\mc{T}(s):=(T(s)\times \mc{W})\mathrm{diag}(\l^{-2r},\l^{-2r},\l^{2r},\l^{2r})\label{Def-mcT(s)}.\]
		We then have 
		\[G(s)=\int_{X}I(\mc{L}(x)\cap \mc{T}(s)=\emptyset)\,d\mu(x).\]	
		(Here and throughout the article, $I(\cdot)$ denotes the indicator function\label{Def-indicator}.)
	\end{lem}
	
	\begin{proof}
		As $\mf{C}(s)$ is star-shaped with respect to the origin, we have
		\[\{x\in X\mid \widehat{\mc{P}^x}\cap \mf{C}(s)=\emptyset\}=\{x\in X\mid \mc{P}^x\cap \mf{C}(s)=\emptyset\}.\]
		Due to the $\SL{2}{\R}$-invariance of $x\mapsto \mc{P}^x$ we have
		\[\mu(\{x\in X\mid \mc{P}^x\cap \mf{C}(s)=\emptyset\})=\mu(\{x\in X\mid \mc{P}^x\cap \mf{C}(s)A=\emptyset\})\]
		for each $A\in \SL{2}{\R}$. Thus, we may freely modify $\mf{C}(s)$ by multiplying with $\SL{2}{\R}$-matrices, and since the areas of $\mf{C}(s)$ and $T(s)$ coincide we have 
		\[G(s)=\mu(\{x\in X\mid \mc{P}^x\cap T(s)=\emptyset\}).\]
		If $x=\Gamma gh\inv{g}$ we have, using the definition of $\mc{P}^x$, that $\mc{P}^x\cap T(s)=\emptyset$ is equivalent with $\mc{L}(x)\cap (T(s)\times \mc{W})=\emptyset$. Since $\mc{L}(x)=\mc{L}h$ and $\mathrm{diag}(\l^{-2r},\l^{-2r},\l^{2r},\l^{2r})$ commutes with $h\in H$ and fixes $\mc{L}$ the proof of the lemma is complete.
	\end{proof}
	
	Note that $\mc{T}(s)$ contains a ball of radius $\gg s^{1/4}$. 
	Thus, in view of \Cref{lem-rewrite-G(s)}, the question about the asymptotics of $G(s)=\mu(\{x\in X\mid \widehat{\mc{P}^x}\cap \mf{C}(s)=\emptyset\})$ can be answered by understanding the probability that a random lattice $\mc{L}h$ avoids a large (convex) set. Similar questions have been studied in \cite{marklof2011periodic} and \cite{strombergsson2011probability}. As is done in \cite{marklof2011periodic}, we start by describing a good Siegel domain for our space of lattices (more precise definition at the beginning of \Cref{subsec-siegel}).
	
	\subsection{A Siegel domain for $\Gamma_K$ in $H$}
	\label{subsec-siegel}
	
	Let $\mu'$ denote a Haar measure on $\SL{2}{\R}$. Then, a Haar measure $\mu_H$ on $H$ is given by $\mu'\times \mu'$ and this is (up to conjugation with $g$) Haar measure on $H_g$. We wish to determine an explicit \textit{Siegel domain} $\mf{D}$ of $\Gamma_K$ in $H$, i.e.\ an explicit set satisfying $\Gamma_K\mf{D}=H$ and other convenient properties. The construction presented here is a part of well-known general theory (cf. e.g. \cite{borel1969introduction}); however we will give detailed, explicit
	proofs of all the properties we need in our particular situation.
	
	Now recall the Iwasawa decomposition of $\SL{2}{\R}$. Given $x\in \R$, $y>0$ and $\theta\in \R/2\pi \Z$ let
	\[
	\label{Def-nak}n(x):=\begin{pmatrix}
	1 & x\\ 0 & 1
	\end{pmatrix},\quad a(y):=\begin{pmatrix}
	y & 0 \\ 0 & \inv{y}
	\end{pmatrix},\quad k(\theta):=\begin{pmatrix}
	\cos \theta & -\sin \theta \\ \sin \theta & \cos\theta
	\end{pmatrix}.\]
	Then every $A\in \SL{2}{\R}$ can be written uniquely as $A=n(x)a(y)k(\theta)$. With respect to this parametrisation, a Haar measure $\mu'$ is given by $y^{-3}dx\,dy\,d\theta$. We will use the notation $h(x_1,x_2,y_1,y_2,\theta_1,\theta_2)$ to express an arbitrary element of $H$, where \begin{equation}h(x_1,x_2,y_1,y_2,\theta_1,\theta_2):=\begin{pmatrix}n(x_1) & 0\\ 0 & n(x_2)\end{pmatrix} \begin{pmatrix}a(y_1) & 0\\ 0 & a(y_2)\end{pmatrix} \begin{pmatrix}k(\theta_1) & 0\\ 0 & k(\theta_2)\end{pmatrix}\label{Def-h}\end{equation}
	for $x_1,x_2\in \R$, $y_1,y_2\in \R_{>0}$ and $\theta_1,\theta_2\in \R/2\pi \Z$.
	
	Let $\mf{F}\subset \R_{>0}^2$\label{Def-mfF} be a fixed, bounded fundamental domain of $\mc{L}'$. Define, for $t>0$
	\[\label{Def-mfDt}\mf{D}_{t}:=\{h\in H\mid (x_1,x_2)\in\mf{F}, y_1y_2\geq t, y_1/y_2\in [\l^{-2},\l^2],0\leq \theta_1,\theta_2<\pi\}.\]
	We begin by showing that there is $t>0$ such that $\Gamma_K \mf{D}_{t}=H$.
	
	\begin{lem}
		\label{lem-Siegel-1}
		There exists $t>0$ so that $\Gamma_K \mf{D}_{t}=H$.
	\end{lem}
	
	\begin{proof}
		Take $h\in H$. We wish to show that there exists $t>0$ (which is independent of $h$) so that $h$ can be brought into $\mf{D}_{t}$ through multiplication by elements of $\Gamma_K$ from the left. It is clear that there always exists $\gamma\in \Gamma_K$ so that the $x_i$-parameters of $\gamma h$ satisfy $(x_1,x_2)\in \mf{F}$; therefore we focus on the $y_i$- and $\theta_i$-parameters.
		
		Recall that, given $A=(a_{ij})\in \SL{2}{\R}$, the $y$-parameter of its Iwasawa decomposition is $\norm{(a_{21},a_{22})}^{-1}$. Given $r\in \Z$, let $M_r=\mathrm{diag}(\l^{r},\l^{-r},\s(\l^{r}),\s(\l^{-r}))\in \Gamma_K$. It is then seen that if $h$ has $y$-parameters $y_1,y_2$, the $y$-parameters of $M_{2r}h$ are $\l^{2r}y_1$ and $\l^{-2r}y_2$, respectively. The product of these parameters remains invariant, while the quotient changes from $y_1/y_2$ to $\l^{4r}y_1/y_2$. Thus, it suffices to show that there exists $t$ so that we always can modify the $y$-parameters of $h$ through multiplication with $\gamma\in \Gamma_K$ to satisfy $y_1y_2\geq t$.
		
		To this end, let $h=\mathrm{diag}(h_1,h_2)$, $h_1=(a_{ij})$, $h_2=(b_{ij})$ and $A=(\alpha_{ij})\in \SL{2}{\mc{O}_K}$. Then, the $y$-parameters $y_1',y_2'$ of $\mathrm{diag}(A,\sigma(A))h$ are given by
		\begin{align*}
		y_1'&=\norm{(\a_{21}a_{11}+\a_{22}a_{21},\a_{21}a_{12}+\a_{22}a_{22})}^{-1},\\
		y_2'&=\norm{(\sigma(\a_{21})b_{11}+\sigma(\a_{22})b_{21},\sigma(\a_{21})b_{12}+\sigma(\a_{22})b_{22}))}^{-1}.
		\end{align*}
		For $(\alpha_{21},\alpha_{22})\in \mc{O}_K^2$, the vector
		\[v(\alpha_{21},\alpha_{22}):=(\a_{21}a_{11}+\a_{22}a_{21},\a_{21}a_{12}+\a_{22}a_{22},\sigma(\a_{21})b_{11}+\sigma(\a_{22})b_{21},\sigma(\a_{21})b_{12}+\sigma(\a_{22})b_{22})\]
		is verified to belong to a lattice of covolume $8$. By Minkowski's theorem, there is an absolute constant $C$ so that every such lattice contains a non-zero vector of length bounded from above by $C$. Thus, $\a_{21},\a_{22}$ can be chosen so that $0<\norm{v(\alpha_{21},\alpha_{22})}\leq C$. Furthermore, $\a_{21},\a_{22}$ can be chosen to be relatively prime at the expense of enlarging $C$ slightly. Indeed, suppose $d$ is a greatest common divisor of $\a_{21}$ and $\a_{22}$. We may assume that $1\leq d<\l$ and hence $|\s(d)|\gg1$. Letting $\a_{21}'=\a_{21}/d$ and $\a_{22}'=\a_{22}/d$ there is an absolute constant $C'$ so that $0<\norm{v(\alpha_{21}',\alpha_{22}')}\leq C'$. Hence, $y_1',y_2'\geq (C')^{-1}$ and so $y_1'y_2'\geq (C')^{-2}=:t$.
		
		It now suffices to show that we can modify $h$ with $y_1y_1\geq t$, $y_1/y_2\in [\l^{-2},\l^2]$ so that $0\leq \theta_1,\theta_2<\pi$. To this end, assume first that $y_1/y_2\in [\l^{-2},1]$. Then, if \[h'=h(x_1',x_2',y_1',y_2',\theta_1',\theta_2'):=M_1h\] we have 
		$y_1y_2=y_1'y_2'$, $y_1'/y_2'=\l^2(y_1/y_2)\in [1,\l^2]$, $\theta_1'=\theta_1$ and $\theta_2'=\theta_2+\pi$. If $y_1/y_2\in [1,\l^2]$, multiply $h$ with $M_{-1}$ instead. To modify $\theta_1$, modify $h$ with $\s(M_{\pm 1})$ in place of $M_{\pm1}$.	
	\end{proof}
	
	\subsection{Reduction of $\mf{D}_{t}$ for lattices avoiding large balls}
	\label{sec-reduction}
	
	Given $t>0$ so that $\Gamma_K\mf{D}_{t}=H$,
	we now intend to improve the Siegel domain $\mf{D}_{t}$ by removing some redundance, that is, we study when $h,h'\in H$ give rise to the same lattice, i.e.\ when $\mc{L}h=\mc{L}h'$. We will use the following notation throughout the article: if $f,g$ are non-negative functions, we write $f\ll g$ to indicate that there is an absolute constant $C>0$ so that $f\leq Cg$; if $f\ll g$ and $g\ll f$ we write $f\asymp g$ \label{Def-ll}.
	
	The following lemma is inspired by \cite[Lemma 3.1]{marklof2011periodic}.
	
	\begin{lem}
		\label{lem-sizey1y2}
		Take \[h=h(x_1,x_2,y_1,y_2,\theta_1,\theta_2)\in \mf{D}_t\]
		and suppose $\mc{L}h$ avoids a ball of large radius $R$. Then $y_1,y_2\gg R$.
	\end{lem}
	
	\begin{proof}
		We can write $8^{1/4}gh$ explicitly as
		\begin{equation}
		\label{eqn-gh}
		\begin{pmatrix}y_1 & \inv{y_1}x_1 & y_2 & \inv{y_2}x_2\\\sqrt{2}y_1 & \sqrt{2}\inv{y_1}x_1 & -\sqrt{2}y_2 & -\sqrt{2}\inv{y_2}x_2\\0 & \inv{y_1} & 0 & \inv{y_2}\\0 & \sqrt{2}\inv{y_1} & 0 & -\sqrt{2}\inv{y_2}\end{pmatrix}\mathrm{diag}(k(\theta_1),k(\theta_2)).
		\end{equation}
		Let $b_1,\ldots,b_4$\label{Def-b_i-1} denote the row vectors of $8^{1/4}gh$ so that $b_1,\ldots,b_4$ is a lattice basis of $\mc{L}h$. We have
		\[\norm{b_1}\leq y_1+\inv{y_1}|x_1|+y_2+\inv{y_2}|x_2|.\]
		Since $h\in \mf{D}_t$, we have that $y_1,y_2\gg 1$.
		We have $\inv{y_i}|x_i|\ll 1\ll y_i$, for $i\in\{1,2\}$. Also, $y_2\leq \l^2 y_1$. Hence $\norm{b_1}\ll y_1$, and similarly, $\norm{b_2}$, $\norm{b_3}$, $\norm{b_4}\ll y_1$. 
		
		Let $\sum_{i=1}^4r_ib_i$ be the centre of a ball of radius $R$ that has empty intersection with $\mc{L}h$. Find integers $n_i$ with $|n_i-r_i|\leq 1/2$. Note that $\sum_{i}n_ib_i\in \mc{L}h$ and 
		\[R<\norm{\sum_{i=1}^4(n_i-r_i)b_i}\leq \tfrac{1}{2}\sum_{i=1}^4\norm{b_i}\ll y_1\ll y_2,\]
		which is the claim of the lemma.
	\end{proof}
	
	\begin{lem}
		\label{lem-siegel-reduction}
		Fix $t>0$ so that $\Gamma_K\mf{D}_{t}=H$.
		Take $h,h'\in \mf{D}_{t}$ and suppose that $\mc{L}h=\mc{L}h'$. If $\mc{L}h$ and $\mc{L}h'$ avoid a sufficiently large ball, then the $\theta$-parameters of $h$ and $h'$ agree. Furthermore, if $y_1,y_2$ and $y_1',y_2'$ denote the Iwasawa parameters of $h$ and $h'$, respectively, we have $y_1/y_2=\l^{4r} (y_1'/y_2')$ for some $r\in\{-1,0,1\}$. Finally, the set 
		\[\label{Def-mfD}\mf{D}_t':=\{h\in \mf{D}_t\mid y_1/y_2\in (-\l^2,\l^2)\}\]
		gives an irredundant representation of almost all lattices $\mc{L}h$ that avoid a sufficiently large ball; that is, if $h,h'\in \mf{D}_t'$ and $\mc{L}h=\mc{L}h'$, then $h=h'$.
	\end{lem}
	
	\begin{proof}
		In view of \Cref{lem-sizey1y2}, the $y$-parameters of $h$ and $h'$ are large. As usual, $h$ and $h'$ have Iwasawa parameters $x_i,y_i,\theta_i$ and $x_i',y_i',\theta_i'$, respectively. As in the proof of \Cref{lem-sizey1y2}, let $b_i$ and $b_i'$ denote the row vectors of $8^{1/4}gh$ and $8^{1/4}gh'$ respectively. Since the $y$-parameters of $h$ and $h'$ are large, the lengths of $b_3,b_4,b_3'$ and $b_4'$ are large. We first claim that
		\[\Z b_3+\Z b_4=\Z b_3'+\Z b_4'.\]
		Indeed, take $v=\sum_{i=1}^4 n_ib_i\in \mc{L}h=\mc{L}h'$ and consider 
		$v\cdot e_1k=y_1(n_1+n_2\sqrt{2})$ and $v\cdot e_3k=y_2(n_1-n_2\sqrt{2})$,
		where $k=\mathrm{diag}(k(\theta_1),k(\theta_2))$. Since for a unit vector $u$ we have $|u\cdot v|\leq \norm{v}$ by Cauchy--Schwarz inequality, we have that the length of $v$ can only be so small that $v=b_3'$ if $n_1=n_2=0$ (recall that $\norm{b_3'}$ is small when $y_1',y_2'$ are large). Therefore, $b_3'\in \Z b_3+\Z b_4$. Similarly, $b_4'\in \Z b_3+\Z b_4$ and conversely, $b_3,b_4\in \Z b_3'+\Z b_4'$.
		
		Now, $\Z b_3+\Z b_4=\Z b_3'+\Z b_4'$ implies that $|\theta_i-\theta_i'|\in\{0,\pi\}$ for $i\in {1,2}$. To verify this, let $\pi_1,\pi_2:\R^4\longrightarrow\R^2$ be the projections onto the first and last two coordinates, respectively. We have 
		\[\pi_1(\R b_3+ \R b_4)=\R(\sin\theta_1,\cos\theta_1)=\R(\sin\theta_1',\cos\theta_1')=\pi_1(\R b_3'+ \R b_4')\]
		which gives $|\theta_1-\theta_1'|\in \{0,\pi\}$. The other claim is proved similarly using $\pi_2$. Hence, we have proved the first statement of the lemma.
		
		Using the first statement of the lemma and  $\Z b_3+\Z b_4=\Z b_3'+\Z b_4'$ we find that 
		\[\Z(\inv{y_1},\inv{y_2})+\Z\sqrt{2}(\inv{y_1},-\inv{y_2})=\Z((y_1')^{-1},(y_2')^{-1})+\Z\sqrt{2}((y_1')^{-1},-(y_2')^{-1}).\] 
		We also have $y_1y_2=y_1'y_2'$ by comparing the covolumes of $\Z b_3+\Z b_4=\Z b_3'+\Z b_4'$. This implies 
		\begin{equation}
		\label{eqn-lattice-y}
		\Z(y_2,y_1)+\Z\sqrt{2}(y_2,-y_1)=\Z(y_2',y_1')+\Z\sqrt{2}(y_2',-y_1').
		\end{equation}
		Note that $(y_2,y_1)$ is an element of the left hand side of \eqref{eqn-lattice-y}, and that a general element of the right hand side is of the form $(\a y_2',\s(\a)y_1')$ with $\a\in \mc{O}_K$. Thus, there is $\a\in \mc{O}_K$ so that $y_2=\a y_2'$ and $y_1=\s(\a)y_1'$ which implies that $y_2/y_2'=\a$ and $y_1/y_1'=\s(\a)$. From $y_1y_2=y_1'y_2'$ we see that $y_2/y_2'=y_1'/y_1$ which implies that $\a^{-1}=\s(\a)$ so that $\a\in \mc{O}_K^*$. Since $y_1,y_2,y_1',y_2'$ are positive, it follows that both $\a$ and $\s(\a)$ are positive, hence $\a=\l^{2r}$ for some integer $r$.
		Now note that 
		\[y_1/y_2=(\s(x)y_1')/(xy_2')=(\s(x)/x)(y_1'/y_2')=\l^{-4r}(y_1'/y_2').\] 
		Since $y_1/y_2,y_1'/y_2'\in [\l^{-2},\l^2]$ it follows that $r\in\{-1,0,1\}$.
		
		Note that if $r=0$, then $y_1y_2=y_1'y_2'$ and $y_1/y_2=y_1'/y_2'$ implies that $y_1=y_1'$ and $y_2=y_2'$ and also $x_1=x_1'$ and $x_2=x_2'$. If $r\in \{-1,1\}$ then $y_1/y_2\in \{\l^{-2},\l^2\}$. Since $\mf{D}_t\setminus\mf{D}_t'$ is of measure $0$ with respect to Haar measure of $H$, it follow that $\mf{D}_t'$ gives an irredundant representation of almost all lattices $\mc{L}h$ that avoid a large ball.
	\end{proof}
	
	\subsection{Asymptotics of $G(s)$}
	\label{sec-G_M(s)}
	
	Fix $t>0$ so that $\Gamma_K\mf{D}_{t}=H$.
	Fix $s>0$ that is so large that so that \Cref{lem-siegel-reduction} holds, i.e.\ so that $\mf{D}_t'$ gives an irredundant representation of the lattices $\mc{L}h$, $h\in H$, that avoid $\mc{T}(s)$, a set that contains a ball of radius $\asymp s^{1/4}$. Using \Cref{lem-rewrite-G(s)} we find that 
	\begin{align*}
	G(s)=\int_{X}I(\mc{L}(x)\cap \mc{T}(s)=\emptyset)\,d\mu(x)=\int_{\mf{D}_t'}I(\mc{L}h\cap \mc{T}(s)=\emptyset)\,d\mu_H(h),
	\end{align*}
	where $\mu_H$ is the Haar measure on $H$ which induces a probability measure on $\Gamma_K\backslash H$. Thus, with respect to the Iwasawa parametrisation of $H$ used above, there is a constant $c_H$ so that $d\mu_H=c_Hy_1^{-3}y_2^{-3}dx_1dx_2dy_1dy_2d\theta_1d\theta_2$ (see \eqref{eqn-c_H} below for an explicit expression for $c_H$). Thus, we have 
	\[G(s)=c_H\int_{0}^\pi\int_{0}^\pi\int_{Y_t}\int_{\mf{F}}I(\mc{L}h\cap \mc{T}(s)=\emptyset)y_1^{-3}y_2^{-3}dx_1dx_2dy_1dy_2d\theta_1d\theta_2\]
	where $Y_t:=\{(y_1,y_2)\in \R_{>0}^2\mid y_1y_2\geq t, y_1/y_2\in (\l^{-2},\l^2)\}$ and $h=h(x_1,x_2,y_1,y_2,\theta_1,\theta_2)$.
	
	For each $h=h(x_1,x_2,y_1,y_2,\theta_1,\theta_2)\in H$, let $\widetilde{\mc{L}h}\supset \mc{L}h$ denote the set
	\[\widetilde{\mc{L}h}:=((\Z(y_1,0,y_2,0)+\Z\sqrt{2}(y_1,0,-y_2,0)+\R(0,1,0,0)+\R(0,0,0,1))k,\]
	(cf. \eqref{eqn-gh}) where $k=\mathrm{diag}(k(\theta_1),k(\theta_2))$.
	We have
	\begin{equation}\label{eqn-G(s)-Main-Error}
	\begin{split}
	G(s)&=
	\underbrace{c_H\int_{0}^\pi\int_{0}^\pi\int_{Y_t}\int_{\mf{F}}y_1^{-3}y_2^{-3}I(\widetilde{\mc{L}h}\cap \mc{T}(s)=\emptyset)\,dx_1dx_2dy_1dy_2d\theta_1d\theta_2}_{=:G_M(s)}\\
	&+\underbrace{c_H\int_{0}^\pi\int_{0}^\pi\int_{Y_t}\int_{\mf{F}}y_1^{-3}y_2^{-3}I({\mc{L}h}\cap \mc{T}(s)=\emptyset,\widetilde{\mc{L}h}\cap \mc{T}(s)\neq\emptyset)\,dx_1dx_2dy_1dy_2d\theta_1d\theta_2}_{=:G_E(s)}.
	\end{split}
	\end{equation}	
	We show in \Cref{prop-main-decay} below that there is an explicit constant $C_{\mc{P}}$ such that $G_M(s)=C_{\mc{P}}s^{-1}$ and then that $G_E(s)=\mc{O}(s^{-9/8})$ as $s\to\infty$ in \Cref{prop-G(s)-asympt} from which we conclude that $G(s)$ decays like $G_M(s)$ as $s\to \infty$.
	
	For fixed $(\theta_1,\theta_2)\in [0,\pi)^2$, let $L_1=L_1(\theta_1)\subset \R$ be the projection of $\l^{-2r}T(s)k(-\theta_1)$ onto the $x$-axis and $L_2=L_2(\theta_2)\subset \R$ be the projection of $\l^{2r}\mc{W}k(-\theta_2)$ onto the $x$-axis. Given $x\in \R$, let $R_1(\theta_1,x)$ be the intersection of the vertical line through $(x,0)$ and $\l^{-2r}T(s)k(-\theta_1)$ and let $R_2(\theta_2,x)$ be the intersection of the same line with $\l^{2r}\mc{W}k(-\theta_2)$. 
	
	\begin{lem}
		\label{lem-charac-empty-intersection}
		For $h=h(x_1,x_2,y_1,y_2,\theta_1,\theta_2)\in \mf{D}_t$ and $s>0$ we have $\widetilde{\mc{L}h}\cap \mc{T}(s)=\emptyset$ if and only if $\mc{L}'\cap (\inv{y}_1L_1(\theta_1)\times \inv{y}_2L_2(\theta_2))=\emptyset$.
	\end{lem}
	
	\begin{proof}
		Note that \[\widetilde{\mc{L}h}=\{(\alpha y_1,r_1,\sigma(\alpha)y_2,r_2)k\mid \alpha\in \mc{O}_K,r_1,r_2\in \R\}.\]
		where $k=\mathrm{diag}(k(\theta_1),k(\theta_2))$. From the definitions of $\mc{T}(s)$, $L_1(\theta_1)$ and $L_2(\theta_2)$ it follows that $\widetilde{\mc{L}h}\cap \mc{T}(s)\neq\emptyset$ if and only if there exists $\alpha\in \mc{O}_K$ such that $(\alpha y_1,\sigma(\alpha)y_2)\in L_1(\theta_1)\times L_2(\theta_2)$ which gives the claim.
	\end{proof}
	
	\begin{lem}
		\label{lem-J}
		Let $h=h(x_1,x_2,y_1,y_2,\theta_1,\theta_2)$ with $(x_1,x_2)\in \mf{F}$, $(y_1,y_2)\in Y_t$ and $(\theta_1,\theta_2)\in(\tfrac{\pi}{4},\tfrac{3\pi}{4})\times [0,\pi)$. If $s$ is sufficiently large, then 
		$\mc{L}h\cap \mc{T}(s)\neq \emptyset$.
	\end{lem}
	
	\begin{proof}
		From \eqref{eqn-gh} it follows that
		\begin{equation}
		\label{eqn-Lh-element}
		(\alpha y_1,\inv{y_1}(\alpha x_1+\beta),\sigma(\alpha)y_2,\inv{y_2}(\sigma(\alpha) x_2+\sigma(\beta)))k\in \mc{L}h
		\end{equation}
		for all $\a,\b\in \mc{O}_K$ (here $k=\mathrm{diag}(k(\theta_1),k(\theta_2))$ as usual). In particular, with $\a=0$ we find that $(0,\inv{y_1}\beta,0,\inv{y_2}\sigma(\beta))k\in \mc{L}h$. Since $\theta_1\in (\pi/4,3\pi/4)$ we have that \[|R_1(\theta_1,0)|,\footnote{Throughout the article, $|I|$ denotes the length of an interval $I\subset \R$.}\quad |R_2(\theta_2,0)|\gg s^{1/4}\] and hence there exists $\b\in\mc{O}_K$ so that $(0,\inv{y_1}\beta,0,\inv{y_2}\sigma(\beta))k\in \mc{L}h\cap\mc{T}(s)$.
	\end{proof}

	We will need the following fact.
	
	\begin{lem}
		\label{lem-L'}
		There is a constant $C_2>0$ such that for any intervals $R_1,R_2\subset \R$ with $|R_1|\cdot |R_2|\geq C_2$ we have $\mc{L}'\cap (R_1\times R_2)\neq \emptyset$.
	\end{lem}
	
	\begin{proof}
		This is a direct consequence of the fact that $\mc{L}'$ is invariant under $\mathrm{diag}(u,\s(u))$ for each $u\in \mc{O}_K^*$.
	\end{proof}
	
	In view of \Cref{lem-J} we only have to consider 
	\[\theta_1\in J:=[0,\pi)\setminus(\tfrac{\pi}{4},\tfrac{3\pi}{4})\]
	in the study of $G_M(s)$. Define $c(s)\in (\l^{-2},1]$ by $\l^{2r}=s^{1/4}c(s)$. Now define the intervals $\ell_1=\ell_1(\theta_1)$ and $\ell_2=\ell_2(\theta_2)$ through \[L_1(\theta_1)=\frac{s^{1/4}}{c(s)}\ell_1(\theta_1) \text{ and } L_2(\theta_2)=s^{1/4}c(s)\ell_2(\theta_2).\] Note that $\ell_1,\ell_2$ are independent of $s$ and that $\ell_2(\theta_2)$ is the horizontal projection of $\mc{W}k(-\theta_2)$. In the following lemma, we show that $G_M(s)$ decays as a constant times $s^{-1}$.
	
	\begin{lem}
		\label{lem-G_M-1}
		For sufficiently large $s>0$ we have
		\begin{equation}
		\label{eqn-G_M(s)-1}
		G_M(s)=\frac{2\sqrt{2}c_H}{s}\int_{0}^\pi\int_{ J}\int_{Y}\frac{I\left(\mc{L}'\cap \left(\inv{y_1}\ell_1(\theta_1)\times y_2^{-1}\ell_2(\theta_2)\right)=\emptyset\right)dy_1dy_2d\theta_1d\theta_2}{(y_1y_2)^3},\end{equation}
		where $Y:=\{(y_1,y_2)\mid y_1,y_2>0, y_2\in(1,\l^2)\}$.
	\end{lem}
	
	\begin{proof}
		Take $s>0$ so large that \Cref{lem-siegel-reduction} holds, i.e.\ so that $\mf{D}_t'$ gives an irredundant representation of the lattices $\mc{L}h$, $h\in H$, that avoid $\mc{T}(s)$.
		From \eqref{eqn-G(s)-Main-Error}, \Cref{lem-charac-empty-intersection} and the definitions of $\ell_1(\theta_1)$ and $\ell_2(\theta_2)$ we have
		\begin{align*}
			&G_M(s)\\
			&=c_H\int_{0}^\pi\int_{0}^\pi\int_{Y_t}\int_{\mf{F}}I\left(\mc{L}'\cap\left(\frac{s^{1/4}}{y_1c(s)}\ell_1(\theta_1)\times\frac{s^{1/4}c(s)}{y_2}\ell_2(\theta_2)\right)=\emptyset\right)dx_1dx_2\frac{dy_1}{y_1^3}\frac{dy_2}{y_2^3}d\theta_1d\theta_2.
		\end{align*} 
		Since $\ell_1(\theta_1),\ell_2(\theta_2)\gg 1$, \Cref{lem-L'} implies that  $y_1y_2\gg s^{1/2}$ if the indicator function in the above integral is to be non-zero. Hence if $s$ is sufficiently large, with \[Y':=\{(y_1,y_2)\in \R^2\mid y_1y_2>0,y_1/y_2\in (\l^{-2},\l^2)\}\] we have 
		\begin{align*}
			&G_M(s)\\
			&=c_H\int_{0}^\pi\int_{0}^\pi\int_{Y'}\int_{\mf{F}}I\left(\mc{L}'\cap\left(\frac{s^{1/4}}{y_1c(s)}\ell_1(\theta_1)\times\frac{s^{1/4}c(s)}{y_2}\ell_2(\theta_2)\right)=\emptyset\right)dx_1dx_2\frac{dy_1}{y_1^3}\frac{dy_2}{y_2^3}d\theta_1d\theta_2.
		\end{align*} 
		Note that the integrand is independent of $x_1$ and $x_2$. Now, with the change of variables $y_i=s^{1/4}y_i'$ we find that
		\begin{equation}
		\label{eqn-lem-G_M-1-1}
		G_M(s)=\frac{2\sqrt{2}c_H}{s}\int_{0}^\pi\int_{ J}\int_{Y'}\frac{I\left(\mc{L}'\cap \left(\frac{1}{c(s)y_1'}\ell_1(\theta_1)\times\frac{c(s)}{y_2'} \ell_2(\theta_2)\right)=\emptyset\right)}{(y_1'y_2')^{3}}dy_1'dy_2'd\theta_1d\theta_2,\end{equation}
		using the fact that $\mathrm{area}(\mf{F})=2\sqrt{2}$.
		
		Next, we show that the integrand in the last expression for $G_M(s)$ is independent of $s$. To this end, let $S:\R_{>0}^2\longrightarrow \R_{>0}^2$ be the map $(y_1,y_2)\mapsto (\l^2 y_1,\l^{-2}y_2)$ and note that $\mc{L}'$ is invariant under $S$. Note also that $Y'$ differs from a fundamental domain of $\R_{>0}^2$ under the action of the group generated by $S$ only by a set of measure $0$. Another fundamental domain (up to a set of measure $0$) of $\R_{>0}^2$ under this action is \[Y'':=\{(y_1,y_2)\mid y_1,y_2>0, y_1\in(1,\l^2)\}.\] Using also the fact that the measure $y_1^{-3}y_2^{-3}dy_1dy_2$ is $S$-invariant we conclude that we can replace $Y'$ with $Y''$ in the expression \eqref{eqn-lem-G_M-1-1} for $G_M(s)$. With the change of variables $y_2'=c(s)y_2$ we find that
		\begin{equation}
		\label{eqn-lem-G_M-1-2}
		G_M(s)=\frac{2\sqrt{2}c_H}{sc(s)^2}\int_{0}^\pi\int_{ J}\int_{Y''}\frac{I\left(\mc{L}'\cap \left(\frac{1}{c(s)y_1'}\ell_1(\theta_1)\times y_2^{-1} \ell_2(\theta_2)\right)=\emptyset\right)}{(y_1'y_2)^{3}}dy_1'dy_2d\theta_1d\theta_2.\end{equation}
		Finally, a similar trick with $Y$ instead of $Y''$ and the change of variables $y_1'=y_1/c(s)$ gives the desired result.
	\end{proof}
	
	In the next proposition we evaluate the triple integral in \eqref{eqn-G_M(s)-1} explicitly.
	
	\begin{prop}
		\label{prop-main-decay}
		We have $G_M(s)=C_{\mc{P}}s^{-1}$ as for all sufficiently large $s$, where
		\[C_{\mc{P}}:=\frac{|\s(\l)|}{\sqrt{2}\z_K(2)^2}\]
		where $\zeta_K$ is the Dedekind zeta function of $K=\Q(\sqrt{2})$.
	\end{prop}
	
	\begin{proof}
		Let $R_1:=\l/\sqrt{2}$ denote the inner radius of $\mc{W}$ and $R_2:=\sqrt{2+\sqrt{2}}$ its outer radius. With $r_1:=R_1/\l^2$ we have
		\[(-r_1,r_1)\subset \ell_2(\theta_2)/y_2\subset (-R_2,R_2)\]
		for all $\theta_2\in [0,\pi)$ and $y_2\in (1,\l^2)$. Note also that $\ell_2(\theta_2)$ is symmetric about the origin, thus $2r_1<|\ell_2(\theta_2)|/y_2<2R_2$ for all $\theta_2\in [0,\pi)$ and $y_2\in (1,\l^2)$.
		
		Consider now the following figure.
		
		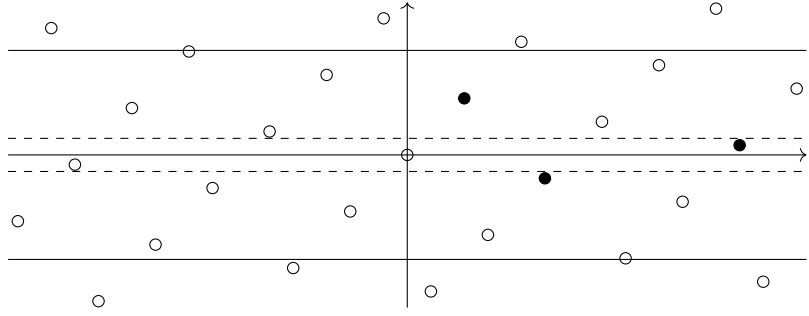
\begin{figure}[H]
			\centering
			\begin{tikzpicture}[scale=0.75]
				\newcommand \xmin {-7};
				\newcommand \xmax {7};
				\newcommand \ymin {-2.7};
				\newcommand \ymax {2.7};
				\newcommand \rone {0.29289};
				\newcommand \Rtwo {1.847759};
				\clip(\xmin,\ymin) rectangle (\xmax,\ymax);
				\fill[black] ({1+2^0.5},{1-2^0.5}) circle (0.1);
				\fill[black] (1,1) circle (0.1);
				\fill[black] ({3+2*2^0.5},{3-2*2^0.5}) circle (0.1);
				\draw[dashed] (\xmin,\rone) -- (\xmax,\rone);
				\draw[dashed] (\xmin,-\rone) -- (\xmax,-\rone);
				\draw (\xmin,\Rtwo) -- (\xmax,\Rtwo);
				\draw (\xmin,-\Rtwo) -- (\xmax,-\Rtwo);
				\draw[->] (\xmin,0) -- (\xmax,0);
				\draw[->] (0,\ymin) -- (0,\ymax);
				\foreach \x in {-10,...,10}
				\foreach \y in {-10,...,10} 
				{\draw ({\x+\y*2^0.5},{\x-\y*2^0.5}) circle (0.1);}
			\end{tikzpicture}
			\caption{The lattice $\mc{L}'$ and the strips of width $2r_1$ (dashed) and $2R_2$ (solid). The black points are, from left to right $(1,1)$, $(\l,\s(\l))$ and $(\l^2,\s(\l^2))$.}
			\label{fig-minkEmb}
		\end{figure}
	
		From \Cref{fig-minkEmb} we conclude that if $2r_1< |\ell_2(\theta_2)|/y_2\leq 2|\s(\l)|$, then for $\theta_1\in J,y_1>0$ we have \[\mc{L}'\cap \left(\inv{y_1}\ell_1(\theta_1)\times \inv{y_2}\ell_2(\theta_2)\right)=\emptyset\] if and only if $|\ell_1(\theta_1)|/y_1<\l^2$. Let
		\[S_1:=\{(y_2,\theta_2)\in (1,\l^2)\times[0,\pi): 2r_1< |\ell_2(\theta_2)|/y_2\leq 2|\s(\l)|\}.\]
		Similarly, if $2|\s(\l)|<|\ell_2(\theta_2)|/y_2\leq 2$ then for $\theta_1\in J,y_1>0$ we have \[\mc{L}'\cap \left(\inv{y_1}\ell_1(\theta_1)\times \inv{y_2}\ell_2(\theta_2)\right)=\emptyset\] if and only if $|\ell_1(\theta_1)|/y_1<\l$. Let
		\[S_2:=\{(y_2,\theta_2)\in (1,\l^2)\times[0,\pi): 2|\s(\l)|<|\ell_2(\theta_2)|/y_2\leq 2\}.\]
		Finally, if $2<|\ell_2(\theta_2)|/y_2\leq 2R_2$ then for $\theta_1\in J,y_1>0$ we have \[\mc{L}'\cap \left(\inv{y_1}\ell_1(\theta_1)\times \inv{y_2}\ell_2(\theta_2)\right)=\emptyset\] if and only if $|\ell_1(\theta_1)|/y_1<1$. Let
		\[S_3:=\{(y_2,\theta_2)\in (1,\l^2)\times[0,\pi): 2< |\ell_2(\theta_2)|/y_2<2R_2\}.\]
		
		We now write down an explicit expressions for $|\ell_1(\theta_1)|$ and $|\ell_2(\theta_2)|$. For $\theta_1\in [0,\pi/4]$ we have
		$|\ell_1(\theta_1)|=(\cos\theta_1+\sin\theta_1)\theta(\widehat{\mc{P}})^{-1/2}$, while $|\ell_1(\theta_1)|=|\cos\theta_1-\sin\theta_1|\theta(\widehat{\mc{P}})^{-1/2}$ for $\theta_1\in [3\pi/4,\pi)$ (these are the absolute values of the $x$-coordinates of $\theta(\widehat{\mc{P}})^{-1/2}(1,\pm1)k(-\theta_1)$).
		Note that $|\ell_2(\theta_2)|$ is periodic with period $\pi/4$ (due to the eightfold symmetry of $\mc{W})$ and that for $0\leq \theta_2\leq \pi/4$ we have
		$|\ell_2(\theta_2)|=\sqrt{2}(\l\cos\theta_2+\sin\theta_2)$ (this is twice the first coordinate of $\frac{1}{\sqrt{2}}(\l,-1)k(-\theta_2)$).
		
		We will now calculate the part of the triple integral in \eqref{eqn-G_M(s)-1} corresponding to integration over $S_1$ and $\theta_1\in [0,\pi/4]$. Note that the defining condition in $S_1$ implies that 
		\[\frac{|\ell_2(\theta_2)|}{2|\s(\l)|}\leq y_2<\frac{|\ell_2(\theta_2)|}{2r_1}.\]
		Now we observe that
		\[\frac{|\ell_2(\theta_2)|}{2r_1}\leq \frac{2R_2}{2r_1}\text{ and } \frac{|\ell_2(\theta_2)|}{2|\s(\l)|}\geq\l R_1=\frac{\l^2}{\sqrt{2}}>1.\]
		Note that $R_2/r_1>\l^2$ so we cannot use the first inequality above to improve the bound $1<y_2<\l^2$.
		We find that
		\begin{align*}
			&\int_{S_1}\int_0^{\pi/4}\int_{0}^\infty\frac{I\left(\mc{L}'\cap \left(\inv{y_1}\ell_1(\theta_1)\times y_2^{-1}\ell_2(\theta_2)\right)=\emptyset\right)dy_1d\theta_1dy_2d\theta_2}{(y_1y_2)^3}\\
			=&\int_{0}^\pi\int_{\frac{|\ell_2(\theta_2)|}{2|\s(\l)|}}^{\l^2}\int_0^{\pi/4}\int_{0}^\infty\frac{I\left(|\ell_1(\theta_1)|/\l^2<y_1\right)dy_1d\theta_1dy_2d\theta_2}{(y_1y_2)^3}\\
			=&\int_{0}^\pi\int_{\frac{|\ell_2(\theta_2)|}{2|\s(\l)|}}^{\l^2}\frac{dy_2d\theta_2}{y_2^3}\int_0^{\pi/4}\int_{|\ell_1(\theta_1)|/\l^2}^\infty\frac{dy_1d\theta_1}{y_1^3}\\
			=&4\int_{0}^{\pi/4}\left(\frac{\s(\l)^2}{(\l\cos\theta_2+\sin\theta_2)^2}-\frac{1}{2\l^4}\right)d\theta_2\int_0^{\pi/4}\frac{\l^4\theta(\widehat{\mc{P}})d\theta_1}{2(\cos\theta_1+\sin\theta_1)^2}\\
			=&\theta(\widehat{\mc{P}})\l^4\left(\frac{1}{\sqrt{2}\l^4}-\frac{\pi}{8\l^4}\right)=\theta(\widehat{\mc{P}})\left(\frac{1}{\sqrt{2}}-\frac{\pi}{8}\right)
		\end{align*}
		using \begin{equation}
			\label{eqn-num-int}
			\int_0^{\pi/4}(\l\cos\theta_2+\sin \theta_2)^{-2}\,d\theta_2=\frac{1}{\sqrt{2}\l^2}\quad\text{ and }\quad\int_0^{\pi/4}(\cos\theta_1+\sin\theta_1)^{-2}\,d\theta_1=\frac{1}{2}.
		\end{equation}
		
		Using the fact that 
		\[\int_{3\pi/4}^{\pi}(\cos\theta_1-\sin\theta_1)^{-2}\,d\theta_1=\frac{1}{2}\]
		it is seen that 
		\begin{align*}
			&\int_{S_1}\int_{3\pi/4}^{\pi}\int_{0}^\infty\frac{I\left(\mc{L}'\cap \left(\inv{y_1}\ell_1(\theta_1)\times y_2^{-1}\ell_2(\theta_2)\right)=\emptyset\right)dy_1d\theta_1dy_2d\theta_2}{(y_1y_2)^3}\\&=\theta(\widehat{\mc{P}})\left(\frac{1}{\sqrt{2}}-\frac{\pi}{8}\right)
		\end{align*}
		as well.
		
		Next, we calculate the part of the triple integral in \eqref{eqn-G_M(s)-1} corresponding to integration over $S_2$ and $\theta_1\in [0,\pi/4]$. Note that the defining condition in $S_2$ implies that 
		\[\frac{|\ell_2(\theta_2)|}{2}\leq y_2<\frac{|\ell_2(\theta_2)|}{2|\s(\l)|}.\]
		Now we observe that
		\[1<R_1\leq \frac{|\ell_2(\theta_2)|}{2}\text{ and } \frac{|\ell_2(\theta_2)|}{2|\s(\l)|}\leq\l R_2<\l^2.\]
		We find that
		\begin{align*}
			&\int_{S_2}\int_0^{\pi/4}\int_{0}^\infty\frac{I\left(\mc{L}'\cap \left(\inv{y_1}\ell_1(\theta_1)\times y_2^{-1}\ell_2(\theta_2)\right)=\emptyset\right)dy_1d\theta_1dy_2d\theta_2}{(y_1y_2)^3}\\
			=&\int_{0}^\pi\int_{\frac{|\ell_2(\theta_2)|}{2}}^{\frac{|\ell_2(\theta_2)|}{2|\s(\l)|}}\int_0^{\pi/4}\int_{0}^\infty\frac{I\left(|\ell_1(\theta_1)|/\l<y_1\right)dy_1d\theta_1dy_2d\theta_2}{(y_1y_2)^3}\\
			=&\int_{0}^\pi\int_{\frac{|\ell_2(\theta_2)|}{2}}^{\frac{|\ell_2(\theta_2)|}{2|\s(\l)|}}\frac{dy_2d\theta_2}{y_2^3}\int_0^{\pi/4}\int_{|\ell_1(\theta_1)|/\l}^\infty\frac{dy_1d\theta_1}{y_1^3}\\
			=&4\int_{0}^{\pi/4}\frac{(1-|\s(\l)|^2)}{(\l\cos\theta_2+\sin\theta_2)^2}d\theta_2\int_0^{\pi/4}\frac{\l^2\theta(\widehat{\mc{P}})d\theta_1}{2(\cos\theta_1+\sin\theta_1)^2}=\theta(\widehat{\mc{P}})\left(\frac{1}{\sqrt{2}}-\frac{1}{\sqrt{2}\l^2}\right),
		\end{align*}
		again using the integrals in \eqref{eqn-num-int}. As before, it is seen that
		\begin{align*}
			&\int_{S_2}\int_{3\pi/4}^{\pi}\int_{0}^\infty\frac{I\left(\mc{L}'\cap \left(\inv{y_1}\ell_1(\theta_1)\times y_2^{-1}\ell_2(\theta_2)\right)=\emptyset\right)dy_1d\theta_1dy_2d\theta_2}{(y_1y_2)^3}=\theta(\widehat{\mc{P}})\left(\frac{1}{\sqrt{2}}-\frac{1}{\sqrt{2}\l^2}\right)
		\end{align*}
		as well.
		
		Finally, we calculate the part of the triple integral in \eqref{eqn-G_M(s)-1} corresponding to integration over $S_3$ and $\theta_1\in [0,\pi/4]$. The defining condition in $S_3$ gives
		\[\frac{|\ell_2(\theta_2)|}{2R_2}\leq y_2<\frac{|\ell_2(\theta_2)|}{2}.\]
		Now we note that
		\[\frac{R_1}{R_2}\leq \frac{|\ell_2(\theta_2)|}{2R_2}\text{ and } \frac{|\ell_2(\theta_2)|}{2}\leq R_2<\l^2.\] The first bound cannot be used to improve the bound $1<y_2<\l^2$.
		We find that
		\begin{align*}
			&\int_{S_3}\int_0^{\pi/4}\int_{0}^\infty\frac{I\left(\mc{L}'\cap \left(\inv{y_1}\ell_1(\theta_1)\times y_2^{-1}\ell_2(\theta_2)\right)=\emptyset\right)dy_1d\theta_1dy_2d\theta_2}{(y_1y_2)^3}\\
			=&\int_{0}^\pi\int_{1}^{\frac{|\ell_2(\theta_2)|}{2}}\int_0^{\pi/4}\int_{0}^\infty\frac{I\left(|\ell_1(\theta_1)|<y_1\right)dy_1d\theta_1dy_2d\theta_2}{(y_1y_2)^3}\\
			=&\int_{0}^\pi\int_{1}^{\frac{|\ell_2(\theta_2)|}{2}}\frac{dy_2d\theta_2}{y_2^3}\int_0^{\pi/4}\int_{|\ell_1(\theta_1)|}^\infty\frac{dy_1d\theta_1}{y_1^3}\\
			=&4\int_{0}^{\pi/4}\left(\frac{1}{2}-\frac{1}{(\l\cos\theta_2+\sin\theta_2)^2}\right)d\theta_2\int_0^{\pi/4}\frac{\theta(\widehat{\mc{P}})d\theta_1}{2(\cos\theta_1+\sin\theta_1)^2}\\
			=&\theta(\widehat{\mc{P}})\left(\frac{\pi}{8}-\frac{1}{\sqrt{2}\l^2}\right)
		\end{align*}
		using \eqref{eqn-num-int}. As before, we have
		\begin{align*}
			&\int_{S_3}\int_{3\pi/4}^{\pi}\int_{0}^\infty\frac{I\left(\mc{L}'\cap \left(\inv{y_1}\ell_1(\theta_1)\times y_2^{-1}\ell_2(\theta_2)\right)=\emptyset\right)dy_1d\theta_1dy_2d\theta_2}{(y_1y_2)^3}=\theta(\widehat{\mc{P}})\left(\frac{\pi}{8}-\frac{1}{\sqrt{2}\l^2}\right)
		\end{align*}
		as well.
		
		Hence we conclude that
		\begin{align*}
		G_M(s)=&\inv{s}4\sqrt{2}c_H\theta(\widehat{\mc{P}})\left(\frac{1}{\sqrt{2}}-\frac{\pi}{8}+\frac{1}{\sqrt{2}}-\frac{1}{\sqrt{2}\l^2}+\frac{\pi}{8}-\frac{1}{\sqrt{2}\l^2}\right)\\
		=&\inv{s}8c_H\theta(\widehat{\mc{P}})\left(1-\frac{1}{\l^2}\right)=\inv{s}16c_H\theta(\widehat{\mc{P}})|\s(\l)|.
		\end{align*}
		By \cite[Theorem 4.9]{hammarhjelm2019density} we have $\theta(\widehat{\mc{P}})=\frac{1}{\zeta_K(2)}=\frac{48\sqrt{2}}{\pi^4}$, where $\zeta_K$ is the Dedekind zeta function of $K$. In \eqref{eqn-c_H} below we will show that $c_H=\frac{1}{8^{3/2}\zeta_K(2)}$. Thus, the claim of the lemma follows.
	\end{proof}	
	
	\begin{prop}
		\label{prop-G(s)-asympt}
		For large $s>0$ we have $G_E(s)\ll s^{-9/8}$, hence $G(s)=C_{\mc{P}}s^{-1}+\mc{O}(s^{-9/8})$ as $s\to\infty$, where $C_{\mc{P}}$ is the constant given in the statement of \Cref{prop-main-decay}.
	\end{prop}
	
	\begin{proof}
		For $x_1,x_2\in \mf{F}$ and $\theta_1,\theta_2\in [0,\pi)$, define
		\[V=V(x_1,x_2,\theta_1,\theta_2):=\{(y_1,y_2)\in \R_{>0}^2\mid h\in \mf{D}_t,\mc{L}h\cap \mc{T}(s)=\emptyset,\widetilde{\mc{L}h}\cap\mc{T}(s)\neq\emptyset\},\]
		where we recall that $h=h(x_1,x_2,y_1,y_2,\theta_1,\theta_2)$. We now study $V$ for fixed $(\theta_1,\theta_2)\in [0,\pi/4]\times [0,\pi)$ and fixed $(x_1,x_2)\in\mf{F}$.
	
	If $(y_1,y_2)\in V$, then the fact that $h$ belongs to $\mf{D}_t$ implies that $y_1/y_2\in (\l^{-2},\l^2)$ which further implies that $(y_1,y_2)$ belongs to a sector of points with arguments in $[v,\pi/2-v]$ for some absolute constant $v>0$. In view of \Cref{lem-sizey1y2} we have $y_1,y_2\gg s^{1/4}$ if $s$ is sufficiently large. If we are to have $\widetilde{\mc{L}h}\cap \mc{T}(s)\neq \emptyset$, then in view of \Cref{lem-charac-empty-intersection}, we need $y_1,y_2\ll s^{1/4}$. Hence, $y_1,y_2\asymp s^{1/4}$ if $(y_1,y_2)\in V$ (we emphasise that the implied constants are independent of $\theta_1,\theta_2$). For $w\in [v,\pi/2-v]$, define
	\begin{align*}
		r_0(w)&=r_0(w,\theta_1,\theta_2)\\
		&:=\inf\{r>0\mid \widetilde{\mc{L}h}\cap \mc{T}(s)=\emptyset \text{ for all } (y_1,y_2)=r'(\cos w,\sin w)\text{ with }r'>r\}
	\end{align*}
	(recall that $\widetilde{\mc{L}h}$ is independent of $x_1,x_2$).
	
	Note that $r_0$ is a continuous function of $w$ and that $r_0(w)\asymp s^{1/4}$, where the implied constants can be chosen to be independent of $w$. Note that if $r<r_0(w)$ then $(y_1,y_2)=r(\cos w,\sin w)$ gives $\widetilde{\mc{L}h}\cap \mc{T}(s)\neq\emptyset$ by \Cref{lem-charac-empty-intersection}. We will now show that $V$ is contained in a thin strip around the curve defined by $r_0$.
	
	Assume that $r<r_0(w)$ and that $(y_1,y_2):=r(\cos w,\sin w)\in V$. Then, by \Cref{lem-charac-empty-intersection}, there is $\alpha\in \mc{O}_K\cap \R_{>0}$ such that $(\alpha y_1,\sigma(\alpha)y_2)\in L_1(\theta_1)\times L_2(\theta_2)$, that is, the intervals $R_1(\theta_1,\a y_1)$ and $R_2(\theta_2,\sigma(\alpha)y_2)$ (recall the definitions of these intervals just before \Cref{lem-charac-empty-intersection}) are non-empty which implies that $|\alpha y_1|,|\sigma(\alpha)y_2|\ll s^{1/4}$. Since $y_1,y_2\asymp s^{1/4}$ it follows that $0<|\a|,|\s(\a)|\ll 1$, i.e.\ there are only finitely many possibilites for $\a$.

	Recall the points of the form \eqref{eqn-Lh-element} that belong to $\mc{L}h$.
	Such a point belongs to $\mc{T}(s)$ if and only if \[(\beta,\sigma(\beta))\in (y_1R_1(\theta_1,\alpha y_1)-\alpha x_1)\times (y_2R_2(\theta_2,\sigma(\alpha)y_2)-\sigma(\alpha) x_2).\] 
	Let now $t:=1-s^{-1/8}$. Fix $C>0$ so that $|R_1(\theta_1,Cs^{1/4})|$, $|R_2(\theta_2,Cs^{1/4})|\gg s^{1/4}$ for all $(\theta_1,\theta_2)\in [0,\pi/4]\times [0,\pi)$. We may assume that $C$ is so small that for every $(y_1,y_2)\in V$ and every possible $\alpha$ we have $t\alpha y_1,t|\s(\alpha)|y_2\geq Cs^{1/4}$. Let \[a(\theta_1)\asymp s^{1/4}\] be the right end point of $L_1(\theta_1)$ (recall that $L_1(\theta_1)$ is the projection of $\l^{-2r}T(s)k(-\theta_1)$ onto the $x$-axis). We have that $|R_1(\theta_1,ta(\theta_1))|$ is a decreasing function of $\theta_1$ and \[|R_1(\pi/4,ta(\pi/4))|=|a(\pi/4)-ta(\pi/4)|=(1-t)|a(\pi/4)|=s^{-1/8}|a(\pi/4)|\gg s^{1/8}.\]
	Thus, $|R_1(\theta_1,t\alpha y_1)|\gg s^{1/8}$ since \[|R_1(\theta_1,t\alpha y_1)|\geq \min \{|R_1(\theta_1,Cs^{1/4})|,|R_1(\pi/4,ta(\pi/4))|\}.\] A bound $|R_2(\theta_2,t\sigma(\alpha)y_2)|\gg s^{1/8}$ is straightforward to establish. Hence, if $(y_1,y_2)\in V$ then $t(y_1,y_2)$ does not belong to $V$ (and neither does $t'(y_1,y_2)$ for any $t'<t$), as there for large $s$ certainly exists $\beta\in \mc{O}_K$ such that \[(\beta,\sigma(\beta))\in (ty_1R_1(\theta_1,\alpha ty_1)-\alpha x_1,ty_2R_2(\theta_2,\sigma(\alpha)ty_2)-\sigma(\alpha) x_2),\]
	so that $\mc{L}h(x_1,x_2,ty_1,ty_2,\theta_1,\theta_2)\cap \mc{T}(s)\neq \emptyset$.
	Thus, $V$ is contained in the set whose polar coordinates $(r,w)$ satisfy $v\leq w\leq \pi/2-v$, and $t r_0(w)\leq r\leq r_0(w)$. Hence we find that
	\begin{align*}
	&~\int_{ \mf{F}}\int_{0}^{\pi}\int_{0}^{\pi/4}\int_{(y_1,y_2)\in V}\frac{dy_1dy_2d\theta_1d\theta_2dx_1dx_2}{y_1^3y_2^3}
	\ll\int_{0}^{\pi}\int_{0}^{\pi/4}\int_{v}^{\pi/2-v}\int_{tr_0(w)}^{r_0(w)}\frac{r\,drdwd\theta_1d\theta_2}{r^6\cos^3w\sin^3w}\\
	&\ll\int_v^{\pi/2-v}\int_{r=tr_0(w)}^{r_0(w)}r^{-5}drdw \ll\int_v^{\pi/2-v}r_0(w)^{-4}\left(1-t^{-4}\right)dw\ll s^{-1}\left(1-t^{-4}\right)\ll s^{-9/8}.
	\end{align*}
	We bound the analogous integral with $\theta_1\in[3\pi/4,\pi)$ in a similar way. For $\theta_1\in(\pi/4,3\pi/4)$ we note that $R_1(\theta_1,0)$, $R_2(\theta_2,0)$ are long intervals, so $\mc{L}h\cap \mc{T}(s)$ cannot be empty. Hence the statement of the lemma follows.
	\end{proof}
	
	\section{Asymptotics of the limiting gap distribution}
	\label{sec-F(s)}
	
	Next, we turn to the main problem of the present study, which is to determine the asymptotic behaviour of the limiting gap distribution $F(s)$, as given in \eqref{eqn-F(s)}, in the particular case of the Ammann--Beenker like point set $\mc{P}$. Since $F(s)=-G'(s)$ and $G(s)$ decays like $C_{\mc{P}}s^{-1}$, we might heuristically expect that $F(s)$ decays like $C_{\mc{P}}s^{-2}$. The remainder of this paper is devoted to proving that this is indeed the case, by showing that $F(s)=C_{\mc{P}}'s^{-2}+\mc{O}(s^{-p})$, as $s\to\infty$, for some $p>2$ and $C_{\mc{P}}'\in \R$. By integrating $F(s)$, a general argument then shows that $C_{\mc{P}}'=C_\mc{P}$.
	
	Let $\mc{T}'(s)=T(s)\times \mc{W}$\label{Def-T'(s)}. Recall that $G(s)=\mu(\{x\in X\mid \mc{P}^x\cap T(s)=\emptyset\})$.  We set up the difference quotient
	\begin{align*}
	F(s)&=\lim_{\eta\to 0^+}\frac{1}{\eta}\left(\mu(\{x\in X\mid \mc{L}(x)\cap \mc{T}'(s)=\emptyset\}-\mu(\{x\in X\mid \mc{L}(x)\cap \mc{T}'(s+\eta)=\emptyset\})\right)\\
	&=\lim_{\eta\to 0^+}\frac{1}{\eta}\mu(\{x\in X\mid \mc{L}(x)\cap \mc{T}'(s)=\emptyset,\mc{L}(x)\cap \mc{T}'(s+\eta)\neq\emptyset\}).
	\end{align*}
	Thus, intuitively, we are interested in the probability that a random lattice $\mc{L}(x)$ avoids $\mc{T}'(s)$, given that it contains a point whose last two coordinates belong to $\mc{W}$ and whose first two coordinates belong to the vertical part of the boundary of $T(s)$. We will make this intuition precise by using measures that allow conditioning on a lattice containing a specific point.
	
	\subsection{Conditional measures}
	\label{sec-cond-meas}
	
	In the present section we will present some results from \cite{marklofKineticPrivate}, which contains a general treatment of conditional measures on homogeneous spaces associated to cut-and-project sets (similar to the contruction of conditional measure on the space of lattices in \cite[Section 7]{marklof2010distribution}). With the aim of giving a self-contained presentation, we will reproduce some proofs of important results from \cite{marklofKineticPrivate} in the particular case we are interested in, namely the homogeneous space associated to the Ammann--Beenker point set, i.e.\ when $H_g=gH\inv{g}$ with $H=\SL{2}{\R}^2$ and $g$ as in \eqref{Def-g}.
	
	Given $z\in \R^4$, let 
	\[(H_g)_z=\{h\in H_g\mid zh=z\}\label{Def-Hgz}\] be the subgroup of $H_g$ of matrices that fix $z$. Given $z_1\in \R^4$ and $z_2\in z_1H_g$, let $M_{z_1,z_2}\in H_g$ denote a matrix such that $z_1M_{z_1,z_2}=z_2$. One can then identify the quotient space $(H_g)_z\backslash H_g$ and the set $zH_g=\{zh\mid h\in H_g\}\subset \R^4$\label{Def-zH_g}, by the map $(H_g)_zh\mapsto zh$ with inverse $zH_g\ni z_1\mapsto (H_g)_zM_{z,z_1}$. Let $\Z^4_*=\Z^4\smpt{0}$\label{Def-Z^4_*}. For each $z\in \Z^4_*H_g$ we have $zH_g=(\R^2\smpt{0})^2\inv{g}$, a set full Lebesgue measure in $\R^4$. We also have that the restriction of Lebesgue measure to $zH_g$ is $H_g$-invariant. We conclude that $(H_g)_z\backslash H_g$ admits an $H_g$-invariant measure.
	
	Now, given $z\in \Z^4_*H_g$, fix the Haar measure $\mu_z$\label{Def-mu_z} on $(H_g)_z$\footnote{One verifies that $(H_g)_z$ is isomorphic to the abelian Lie group $\R^2$ for each $z$, hence its left and right Haar measures coincide.} so that 
	\begin{equation}\label{eqn-mu_z}\int_{H_g}f(h)\,d\mu(h)=\int_{zH_g}\int_{(H_g)_z}f(hM_{z,z_1})\,d\mu_z(h)\,dz_1\end{equation}
	for every $f\in L^1(H_g,\mu)$ (this is possible by e.g.\ \cite[Theorem 8.36]{knapp2013lie}, using the identification of $zH_g$ and $(H_g)_z\backslash H_g$). Since $\mu_z$ is a Haar masure, the integral $\int_{(H_g)_z}f(hM_{z,z_1})\,d\mu_z(h)$ is independent of the choice of $M_{z,z_1}$. Here $dz_1$ is the standard Lebesgue measure on $\R^4$. We here note a fact that will be useful later. For $z_1,z_2\in \Z^4_*H_g$ we have \begin{equation}
		\label{eqn-H_z-conj}
		(H_g)_{z_2}=M_{z_1,z_2}^{-1}(H_g)_{z_1}M_{z_1,z_2}
	\end{equation}
	and that $\mu_{z_2}$ is the pushforward of $\mu_{z_1}$ by the map $h\mapsto M_{z_1,z_2}^{-1}hM_{z_1,z_2}$ from $(H_g)_{z_1}$ to $(H_g)_{z_2}$.
	
	Given $z\in \R^4$, let 
	\[\label{Def-X(z)}X(z)=\{\Gamma h\in X\mid h\in H_g,z\in \Z^4h\}\subset X\]
	and, given $k\in \Z^4$, let 
	\begin{equation}
	\label{Def-X(k,z)}X(k,z)=\{\Gamma h\in X\mid h\in H_g,kh=z\}\subset X(z).
	\end{equation}
	Note that $X(z)\neq\emptyset$ if and only if $z\in \Z^4H_g$ and that $X(0)=X$, therefore we henceforth assume that $z\in \Z^4_*H_g$.
	Then we have $X(0,z)=\emptyset$, while the sets $X(k,z)$, $k\in\Z^4_*$, form a covering of $X(z)$. If $R\subset \Z^4$ is chosen as a set of representatives of $\Z^4_*$ under the action of $\Gamma_g$, then $X(k,z)$, $k\in R$, give a disjoint covering of $X(z)$. We fix a convenient choice of $R$. First, fix a set of representatives $L$\label{Def-L} of $(\mc{O}_K\smpt{0})/\mc{O}_K^*$ with $1\in L$, where $\mc{O}_K^*$ denotes the group of units of $\mc{O}_K$. Now choose $R$ so that
	\[\label{Def-R}\{\delta^{1/4}kg\mid k\in R\}=\{(0,\ell,0,\sigma(\ell))\mid \ell\in L\}.\]
	By noting that for $k_1,k_2\in \Z_*^4$ one has $k_1\Gamma_g=k_2\Gamma_g$ if and only if $\delta^{1/4}k_1g\Gamma_K=\delta^{1/4}k_2g\Gamma_K$, the fact that $R$ has the claimed property follows from the fact that $\{(0,\ell)\mid \ell\in L\}$ gives a set of representatives of $\mc{O}_K^2\setminus \{0\}$ under the action of $\SLK$.
	
	Given $k\in R$, let $\Gamma_k=\Gamma \cap (H_g)_k$; this is a lattice in $(H_g)_k$, as will be verified in \Cref{sec-reformulation} below. Given $z\in kH_g$, fix $M_{k,z}$. Then 
	\[X(k,z)=\{\Gamma h\mid h\in (H_g)_kM_{k,z}\}.\]
	The map $\Gamma_k\backslash (H_g)_k\longrightarrow X(k,z)$, $\Gamma_k h\mapsto \Gamma h M_{k,z}$ is well-defined and bijective. Thus, the push-forward of the $(H_g)_k$-invariant measure $\mu_k$ on $\Gamma_k\backslash (H_g)_k$ under this map gives a measure $\nu_z$\label{Def-nu_z} on $X(k,z)$ which is independent of the choice of $M_{k,z}$. We obtain a measure $\nu_z$ on $X(z)$ by combining the measures $\nu_z$ defined on each $X(k,z)$. We also consider $\nu_z$ as a measure on $X$ supported on $X(z)$.
	
	We state and reproduce the proof of the following result from \cite{marklofKineticPrivate} in the particular case of our $X$.
	
	\begin{thm} 
		\label{thm-prop-3.7}
		\emph{\cite[Prop.\ 3.7]{marklofKineticPrivate}}
		If $\mc{E}\subset X$ is any Borel set, then $z\mapsto \nu_z(\mc{E})$ is a measurable, almost everywhere defined function on $\R^4$, and for any Borel set $U\subset \R^4$ we have
		\[\int_\mc{E}\#(\Z^4_*h\cap U)\,d\mu(\Gamma h)=\int_U\nu_z(\mc{E})\,dz\]
		where $dz$ denotes standard Lebesgue measure on $\R^4$.
	\end{thm}
	
	\begin{proof}
		Let $\pi:H_g\longrightarrow X$ denote the projection map and set $\mathcal{E}_0=\inv{\pi}(\mc{E})$.
		For any choice of a fundamental domain $\mc{F}_k\subset (H_g)_k$ of $\Gamma_k\backslash(H_g)_k$ we have, for every $z\in \Z^4_*H_g$,
		\begin{equation}
		\label{eqn-nu_z-1}
		\nu_z(\mc{E})=\sum_{k\in R}\int_{\mc{F}_k}\chi_{\mc{E}_0}(hM_{k,z})\,d\mu_k(h).\end{equation}
		
		Let now $\mc{F}$ be a fundamental domain of $\Gamma_g\backslash H_g$. Fix $k\in R$. We have that $\mc{F}M_{k,z}^{-1}$ is also a fundamental domain of $\Gamma_g\backslash H_g$. Note that $\Gamma_k$ is a subgroup of $\Gamma_g$; let $S^{(k)}$ be a set of representatives so that
		$\bigsqcup_{\gamma\in S^{(k)}}\Gamma_k\gamma =\Gamma_g$
		(we use $\bigsqcup$ to denote a disjoint union).
		It is verified that $\mc{F}_k':=\bigsqcup_{\gamma\in S^{(k)}}\gamma \mc{F}M_{k,z}^{-1}$
		is a fundamental domain of $\Gamma_k\backslash H_g$ and that $\mc{F}_k=\mc{F}_k'\cap (H_g)_k$ is a fundamental domain of $\Gamma_k\backslash(H_g)_k$. We now use this particular choice of $\mc{F}_k$.
		
		We have
		\begin{align*}
		&\int_{\mc{F}_k}\chi_{\mc{E}_0}(hM_{k,z})\,d\mu_k(h)=\int_{(H_g)_k}\chi_{\mc{F}_k'\cap\mc{E}_0M_{k,z}^{-1}}(h)\,d\mu_k(h)\\
		&=\sum_{\gamma\in S^{(k)}}\int_{(H_g)_k}\chi_{\gamma\mc{F}M_{k,z}^{-1}\cap\mc{E}_0M_{k,z}^{-1}}(h)\,d\mu_k(h)=\sum_{\gamma\in S^{(k)}}\int_{(H_g)_k}\chi_{\mc{F}\cap\mc{E}_0}(\inv{\gamma}hM_{k,z})\,d\mu_k(h)
		\end{align*}
		where we in the last equality used that $\mc{E}_0$ is left $\Gamma_g$-invariant. Combining with \eqref{eqn-nu_z-1} we find that
		\begin{equation*}
		\nu_z(\mc{E})=\sum_{k\in R}\sum_{\gamma\in S^{(k)}}\int_{(H_g)_k}\chi_{\mc{F}\cap\mc{E}_0}(\inv{\gamma}hM_{k,z})\,d\mu_k(h).
		\end{equation*}
		Now fix any $m\in \Z_*^4$ once and for all. Then
		$mH_g=\Z^4_*H_g=(\R^2\smpt{0})^2\inv{g}$.
		It follows that
		\[\int_U\nu_z(\mc{E})\,dz=\int_{U\cap mH_g}\sum_{k\in R}\sum_{\gamma\in S^{(k)}}\int_{(H_g)_k}\chi_{\mc{F}\cap\mc{E}_0}(\inv{\gamma}hM_{k,z})\,d\mu_k(h)\,dz.\]
		For fixed $k$, using the fact that $(H_g)_k=\inv{M}_{m,k}(H_g)_mM_{m,k}$, we find that 
		\begin{align*}
		\int_{(H_g)_k}\chi_{\mc{F}\cap\mc{E}_0}(\inv{\gamma}hM_{k,z})\,d\mu_k(h)=\int_{(H_g)_m}\chi_{\mc{F}\cap\mc{E}_0}((M_{m,k}\gamma)^{-1}hM_{m,k}M_{k,z})\,d\mu_m(h).
		\end{align*}
		Note that $M_{m,k}\gamma\in (H_g)_mM_{m,k\gamma}$ and that $M_{m,k}M_{k,z}\in (H_g)_mM_{m,z}$. Using both the left and right $(H_g)_m$-invariance of $\mu_m$ we conclude that  
		\begin{align*}
		&\int_U\nu_z(\mc{E})\,dz=\int_{U\cap mH_g}\sum_{k\in R}\sum_{\gamma\in S^{(k)}}\int_{(H_g)_m}\chi_{\mc{F}\cap\mc{E}_0}(M_{m,k\gamma}^{-1}hM_{m,z})\,d\mu_m(h)\,dz.
		\end{align*}
		Now, observe that as $k$ and $\g$ run through the iterated sum above, $k\gamma$ ranges through $\Z^4\cap mH_g=\Z_*^4$, visiting each value exactly once. Thus
		\begin{align*}
		\int_U\nu_z(\mc{E})\,dz=\int_{U\cap mH_g}\sum_{k\in \Z^4_*}\int_{(H_g)_m}\chi_{\mc{F}\cap\mc{E}_0}(M_{m,k}^{-1}hM_{m,z})\,d\mu_m(h)\,dz.
		\end{align*}
		Let now $B_k=\{h'\in \mc{F}\cap \mc{E}_0\mid kh'\in U\}$. We have
		\begin{align*}
		\int_U\nu_z(\mc{E})\,dz=\sum_{k\in \Z^4_*}\int_{ mH_g}\int_{(H_g)_m}\chi_{B_k}(M_{m,k}^{-1}hM_{m,z})\,d\mu_m(h)\,dz.
		\end{align*} 
		Using \eqref{eqn-mu_z}, we find that
		\begin{align*}
		\int_U\nu_z(\mc{E})\,dz&=\sum_{k\in \Z^4_*}\int_{ H_g}\chi_{B_k}(M_{m,k}^{-1}h)\,d\mu(h)\\
		&=\sum_{k\in \Z^4_*}\mu(M_{m,k}B_k)=\sum_{k\in \Z^4_*}\mu(B_k)
		\end{align*} 
		by the $H_g$-invariance of $\mu$. Finally
		\begin{align*}
		\sum_{k\in \Z^4_*}\mu(B_k)&=\sum_{k\in \Z^4_*}\int_{\mc{F}\cap \mc{E}_0}I(kh\in U)\,d\mu(h)=\int_{\mc{F}\cap \mc{E}_0}\sum_{k\in \Z^4_*}I(kh\in U)\,d\mu(h)\\
		&=\int_{\mc{F}\cap\mc{E}_0}\#(\Z^4_*h\cap U)\,d\mu(h)=\int_\mc{E}\#(\Z^4_*h\cap U)\,d\mu(\Gamma h)
		\end{align*}
		and the proof is complete.
	\end{proof}
	
	\Cref{thm-prop-3.7} allows us to think of $\nu_z$ as a conditional measure on $X$, namely, it allows us to condition on the event that a lattice should contain a given point. Given $z\in (\R^2\smpt{0})^2$, let $\overline{z}=\delta^{-1/4}z\inv{g}\in\Z_*^4H_g$\label{Def-z-bar} and define $\overline{\nu}_z=\nu_{\overline{z}}$\label{Def-nu-bar}. As an immediate corollary of \Cref{thm-prop-3.7} we have:
	
	\begin{cor}
		\label{cor-prop-3.7} Given a Borel set $\mc{E}\subset X$ we have that $z\mapsto \overline{\nu}_z(\mc{E})$ is a measurable, almost everywhere defined function on $\R^4$ and for each Borel set $U\subset \R^4$ we have 
		\[\int_{\mc{E}}\#(\delta^{1/4}\Z^4_*hg\cap U)\,d\mu(\Gamma h)=\inv{\d}\int_{U}\overline{\nu}_z(\mc{E})\,dz.\]
	\end{cor}
	
	The next corollary will allow us to determine $c_H$. We reproduce the proof from \cite{marklofKineticPrivate}.
	
	\begin{cor}
		For every $z\in \Z^4_*H_g$, $\nu_z$ is a probability measure on $X$. 
	\end{cor}
	
	\begin{proof}
		Fix a Borel set $U\subset \R^4$ of positive, finite measure. Apply \cite[Theorem 5.1]{marklof2014free} (see also \cite{marklof2020correction}) with $f=\chi_{Ug}$ to obtain
		\[\int_X\#(\Z^4_*h\cap U)\,d\mu(\Gamma h)=\int_X\sum_{m\in \Z^4_*hg}\chi_{Ug}(m)\,d\mu(\Gamma h)=\vol(Ug)=\vol(U).\]
		By \Cref{thm-prop-3.7}, the left hand side is equal to $\int_U\nu_z(X)\,dz$. Since this holds for every Borel set $U$ of finite measure, we have that $\nu_z(X)=1$ for almost all $z\in \Z^4_*H_g$. In particular, this must hold for \textit{some} $z\in \Z^4_*H_g$. Now note that $zH_g=\Z^4_*H_g$ and that $\nu_{z}(X)=\nu_{zh}(X)$ for each $h\in H_g$ (since $\nu_z(X(k,z))=\nu_{zh}(X(k,zh))$ for every $k\in R$). Thus the claim of the corollary follows.
	\end{proof}
	
	\subsection{Reformulation of the problem using conditional measures}
	
	\label{sec-reformulation}
	
	Let us set $z_0=8^{1/4}(0,0,1,0)$ so that $z_0g=(0,1,0,1)$. It is then verified that the map \[(x_1,x_2)\mapsto g\,\mathrm{diag}(n(x_1),n(x_2))\inv{g}\] is Lie group isomorphism of $\R^2$ and $(H_g)_{z_0}$. With respect to this parametrisation, a Haar measure on $(H_g)_{z_0}$ is $dx_1\,dx_2$. Hence $\mu_{z_0}$ is a constant multiple of $dx_1\,dx_2$, in fact one verifies by a computation starting from \eqref{eqn-mu_z} that $\mu_{z_0}=c_Hdx_1\,dx_2$, where $c_H$ is the constant introduced in the beginning of \Cref{sec-G_M(s)}. Fix $k\in (\R^2\smpt{0})^2$. We then have $(H_g)_k=M_{z_0,k}^{-1}(H_g)_{z_0}M_{z_0,k}$ by \eqref{eqn-H_z-conj}, hence the map
	\[(x_1,x_2)\mapsto M_{z_0,k}^{-1} g\,\mathrm{diag}(n(x_1),n(x_2))\inv{g}M_{z_0,k}\]
	is a Lie group isomorphism of $\R^2$ and $(H_g)_k$. In this parametrisation we have $\mu_k=c_Hdx_1\,dx_2$ as well. Now choose $\ell\in L$ so that $\d^{1/4}kg=(0,\ell,0,\s(\ell))$. Then one can take 
	\[M_{z_0,k}=g\,\mathrm{diag}(a(8^{1/4}\inv{\ell}),a(8^{1/4}\inv{\s(\ell})))\inv{g}.\]
	This turns the above parametrisation of $(H_g)_k$ to
	\[(x_1,x_2)\mapsto  g\,\mathrm{diag}(n(8^{-1/2}\ell^2x_1),n(8^{-1/2}\s(\ell)^2x_2))\inv{g}.\]
	By making a linear change of variables we can modify this parametrisation to
	\begin{equation}
		\label{eqn-param-gamma-k}
		(x_1,x_2)\mapsto  g\,\mathrm{diag}(n(x_1),n(x_2))\inv{g}
	\end{equation}
	and with respect to this parametrisation we have $\mu_k=8N(\ell)^{-2}c_Hdx_1dx_2$. We note that in this parametrisation, the subgroup $\Gamma_k=\Gamma \cap (H_g)_k$ corresponds to the lattice $\mc{L}'\subset \R^2$, verifying the earlier claim that $\Gamma_k$ is a lattice in $(H_g)_k$.
	
 	Let us note that \Cref{cor-prop-3.7} allows us to determine $c_H$ explicitly. By \Cref{cor-prop-3.7} we know, on one hand, that $\nu_{z}(X)=1$ for every $z\in \Z^4_*H_g$. We calculate $\nu_{z}(X)$ in another way, using the covering $X(k,z)$, $k\in R$, of $X(z)$. Fix $k\in R$ and a corresponding $\ell\in L$. By the definition of $\nu_z$ we have
 	\[\nu_z(X(k,z))=\mu_k(\Gamma_k\backslash (H_g)_k)=8N(\ell)^{-2}c_H\int_{\R^2\backslash \mc{L}'}dx_1dx_2=8^{3/2}N(\ell)^{-2}c_H.\]
 	It follows that
 	\[1=\nu_z(X)=\sum_{k\in R}\nu_z(X(k,z))=8^{3/2}c_H\sum_{\ell\in L}N(\ell)^{-2}=8^{3/2}c_H\zeta_K(2)\]
 	where $\zeta_K$ is the Dedekind zeta function of $K$.
 	We conclude that
 	\begin{equation}
 		\label{eqn-c_H}
 		c_H=\frac{1}{8^{3/2}\zeta_K(2)}.
 	\end{equation}
 	The value of $c_H$ can also be obtained using results from \cite{siegel1936volume}.
	
	Recall that	
	\begin{equation}F(s)=\lim_{\eta\to 0^+}\frac{1}{\eta}\underbrace{\mu(\{x\in X\mid \mc{L}(x)\cap \mc{T}'(s)=\emptyset,\mc{L}(x)\cap \mc{T}'(s+\eta)\neq\emptyset\})}_{=:F_1(s,\eta)}\label{Def-F_1}.\end{equation}
	
	For $s>0$, let \[\mc{E}(s):=\{x\in X\mid \mc{L}(x)\cap \mc{T}'(s)=\emptyset\}\label{def-mc-E}.\]	
	Let also \[c:=\theta(\widehat{\mc{P}})^{1/2}\label{Def-c}.\]
	Given $z\in \R_{>0}\times \R\times \mc{W}$, with first coordinate $z_1$, note that $z\in \partial T(z_1^2c^2)\times \mc{W}$. 
	
	\begin{lem}
		\label{lem-single-value-of-k}
		For any $z_1>0$, $|z_2|<z_1$, $k\in R$, $w\in \mc{W}$ and $z=(z_1,z_2,w)$, if either $k\neq (0,0,1,0)$ or $w\in |\s(\l)|\mc{W}$, then $X(k,\overline{z})\cap \mc{E}(z_1^2c^2)=\emptyset$.
	\end{lem}
	
	\begin{proof}
		Let $s=z_1^2c^2$.
		We have $\delta^{1/4}(0,0,1,0)g=(0,1,0,1)$. Thus, if $k\in R\smpt{(0,0,1,0)}$, then we have $\delta^{1/4}kg=(0,\ell,0,\s(\ell))$, where $\ell$ is not a unit, i.e.\ there is a prime $\pi\in \mc{O}_K$ such that $\pi\mid \ell$. We may assume that $1< \pi<\l$, from which it follows that $|\sigma(\pi)|>1$. Assume towards a contradiction that there is $x\in X(k,\overline{z})\cap \mc{E}(s)$. Write $x=\Gamma gh\inv{g}$ for some $h\in H$ with $kgh\inv g=\overline{z}$. We have that $kgh\inv g=\overline{z}$ which is equivalent with $(0,\ell,0,\sigma(\ell))h=z$. Furthermore, $z\in \mc{L}h=\mc{L}(x)$ and $z\in (\partial T(s)\times \mc{W})$. It follows that $(0,\ell/\pi,0,\sigma(\ell/\pi))h\in \mc{L}h$, but this contradicts $\mc{L}(x)\cap \mc{T}'(s)=\emptyset$, since $T(s)$ and $\mc
		W$ are star-shaped with respect to $0$ and $\mc{W}$ satisfies $-\mc{W}=\mc{W}$.
		
		Suppose now that $w\in |\s(\l)|\mc{W}$. Then, given $x=\Gamma gh\inv{g}\in X(k,\overline{z})\cap \mc{E}(s)$ with $h$ as before, we have $(0,\ell,0,\s(\ell))h=z\in (\partial T(s)\times |\s(\l)|\mc{W})\cap \mc{L}(x)$ but then the point $(0,\ell/\l,0,\s(\ell/\l))h\in\mc{L}(x)$ belongs to $\mc{T}'(s)$. This is a contradiction.
	\end{proof}
	
	\begin{lem}
		\label{lem-F1}
		For $s,\eta>0$ we have
		\[F_1(s,\eta)=\d^{-1}\int_{\mc{T}'(s+\eta)\setminus \mc{T}'(s)}\overline{\nu}_z(\mc{E}(z_1^2c^2))\,dz\]
		(recall the definition of $F_1(s,\eta)$ from \eqref{Def-F_1}).
	\end{lem}
	
	\begin{proof}
		Fix $s,\eta>0$ and $N\in \bb{Z}_{>0}$. Let $\xi_j=s+\frac{j}{2^N}\eta$ for $j\in\{0,1,\ldots,2^N\}$ and
		\[\mc{E}_j^r=\{x\in X\mid \#(\mc{L}(x)\cap \mc{T}'(\xi_j))\leq r\}\]
		for $r\in \{0,1\}$. Let also $\mc{T}'_j=\mc{T}'(\xi_j)\setminus \mc{T}'(\xi_{j-1})$ for $j\in \{1,2,\ldots,2^N\}$.
		
		Note that we can write the set measured in $F_1(s,\eta)$ as a disjoint union
		\[\bigcup_{j=1}^{2^N}\{x\in X\mid \mc{L}(x)\cap \mc{T}'(\xi_{j-1})=\emptyset,\mc{L}(x)\cap \mc{T}_j'\neq \emptyset\}.\]
		Furthermore, for each $1\leq j\leq 2^N$ we have
		\begin{align*}
		&\mu(\{x\in X\mid \mc{L}(x)\cap \mc{T}'(\xi_{j-1})=\emptyset,\mc{L}(x)\cap \mc{T}_j\neq \emptyset\})\\
		\geq & \mu(\{x\in X\mid \mc{L}(x)\cap \mc{T}'(\xi_{j-1})=\emptyset,\#(\mc{L}(x)\cap \mc{T}_j)=1\})\\
		=&\mu(\{x\in X\mid \#(\mc{L}(x)\cap \mc{T}'(\xi_{j}))\leq 1,\#(\mc{L}(x)\cap \mc{T}_j)=1\})=\int_{\mc{E}_j^1}\#(\mc{L}(x)\cap \mc{T}'_j)\,d\mu.
		\end{align*}
		We also have 
		\begin{align*}
		\mu(\{x\in X\mid \mc{L}(x)\cap \mc{T}'(\xi_{j-1})=\emptyset,\mc{L}(x)\cap \mc{T}_j\neq \emptyset\})\leq\int_{\mc{E}_{j-1}^0}\#(\mc{L}(x)\cap \mc{T}'_j)\,d\mu.
		\end{align*}
		We conclude that 
		\[\sum_{j=1}^{2^N}\int_{\mc{E}_j^1}\#(\mc{L}(x)\cap \mc{T}'_j)\,d\mu\leq F_1(s,\eta)\leq\sum_{j=1}^{2^N}\int_{\mc{E}_{j-1}^0}\#(\mc{L}(x)\cap \mc{T}'_j)\,d\mu.\]
		By \Cref{cor-prop-3.7} we have 
		\[\sum_{j=1}^{2^N}\int_{\mc{E}_{j-1}^0}\#(\mc{L}(x)\cap \mc{T}'_j)\,d\mu=\inv{\d}\sum_{j=1}^{2^N}\int_{\mc{T}'_j}\overline{\nu}_z(\mc{E}_{j-1}^0)\,dz\]
		where the right hand side can further be simplified to 
		\[\inv{\d}\int_{\mc{T}'(s+\eta)\setminus \mc{T}'(s)}\overline{\nu}_z(\{x\in X\mid \mc{L}(x)\cap \mc{T}'(\xi_-^N(z))=\emptyset\})\,dz\]
		with $\xi_-^N(z):=\max\{\xi_j\mid z\notin \mc{T}'(\xi_j)\}$. Note that for every $z\in \mc{T}'(s+\eta)\setminus \mc{T}'(s)$, the sequence $\xi_-^N(z)$ increases with $N$. Hence, the sequence of sets $\{x\in X\mid \mc{L}(x)\cap \mc{T}'(\xi_-^N(z))=\emptyset\}$ decreases with $N$. Note that $\xi_-^N(z)$ tends to $z_1^2c^2$ as $N\to\infty$. We then have pointwise convergence
		\[\overline{\nu}_z(\{x\in X\mid \mc{L}(x)\cap \mc{T}'(\xi_-^N(z))=\emptyset\})\to \overline{\nu}_z(\{x\in X\mid \mc{L}(x)\cap \mc{T}'(z_1^2c^2)=\emptyset\})\]
		for every $z\in \mc{T}'(s+\eta)\setminus \mc{T}'(s)$. By Lebesgue's theorem of dominated convergence we have 
		\begin{equation}
		\label{eqn-lem-F_1-1}\inv{\d}
		\sum_{j=1}^{2^N}\int_{\mc{E}_{j-1}^0}\#(\mc{L}(x)\cap \mc{T}'_j)\,d\mu\longrightarrow \int_{\mc{T}'(s+\eta)\setminus \mc{T}'(s)}\overline{\nu}_z(\{x\in X\mid \mc{L}(x)\cap \mc{T}'(z_1^2c^2)=\emptyset\})\,dz
		\end{equation}
		as $N\longrightarrow\infty$.
		
		Now, once more applying \Cref{cor-prop-3.7}, we find that 
		\[\sum_{j=1}^{2^N}\int_{\mc{E}_{j}^1}\#(\mc{L}(x)\cap \mc{T}'_j)\,d\mu=\inv{\d}\sum_{j=1}^{2^N}\int_{\mc{T}'_j}\overline{\nu}_z(\mc{E}_{j}^1)\,dz\]
		where the right hand side can be written as
		\[\inv{\d}\int_{\mc{T}'(s+\eta)\setminus \mc{T}'(s)}\overline{\nu}_z(\{x\in X\mid \#(\mc{L}(x)\cap \mc{T}'(\xi_+^N(z)))\leq1\})\,dz\]
		with $\xi_+^N(z):=\min\{\xi_j\mid z\in \mc{T}'(\xi_j)\}$. Note that $\xi_+^N(z)$ decreases pointwise with $N$ and hence that the sets $\{x\in X\mid \#(\mc{L}(x)\cap \mc{T}'(\xi_+^N(z)))\leq1\}$ increase with $N$ towards \[\{x\in X\mid \#(\mc{L}(x)\cap \mc{T}'_1(z_1^2c^2))\leq1\}\]
		where $\mc{T}'_1(z_1^2c^2)=\bigcap_{\xi>z_1^2c^2}\mc{T}'(\xi)$ and $z_1^2c^2$ is the limit of $\xi_+^N(z)$ as $N\longrightarrow \infty$. By another application of Lebesgue's theorem of dominated convergence we conclude that 
		\begin{equation}
		\label{eqn-lem-F_1-2}
		\sum_{j=1}^{2^N}\int_{\mc{T}'_j}\overline{\nu}_z(\mc{E}_{j}^1)\,dz\longrightarrow \inv{\d}\int_{\mc{T}'(s+\eta)\setminus \mc{T}'(s)}\overline{\nu}_z(\{x\in X\mid \#(\mc{L}(x)\cap \mc{T}'_1(z_1^2c^2))\leq1\})\,dz
		\end{equation}
		as $N\to\infty$.
		
		It now remains to show that the right hand sides of equations \eqref{eqn-lem-F_1-1} and \eqref{eqn-lem-F_1-2} are in fact equal. To this end, it suffices to show that for Lebesgue almost every $z\in \mc{T}'(s+\eta)\setminus \mc{T}'(s)$ we have 
		\[\#(\mc{L}(x)\cap \mc{T}'_1(z_1^2c^2))\leq1 \iff\mc{L}(x)\cap \mc{T}'(z_1^2c^2)=\emptyset\]
		for $\overline{\nu}_z$-almost every $x\in X$. Note that the right implication is always true, since $z\in \mc{L}(x)\cap \mc{T}'_1(z_1^2c^2)$. Suppose therefore that $\mc{L}(x)\cap \mc{T}'(z_1^2c^2)=\emptyset$ and $\#(\mc{L}(x)\cap \mc{T}'_1(z_1^2c^2))\geq 2$. By \Cref{lem-single-value-of-k} we can assume that $x\in X(k,\overline{z})$, with $k=(0,0,1,0)$. 
		
		Let us therefore fix $z=(z_1,z_2,w)\in \mc{T}'(s+\eta)\setminus \mc{T}'(s)$  and assume that there is $x\in X(k,\overline{z})$ such that $\#(\mc{L}(x)\cap \mc{T}'_1(z_1^2c^2))\geq 2$. Write $x=\Gamma gh\inv{g}$ where $h=\mathrm{diag}(h_1,h_2)$. Then $\mc{L}(x)=\mc{L}h$ and there must be two distinct points of $\mc{L}h$ whose first coordinates agree. If $h_1=n(x_1)a(y_1)k(\theta_1)$, where $y_1,\theta_1$ are determined by $(z_1,z_2)=\inv{y_1}(\sin\theta_1,\cos\theta_1)$, this implies that there must exist distinct pairs $(\a_1,\b_1),(\a_2,\b_2)\in \mc{O}_K^2$ such that
		$(\a_i,\b_i)h_1$, $i\in\{1,2\}$, share the same first coordinate.
		
		A straightforward calculation then shows that \[y_1^2z_2(\a_1-\a_2)+z_1(\b_1-\b_2)=x_1z_1(\a_2-\a_1).\]
		Now, $\a_1=\a_2$ implies that $z_1=0$, so this case is excluded. Hence, as there are only countably many possibilities for $\a_1,\a_2,\b_1,\b_2$ we see that $x_1$ is restricted to a countable set. This finishes the proof as $\overline{\nu}_z$ on $X(k,\overline{z})$ is $dx_1\,dx_2$ (where $x_2$ appears as a parameter in $h_2$).
	\end{proof}	
	
	Now parametrise $\mc{T}'(s+\eta)\setminus \mc{T}'(s)$ by \[\{(z_1,z_2,w)\mid z_1\in [s^{1/2}c^{-1},(s+\eta)^{1/2}c^{-1}),|z_2|<z_1,w\in\mc{W}\}.\] Writing $z=(z_1,z_2,w)$ with $z_1,z_2\in \R$ and $w\in \R^2$ then \Cref{lem-single-value-of-k} gives
	\begin{equation}F_1(s,\eta)=\inv{\d}\int_{z_1=s^{1/2}c^{-1}}^{(s+\eta)^{1/2}c^{-1}}\underbrace{\int_{z_2=-z_1}^{z_1}\int_{\mc{W}'}\nu_{\overline{z}}(\{X(k,\overline{z})\cap \mc{E}(z_1^2c^2)=\emptyset\})dwdz_2}_{=:F_2(z_1)}dz_1\label{Def-F_2}\end{equation}
	where $k:=(0,0,1,0)$ and \[\mc{W}':=\mc{W}\setminus |\s(\l)|\mc{W}\label{Def-W-2}.\]
	Note that 
	\begin{equation}
		\label{eqn-X(k,z)}
		X(k,\overline{z})=\{\Gamma gh\inv{g}\mid h\in H: (0,1,0,1)h=z\}.
	\end{equation}
	
	Recall that for each $z=(z_1,z_2,w)\in (\R^2\setminus \{0\})^2$, the measure $\nu_{\overline{z}}$ is the push-forward of $\mu_k$ on $\Gamma_k\backslash (H_g)_k$ under the map $\Gamma_kh\mapsto \Gamma h M_{k,\overline{z}}$. With respect to the parametrisation of $\Gamma_k\backslash (H_g)_k$ given in \eqref{eqn-param-gamma-k}, we have $\mu_k=\frac{dx_1dx_2}{8^{1/2}\zeta_K(2)}$. Hence 
	\begin{equation}
	\label{eqn-F_2(z_1)}
		F_2(z_1)=\frac{1}{8^{1/2}\zeta_K(2)}\int_{z_2=-z_1}^{z_1}\int_{\mc{W}'}\int_{\mf{F}}I(\mc{L}h\cap \mc{T}'(z_1^2c^2)=\emptyset)dx_1dx_2dwdz_2
	\end{equation}
	where $\mf{F}\subset \R_{>0}^2$ is a fixed fundamental domain of $\mc{L}'$ and $h=h(x_1,x_2,y_1,y_2,\theta_1,\theta_2)\in H$ (cf.\ \eqref{Def-h}) has Iwasawa parameters $y_1,y_2,\theta_1,\theta_2$ determined by
	\begin{equation}
	\label{eqn-y_i-theta_i-determined}
	\inv{y_1}(\sin\theta_1,\cos\theta_1)=(z_1,z_2),\quad \inv{y_2}(\sin\theta_2,\cos\theta_2)=w.
	\end{equation}
	
	\begin{lem}
		\label{lem-F2-cont}
		The function $F_2(z_1)$ is continuous at every $z_1\geq 0$.
	\end{lem}
	
	\begin{proof}		
		It is clear that $F_2$ is right continuous at $0$. Thus, it suffices to prove that for any $C>1$, $F_2(z_1)$ is continuous on the interval $[C^{-1},C]$. To this end, fix $C>1$. We observe that as $(z_1,z_2,w,x_1,x_2)$ varies through $[C^{-1},C]\times [-C,C]\times \mc{W}'\times \mf{F}$, the set $\mc{T}'(z_1^2c^2)\inv{h}$ stays within a bounded subset of $\R^4$. Let $\mc{L}_C$ be the intersection of $\mc{L}$ with this bounded subset. Note that $\mc{L}_C$ is finite. We then have
		\[F_2(z_1)=\frac{1}{8^{1/2}\zeta_K(2)}\int_{z_2=-z_1}^{z_1}\int_{\mc{W}'}\int_{\mf{F}}I(\mc{L}_Ch\cap \mc{T}'(z_1^2c^2)=\emptyset)dx_1dx_2dwdz_2\]
		by \eqref{eqn-F_2(z_1)}.
		
		To make our calculations more explicit, we fix
		\[\mf{F}=\{(1,2)+r_1(1,1)+r_2(\sqrt{2},-\sqrt{2})\mid r_1,r_2\in [0,1)\}.\]
		Now let $J_\mf{F}:=(2-\sqrt{2},3)$ be the projection of $\mf{F}$ onto the $x_2$-axis, and, given $x_2\in J_\mf{F}$, let 
		\[\mf{F}_{x_2}:=\{x_1>0\mid (x_1,x_2)\in\mf{F}\}.\]
		We then have 
		\[F_2(z_1)=\frac{1}{8^{1/2}\zeta_K(2)}\int_{\mc{W}'}\int_{J_{\mf{F}}}f(z_1,w,x_2)dx_2dw\]
		where 
		\[f(z_1,w,x_2):=\int_{-z_1}^{z_1}\int_{\mf{F}_{x_2}}I(\mc{L}_Ch\cap \mc{T}'(z_1^2c^2)=\emptyset)dx_1dz_2.\]
		
		We will now show that for every $\e>0$, there exists $\delta>0$ such that if $z_1,z_1'\in[C^{-1},C]$ and $|z_1-z_1'|<\delta$, we have
		\[|f(z_1,w,x_2)-f(z_1',w,x_2)|<\e\]
		for every $w\in \mc{W}'$ and $x_2\in J_{\mf{F}}$. This implies the statement of the lemma. 
		
		Without loss of generality, we assume that $z_1\leq z_1'$. By noting that $|\mf{F}_{x_2}|< 3$ for any $x_2$ we have
		\begin{align}
			&|f(z_1,w,x_2)-f(z_1',w,x_2)|\nonumber\\
			&< 6(z_1'-z_1)+\sum_{v\in \mc{L}_C}\int_{-z_1}^{z_1}\int_{\mf{F}_{x_2}}\left|I(vh\in \mc{T}'(z_1^2c^2))-I(vh'\in \mc{T}'((z_1')^2c^2))\right|dx_1dz_2\label{eqn-lem-cont-1}
		\end{align}
		where $h'=h'(x_1,x_2,y_1',y_2,\theta_1',\theta_2)$ with $y_1'$ and $\theta_1'$ determined by \[(y_1')^{-1}(\sin\theta_1',\cos\theta_1')=(z_1',z_2).\]
		In \eqref{eqn-lem-cont-1}, the last two coordinates of $vh$ and $vh'$ agree. Hence, if $v_1,v_2$ are the first two coordinates of $v$, we have that the integrand in \eqref{eqn-lem-cont-1} is bounded everywhere by
		\begin{equation}
		\label{eqn-lem-cont-2}
		\left|I((v_1,v_2)n(x_1)a(y_1)k(\theta_1)\in T_{z_1})-I((v_1,v_2)n(x_1)a(y_1')k(\theta_1')\in T_{z_1'})\right|\end{equation}
		where, for $r>0$, we let $T_r\subset \R^2$ be the open triangle with vertices at $(0,0)$ and $r(1,\pm1)$.
		
		We observe that when $v_1=0$, the difference in \eqref{eqn-lem-cont-2} is zero. Indeed, $\inv{y_1}(\sin\theta_1,\cos\theta_1)=(z_1,z_2)$ implies that $(0,v_2)n(x_1)a(y_1)k(\theta_1)=v_2(z_1,z_2)$. We are only interested in when $|z_2|<z_1$ and then \[(0,v_2)n(x_1)a(y_1)k(\theta_1)=v_2(z_1,z_2)\in T_{z_1}\] if and only if $0<v_2<1$. Similarly, $(v_1,v_2)n(x_1)a(y_1')k(\theta_1')\in T_{z_1'}$ if and only if $0<v_2<1$.
		It follows that with $\mc{L}_C'=\{(v_1,v_2)\in \mc{L}_C\mid v_1\neq 0\}$ we have 
		\begin{align}
		|f(z_1,w,x_2)-f(z_1',w,x_2)|
		< 6(z_1'-z_1)+\sum_{v\in \mc{L}_C'}\int_{-z_1}^{z_1}\int_{\mf{F}_{x_2}}I\left((v_1,v_2)n(x_1)\in T^\triangle\right)dx_1dz_2\nonumber
		\end{align}
		where $T^\triangle=T^{\triangle}(z_1,z_1',z_2)=(T_{z_1}\inv{k(\theta_1)}\inv{a(y_1)})\triangle (T_{z_1'}\inv{k(\theta_1')}\inv{a(y_1')})$ (here $\triangle$ denotes symmetric difference).
		
		Using $\inv{y_1}(\sin\theta_1,\cos\theta_1)=(z_1,z_2)$, it is verified that the vertices of $T_{z_1}\inv{k(\theta_1)}\inv{a(y_1)}$ are $(0,0)$ and $z_1\left(z_2\pm z_1,\frac{z_1\mp z_2}{z_1^2+z_2^2}\right)$. By continuity and compactness, for every $\epsilon'>0$ there is a $\delta'>0$ such for any $z_1,z_1'\in [C^{-1},C]$ with $|z_1'-z_1|<\delta'$ and any $z_2\in [-z_1,z_1]$ we have that
		\[\left|z_1\left(z_2\pm z_1,\frac{z_1\mp z_2}{z_1^2+z_2^2}\right)-z_1'\left(z_2\pm z_1',\frac{z_1'\mp z_2}{(z_1')^2+z_2^2}\right)\right|<\epsilon'.\]
		This implies that if $|z_1'-z_1|<\delta'$, then
		$T^\triangle \subset \partial T_{z_1}\inv{k(\theta_1)}\inv{a(y_1)}+B_{\epsilon'}(0)$. Using the fact that $(v_1,v_2)n(x_1)=(v_1,v_1x_1+v_2)$ we now focus on bounding the expression
		\[\sum_{v\in \mc{L}_C'}\int_{-z_1}^{z_1}\int_{\mf{F}_{x_2}}I\left((v_1,v_1x_1+v_2)\in \partial T_{z_1}\inv{k(\theta_1)}\inv{a(y_1)}+B_{\e'}(0)\right)dx_1dz_2\]
		for a fixed $\e'>0$. This expression clearly is bounded by 
		\[\sum_{v\in \mc{L}_C'}|v_1|^{-1}\int_{-z_1}^{z_1}\int_{\R}I\left((v_1,x_1)\in \partial T_{z_1}\inv{k(\theta_1)}\inv{a(y_1)}+B_{\e'}(0)\right)dx_1dz_2.\]
		The intersection of the line $(v_1,0)+\R(0,1)$ and $\partial T_{z_1}\inv{k(\theta_1)}\inv{a(y_1)}+B_{\e'}(0)$ can only be large if $\theta_1$ is close to $\pi/4$, $3\pi/4$ (and $v_1$ is close to $0$) or if $\theta_1$ is close to $0$. However, we have $v_1\neq 0$ since $(v_1,v_2)\in \mc{L}_C'$, and moreover, since $\sin\theta_1=z_1y_1=\frac{z_1}{\sqrt{z_1^2+z_2^2}}$ is bounded away from $0$, we have that $|v_1|^{-1}\int_{-z_1}^{z_1}\int_{\R}\left|I((v_1,x_1)\in \partial T_{z_1}\inv{k(\theta_1)}\inv{a(y_1)}+B_{\e'}(0))\right|dx_1dz_2<\epsilon$ for each $v\in \mc{L}_C'$ if $\e'$ is sufficiently small. Choose such $\epsilon'$ and choose $\delta$ so that $T^\triangle \subset \partial T_{z_1}\inv{k(\theta_1)}\inv{a(y_1)}+B_{\epsilon'}(0)$ for all $z_1,z_1'\in [C^{-1},C]$ and $z_2\in [-z_1,z_1]$ with $|z_1'-z_1|<\delta$. Then
		\[|f(z_1,w,x_2)-f(z_1',w,x_2)|<6\d+\e\cdot \#\mc{L}_C'\]
		and we are done.
	\end{proof}

	Recall that $F(s)=\lim_{\eta\to 0^+}\inv{\eta}F_1(s,\eta)$ and that $F_1(s,\eta)$ is given in \eqref{Def-F_2}. By \Cref{lem-F2-cont} and the mean value theorem for integrals we get	
	\[F(s)=\frac{1}{2\d cs^{1/2}}F_2(\inv{c}s^{1/2}).\]
	Thus, by \eqref{eqn-F_2(z_1)} we have
	\[F(s)=\frac{1}{2^{5/2}\d c\zeta_K(2)s^{1/2}}\int_{z_2=-c^{-1}s^{1/2}}^{c^{-1}s^{1/2}}\int_{\mc{W}'}\int_{\mf{F}}I(\mc{L}h\cap \mc{T}'(s)=\emptyset)dx_1dx_2dwdz_2.\]
	
	Next, we fix $s>0$ and $r\in \Z$ so that $\l^{2r}\leq s^{1/4}<\l^{2(r+1)}$. Recall that $\mc{T}(s)=\mc{T}'(s)\mathrm{diag}(\l^{-2r},\l^{-2r},\l^{2r},\l^{2r})$ (see p.\  \pageref{Def-T(s)}) and note that $\mc{T}(s)$ contains a ball of radius $\gg s^{1/4}$. Since $\mathrm{diag}(\l^{-2r},\l^{-2r},\l^{2r},\l^{2r})$ commutes with every $h\in H$ and fixes $\mc{L}$ we have $\mc{L}h\cap \mc{T}(s)=\emptyset$ if and only if $\mc{L}h\cap \mc{T}'(s)=\emptyset$ for every $h\in H$. It follows that
	\begin{equation}
	\label{eqn-F(s)-1'}
	F(s)=\frac{1}{2^{5/2}\d c\zeta_K(2)s^{1/2}}\int_{z_2=-c^{-1}s^{1/2}}^{c^{-1}s^{1/2}}\int_{\mc{W}'}\int_{\mf{F}}I(\mc{L}h\cap \mc{T}(s)=\emptyset)dx_1dx_2dwdz_2
	\end{equation}
	where the $y_1,y_2,\theta_1,\theta_2$ parameters of $h=h(x_1,x_2,y_1,y_2,\theta_1,\theta_2)$
	are determined by
	\begin{equation}
		\label{eqn-y_i-theta_i-determined'}
		\inv{y_1}(\sin\theta_1,\cos\theta_1)=\l^{-2r}(c^{-1}s^{1/2},z_2),\quad \inv{y_2}(\sin\theta_2,\cos\theta_2)=\l^{2r}w.
	\end{equation}
	
	\begin{lem}
		\label{lem-size-y_1-y_2}
		If the parameters $h=h(x_1,x_2,y_1,y_2,\theta_1,\theta_2)$ are determined by \eqref{eqn-y_i-theta_i-determined'} from $z(z_2,w)$ with $|z_2|<c^{-1}s^{1/2}$ and $w\in\mc{W}'$ then $y_1,y_2\asymp s^{-1/4}$.
	\end{lem}
	
	\begin{proof}
		This result is immediate from \eqref{eqn-y_i-theta_i-determined'}, recalling the definition of $\mc{W}'$ and the fact that $\l^{2r}\asymp s^{1/4}$.
	\end{proof}
	
	Given $(z_2,w)$ with $|z_2|<c^{-1}s^{1/2}$ and $w\in \mc{W}'$ and corresponding $h=h(x_1,x_2,y_1,y_2,\theta_1,\theta_2)$ we now want to determine for which $(x_1,x_2)\in \mf{F}$ we have that the $\mc{L}h\cap\mc{T}(s)=\emptyset$. In the following lemma we show that these $(x_1,x_2)$ are confined to finitely many small boxes.
	
	\begin{lem}
		\label{lem-F-A}
		There is a constant $C_1>0$ and finite set $A\subset \mc{O}_K^2$ such that for every $(\a_3,\a_4)\in A$ we have $\a_3>0$ and $\gcd(\a_3,\a_4)=1$ and such that, for every sufficiently large $s>0$,
		\begin{equation}
		\label{eqn-F(s)-2}
		F(s)=\frac{\inv{\d}}{2^{5/2}c\z_K(2)}\sum_{(\a_3,\a_4)\in A}F_{(\a_3,\a_4)}(s)
		\end{equation}
		where \[F_{(\a_3,\a_4)}(s):=s^{-1/2}\int_{|z_2|<c^{-1}s^{1/2}}\int_{\mc{W}'}\int_{B(\a_3,\a_4)}I(\mc{L}h\cap \mc{T}(s)=\emptyset)\,dx_1\,dx_2\,dw\,dz_2\]
		and \[B(\a_3,\a_4):=\{(x_1,x_2):|x_1+\a_4/\a_3|,|x_2+\s(\a_4/\a_3)|\leq C_1s^{-1/2}\}.\]
	\end{lem}
	
	\begin{proof}
		Fix $h$ corresponding to $(z_2,w)$ and $(x_1,x_2)\in \mf{F}$ for which the integrand in \eqref{eqn-F(s)-1'} is non-zero. Fix $t>0$ so that $\Gamma_K\mf{D}_t=H$.
		There is some $\gamma\in \Gamma_K$ so that $h_0:=\gamma h\in \mf{D}_t$. Since $\mc{T}(s)$ contains a ball of radius $\gg s^{1/4}$, \Cref{lem-sizey1y2} implies that the $y$-Iwasawa parameters of $h_0$ are $\gg s^{1/4}$. By using the fact that the $y$-Iwasawa parameter of $\begin{pmatrix} a & b\\ c & d
		\end{pmatrix}$ is given by $(c^2+d^2)^{-1/2}$, we find that
		\begin{equation}
		\label{eqn-a_3_a_4-1}
		((\a_3y_1)^2+(\a_3x_1+\a_4)^2y_1^{-2})^{-1/2}\gg s^{1/4},\quad ((\s(\a_3)y_2)^2+(\s(\a_3)x_2+\s(\a_4))^2y_2^{-2})^{-1/2}\gg s^{1/4},
		\end{equation}
		where
		$\gamma=\mathrm{diag}\left(\begin{pmatrix}\a_1 & \a_2\\ \a_3 & \a_4
		\end{pmatrix},\sigma\begin{pmatrix}\a_1 & \a_2\\ \a_3 & \a_4
		\end{pmatrix}\right)$. Note that we may replace $\gamma$ by $-\gamma$ so as to make $\a_3$ non-negative while still achieving the bounds above. In particular, $|\a_3|y_1\ll s^{-1/4}$ and $|\s(\a_3)|y_2\ll s^{-1/4}$ which implies that $|\a_3|,|\s(\a_3)|\ll 1$. Thus, there are only finitely many possibilities for $\a_3$. We note that $\a_3$ cannot be $0$: Indeed, if $\a_3=0$ then $\a_4\neq 0$. However, \eqref{eqn-a_3_a_4-1} then implies that $|\a_4|^{-1}\gg s^{1/2}$, $|\s(\a_4)|^{-1}\gg s^{1/2}$. By taking $s$ sufficiently large, this is seen to be impossible. Thus, we will henceforth assume that $\a_3>0$. Now, \eqref{eqn-a_3_a_4-1} also tells us that $|x_1+\a_4/\a_3|,|x_2+\s(\a_4/\a_3)|\leq C_1s^{-1/2}$ for some $C_1>0$. Thus, $(x_1,x_2)$ has to live in the box \[\label{Def-B(a3,a4)}B(\a_3,\a_4)=\{(x_1,x_2):|x_1+\a_4/\a_3|,|x_2+\s(\a_4/\a_3)|\leq C_1s^{-1/2}\}.\] Next, we verify that there are only finitely many possibilities for $\a_4$. Indeed, fix $\a_3$. We are only interested in the cases when the box corresponding to $-(\a_4/\a_3,\s(\a_4/\a_3))$ has non-empty intersection with $\mf{F}$. It follows that $|\a_4|,|\s(\a_4)|$ must be bounded by a constant (that depends on $\a_3$), hence there are only finitely many possibilities for $\a_4$. We may discard some $\a_3,\a_4$ if necessary to obtain a finite set $A$ of pairs $(\a_3,\a_4)$ such that all quotients $\a_4/\a_3$ are distinct and so that for sufficiently large $s$ the boxes $B(\a_3,\a_4)$ are pairwise disjoint. We may also take the constant $C_1$ above to be independent of $\a_3,\a_4$.
	\end{proof}
	
	\subsection{Asymptotics of $F_{(\a_3,\a_4)}(s)$}
	\label{sec-contr-a3-a4}
	
	Let $A':=A\cup \{(\a_3,-\a_4)\mid (\a_3,\a_4)\in A\}$.
	The goal of this section is to prove \Cref{prop-a_3-a_4-contr-decay}, which states that for each $(\a_3,\a_4)\in A'$ we have $F_{(\a_3,\a_4)}(s)=c_{(\a_3,\a_4)}s^{-1/2}+\mc{O}(s^{-17/8})$ for some $c_{(\a_3,\a_4)}>0$. This implies that $F(s)=C_\mc{P}s^{-2}+\mc{O}(s^{-17/8})$, which is what we conclude in \Cref{thm-F(s)-asymp-decay}.
	
	Given $(\a_3,\a_4)\in A'$ and $s>0$, let $F_{(\a_3,\a_4)}^+(s)$ be the part of $F_{(\a_3,\a_4)}(s)$ that corresponds to integrating over $z_2>0$ and let $F_{(\a_3,\a_4)}^-(s)$ be the part corresponding to $z_2<0$. We first make the following observation.
	
	\begin{lem}
		\label{lem-wlog-z_2>0}
		For $(\a_3,\a_4)\in A'$ and $s>0$ we have $F_{(\a_3,\a_4)}^-(s)=F_{(\a_3,-\a_4)}^+(s)$.
	\end{lem}
	
	\begin{proof}	
		Recall that $h=\mathrm{diag}(h_1,h_2)$ in $F_{(\a_3,\a_4)}(s)$ is given by $h_i=n(x_i)a(y_i)k(\theta_i)$ where $y_i$ and $\theta_i$ are determined by
		\[\inv{y_1}(\sin\theta_1,\cos\theta_1)=\l^{-2r}(c^{-1}s^{1/2},z_2),\quad\inv{y_2}(\sin\theta_2,\cos\theta_2)=\l^{2r}w\]
		(cf.\ \eqref{eqn-y_i-theta_i-determined'}).
		Make the change of variables $z_2'=-z_2$ in $F_{(\a_3,\a_4)}^-(s)$ to obtain 
		\[F_{(\a_3,\a_4)}^-(s)=s^{-1/2}\int_{0}^{c^{-1}s^{1/2}}\int_{\mc{W}'}\int_{B(\a_3,\a_4)}I(\mc{L}h\cap \mc{T}(s)=\emptyset)dx_1dx_2dwdz_2'.\]
		Now we note that by setting $\theta_1'=\pi-\theta_1$ we have that $\inv{y_1}(\sin\theta_1',\cos\theta_1')=\l^{-2r}(c^{-1}s^{1/2},z_2')$, i.e.\ that $y_1$ and $\theta_1'$ are determined by $z_2'$ as $y_1,\theta_1$ are determined by $z_2$. Note also that $k(\theta_1)=-k(-\theta_1')$.
		
		Let $S:=\mathrm{diag}(1,-1)$, $D_1:=\mathrm{diag}(S,S)$, $D_2:=\mathrm{diag}(S,1,1)$ and $D_3:=\mathrm{diag}(1,1,-1,-1)$ and note that $\mc{T}(s)D_i=\mc{T}(s)$ for $i\in\{1,2,3\}$ and $\mc{L}D_1=\mc{L}$. We also note that $a(y_1)S=Sa(y_1)$, $n(x_1)S=Sn(-x_1)$ and $k(\theta_1)S=Sk(-\theta_1)$. Hence
		\[n(x_1)a(y_1)k(\theta_1)S=Sn(-x_1)a(y_1)k(-\theta_1).\]
		By using these facts, the bijection $B(\a_3,\a_4)\longrightarrow B(\a_3,-\a_4)$ given by $(x_1,x_2)\mapsto -(x_1,x_2)$ and that $-\mc{L}=\mc{L}$ we find 
		\begin{align*}
		F_{(\a_3,\a_4)}^-(s)
		=s^{-1/2}\int_{0}^{c^{-1}s^{1/2}}\int_{\mc{W}'}\int_{B(\a_3,-\a_4)}I\left(\mc{L}h'\cap \mc{T}(s)=\emptyset\right)dx_1dx_2dwdz_2'=F_{(\a_3,-\a_4)}^+(s),
		\end{align*}
		where $h'=h'(x_1,x_2,y_1,y_2,\theta_1',\theta_2)=\mathrm{diag}(n(x_1)a(y_1)k(\theta_1'),n(x_2)a(y_2)k(-\theta_2))$.
	\end{proof}
	
	In view of \Cref{lem-wlog-z_2>0} it suffices to show that $F_{(\a_3,\a_4)}^+(s)=c_{(\a_3,\a_4)}^+s^{-1/2}+\mc{O}(s^{-17/8})$ as $s\to\infty$ for some $c_{(\a_3,\a_4)}^+>0$ in order to conclude that $F(s)=C_{\mc{P}}s^{-2}+\mc{O}(s^{-17/8})$. To this end, fix $(\a_3,\a_4)\in A'$, $s>0$ and let \[d(s):=\inv{c}\l^{-2r}s^{1/2}\label{Def-d(s)}.\] Make the change of variables $z_2'=d(s)-\l^{-2r}z_2$. We then have
	\begin{equation}
	\label{eqn-contr-z>0}
	F_{(\a_3,\a_4)}^+(s)=
	\frac{\l^{2r}}{s^{1/2}}\int_{0}^{d(s)}\int_{\mc{W}'}\int_{B(\a_3,\a_4)}I(\mc{L}h\cap \mc{T}(s)=\emptyset)\,dx_1\,dx_2\,dw\,dz_2',\end{equation}
	where $h=h(x_1,x_2,y_1,y_2,\theta_1,\theta_2)$, but with $y_1,y_2,\theta_1,\theta_2$ determined by  \begin{equation}\label{eqn-y_i-theta_i-determined-2}\inv{y_1}(\sin\theta_1,\cos\theta_1)=(d(s),d(s)-z_2'),\quad \inv{y_2}(\sin\theta_2,\cos\theta_2)=\l^{2r}w.\end{equation} 
	Equivalently, we have $h=\mathrm{diag}(h_1,h_2)$ where 
	\begin{equation}
	\label{eqn-h_1-h_2}
	h_1=h_1(z_2',x_1):=n(x_1)a(y_1)k(\theta_1),\quad h_2=h_2(w,x_2):=n(x_2)a(y_2)k(\theta_2)
	\end{equation}
	and $y_1,\theta_1$ are determined by $z_2'$ according to the first relation in \eqref{eqn-y_i-theta_i-determined-2} and $y_2,\theta_2$ are determined by $w$ from the second relation in \eqref{eqn-y_i-theta_i-determined-2}. Given $z_2'\in (0,d(s))$ and $w\in \mc{W}'$, let 
	\[(z^1,z^2)=(z^1(z_2'),z^2(w)):=(d(s),d(s)-z_2',\l^{2r}w).\]
	
	It will turn out that it is convenient to write the condition $\mc{L}h\cap \mc{T}(s)=\emptyset$ as a conjunction of two other conditions and study those separately. To state these conditions, we introduce some notation.
	
	Let \[v_1:=(\a_3,\a_4,\s(\a_3),\s(\a_4)),\quad v_2:=(\sqrt{2}\a_3,\sqrt{2}\a_4,\s(\sqrt{2}\a_3),\s(\sqrt{2}\a_4))\in\mc{L}\label{Def-v1-v2}.\]
	For $v\in \mc{L}$, define also
	\[\mc{L}_v:=v+\Z v_1+\Z v_2\label{Def-L_v}\]
	which is a two dimensional subgrid of $\mc{L}$. Note that two subgrids $\mc{L}_v$ and $\mc{L}_{v'}$ are equal if and only if $v'\in \mc{L}_v$ and otherwise they are disjoint. We now express the condition $\mc{L}h\cap \mc{T}(s)=\emptyset$ as the conjunction of $\mc{L}_{(0,1,0,1)}h\cap \mc{T}(s)=\emptyset$ and $\mc{L}_{v}h\cap \mc{T}(s)=\emptyset$ for all $v\in \mc{L}\setminus \mc{L}_{(0,1,0,1)}$. More precisely
	\begin{equation}
	\label{eqn-F^+-I_1-I_2}
	F_{(\a_3,\a_4)}^+(s)=
	\frac{\l^{2r}}{s^{1/2}}\int_{0}^{d(s)}\int_{\mc{W}'}\int_{B(\a_3,\a_4)}\mc{I}_1(h)\mc{I}_2(h)\,dx_1\,dx_2\,dw\,dz_2',
	\end{equation}
	where \[\mc{I}_1(h):=I(\mc{L}_{(0,1,0,1)}h\cap \mc{T}(s)=\emptyset)\label{Def-mcI1}\]
	and \[\mc{I}_2(h):=I(\mc{L}_{v}h\cap \mc{T}(s)=\emptyset\text{ for all } v\in \mc{L}\setminus \mc{L}_{(0,1,0,1)})\label{Def-mcI2}.\]
	
	Note that $\mc{L}\cap (\R v_1+\R v_2)=\Z v_1+\Z v_2$ since $\a_3$ and $\a_4$ are relatively prime in $\mc{O}_K$. Now let 
	\begin{equation}
	\label{eqn-b_1-b_2}	
	b_1:=2^{-1}(v_1+2^{-1/2}v_2)=(\a_3,\a_4,0,0),\quad b_2:=2^{-1}(v_1-2^{-1/2}v_2)=(0,0,\s(\a_3),\s(\a_4));
	\end{equation}
	these vectors form a basis of $\R v_1+\R v_2$ and we have
	\begin{equation}
	\label{eqn-connection-v_i-b_i}
	\Z v_1+\Z v_2=\Z(b_1+b_2)+\Z\sqrt{2}(b_1-b_2)=\{\a b_1+\s(\a) b_2\mid \a\in \mc{O}_K\}.
	\end{equation}
	Note that $b_1,b_2$ have the first two and last two coordinates equal to $0$, respectively. Therefore, we will frequently consider those vectors as vectors in $\R^2$. Next we note that the subgrids $\mc{L}_v$ are small.
	
	\begin{lem}
		\label{lem-size-v_i-b_i}
		For any $z_2'\in (0,d(s))$, $w\in \mc{W}'$ and $(x_1,x_2)\in B(\a_3,\a_4)$ we have that $\norm{v_1h}$, $\norm{v_2h}$, $\norm{b_1h}$ and $\norm{b_2h}$ are all $\asymp s^{-1/4}$.
	\end{lem}
	
	\begin{proof}
		It suffices to show that $\norm{v_1h},\norm{v_2h}\asymp s^{-1/4}$ as $b_1h,b_2h$ are well-behaved linear combinations of these vectors. Note that \[v_1h=(\a_3y_1,(\a_3x_1+\a_4)\inv{y_1},\s(\a_3)y_2,(\s(\a_3)x_2+\s(\a_4))\inv{y_2})\mathrm{diag}(k(\theta_1),k(\theta_2)).\]
		By \Cref{lem-size-y_1-y_2}, the definition of $B(\a_3,\a_4)$ and the fact that $\a_3\neq 0$ it follows that $\norm{v_1h}\asymp s^{-1/4}$. To show that $\norm{v_2h}\asymp s^{-1/4}$ is completely analogous.
	\end{proof}
	
	To study the condition $\mc{L}h\cap \mc{T}(s)=\emptyset$ we now study the condition $\mc{L}_vh\cap \mc{T}(s)=\emptyset$ for various $v$.
	
	\subsubsection{The condition $\mc{L}_{(0,1,0,1)}h\cap \mc{T}(s)=\emptyset$}
	
	We now focus on the condition $\mc{L}_{(0,1,0,1)}h\cap \mc{T}(s)=\emptyset$. Recall that $(0,1,0,1)h=(z^1,z^2)$. By \eqref{eqn-connection-v_i-b_i} we have that $\mc{L}_{(0,1,0,1)}h\cap \mc{T}(s)=\emptyset$ is equivalent with $(z^1,z^2)+\a b_1h_1+\s(\a)b_2h_2\notin \mc{T}(s)$ for every $\a\in \mc{O}_K$.
	
	\begin{lem}
		\label{lem-charac-dense-lattice-empty-intersection}
		For any $z_2'\in (0,d(s))$, $w\in \mc{W}'$ and $(x_1,x_2)\in B(\a_3,\a_4)$ define 
		\[I_1=I_1(z_2',x_1):=\{x\in \R\mid z^1+xb_1h_1\in \l^{-2r} T(s)\}\label{Def-I_1-1}\] and \[I_2=I_2(w,x_2):=\{x\in \R\mid z^2+xb_2h_2\in\l^{2r}\mc{W}\}\label{Def-I_2}\]
		where $h_1=h_1(z_2',x_1)$ and $h_2=h_2(w,x_2)$; cf.\ \eqref{eqn-h_1-h_2}.
		It then holds that $\mc{L}_{(0,1,0,1)}h\cap \mc{T}(s)=\emptyset$ if and only if $\mc{L}'\cap (I_1\times I_2)=\emptyset$
	\end{lem}
	
	\begin{proof}
		By using $\mc{L}_{(0,1,0,1)}=(0,1,0,1)+\Z v_1+\Z v_2$, $\mc{T}(s)=\l^{-2r}T(s)\times \l^{2r}\mc{W}$ and \eqref{eqn-connection-v_i-b_i} we have that $\mc{L}_{(0,1,0,1)}h\cap \mc{T}(s)=\emptyset$ if and only if there is no $\a\in \mc{O}_K$ such that 
		\[(0,1)h_1+\a b_1h_1\in \l^{-2r}T(s)\text{ and }(0,1)h_2+\s(\a) b_2h_2\in \l^{2r}\mc{W}.\]
		Now we note that $(0,1)h_1=z^1$ and $(0,1)h_2=z^2$ to conclude the desired result.
	\end{proof}
	
	Recall that we will prove that the order of decay of $F(s)$ as $s\to\infty$ is $s^{-2}$. Thus, any contribution which decays faster than this will be considered small, or \textit{asymptotically negligible}. The following lemma shows that the error introduced by only integrating over $z_2'\in (0,s^{-5/8})$ in $F^{+}_{(\a_3,\a_4)}$ (cf.\ \eqref{eqn-contr-z>0}) is asymptotically negligible.
	
	\begin{lem}
		\label{lem-z_2-large-asymp-neg}
		For all large $s$ we have 
		\begin{equation}
		\label{eqn-contr-1-z>0}
		F^+_{(\a_3,\a_4)}(s)=\frac{\l^{2r}}{s^{1/2}}\int_{0}^{s^{-5/8}}\int_{\mc{W}'}\int_{B(\a_3,\a_4)}I(\mc{L}'\cap (I_1\times I_2)=\emptyset)\mc{I}_2(h)\,dx_1\,dx_2\,dw\,dz_2'+\mc{O}(s^{-17/8}).\end{equation} 
	\end{lem}
	
	\begin{proof}
		We use \Cref{lem-charac-dense-lattice-empty-intersection} to replace $\mc{I}_1(h)$ by $I(\mc{L}'\cap (I_1\times I_2)=\emptyset)$ in $F_{(\a_3,\a_4)}^+(s)$ in \eqref{eqn-F^+-I_1-I_2}.
		We note that unless $b_1h_1$ is vertical, we have $|I_1|\gg z_2's^{1/4}$, since $\norm{b_1h_1}\asymp s^{-1/4}$ by \Cref{lem-size-v_i-b_i}. It is readily verified that $b_1h_1$ is vertical for a set $(z_2',w,x_1,x_2)$ of measure $0$; thus we assume that $b_1h_1$ is not vertical, and hence that $|I_1|\gg z_2's^{1/4}$. 
		Given $w\in \mc{W}$, let
		\begin{equation}
			\label{eqn-delta-w}\delta(w):=\inf_{\varphi\in \R/2\pi\Z}|\{t\in\R: w+t(\cos\varphi,\sin\varphi)\in \mc{W}\}|.
		\end{equation}
		Since $\norm{b_2h_2}\asymp s^{-1/4}$ by \Cref{lem-size-v_i-b_i}, we have $|I_2|\gg \l^{2r}s^{1/4}\d(w)$. Since $\mc{L}'\cap (I_1\times I_2)=\emptyset$ can only hold if $|I_1|\cdot |I_2|<C_2$ by \Cref{lem-L'} we need $z_2'\delta(w)\leq C_3\l^{-2r}s^{-1/2}$ for some $C_3>0$. For fixed $z_2'$, let \[\label{Def-W(z)}\mc{W}(z_2'):=\{w\in \mc{W}\mid z_2'\delta(w)\leq C_3\l^{-2r}s^{-1/2}\}.\] 
		
		We have
		\[\frac{\l^{2r}}{s^{1/2}}\int_{s^{-5/8}}^{d(s)}\int_{\mc{W}'}\int_{B(\a_3,\a_4)}I(\mc{L}'\cap (I_1\times I_2)=\emptyset)\mc{I}_2(h)\,dx_1\,dx_2\,dw\,dz_2'\ll s^{-5/4}\int_{s^{-5/8}}^{d(s)}\int_{\mc{W}'(z_2')}\,dw\,dz_2',\]
		using the fact that $\mathrm{area}(B(\a_3,\a_4))\ll s^{-1}$.  Since $z_2'>s^{-5/8}$ we see that $w\in \mc{W}'(z_2')$ implies $\delta(w)\leq C_3\l^{-2r}s^{-1/2}/z_2'\ll s^{-1/8}$, from which it follows that such $w$ has to be close to a vertex at the boundary of $\mc{W}$. It thus follows that $\mathrm{area}(\mc{W}'(z_2'))\ll (\l^{-2r}s^{-1/2}/z_2')^2\ll s^{-3/2}/(z_2')^2$ so the above integral is
		\[\ll s^{-11/4}\int_{z_2'=s^{-5/8}}^{d(s)} (z_2')^{-2}\,dz_2'\ll s^{-17/8}.\]
	\end{proof}
	
	Now note that $I_1=\{x\in \R\mid xb_1h_1\in \l^{-2r}T(s)-z^1\}$, where $\l^{-2r}T(s)-z^1$ is the sector \[S(z_2'):=\{(z_1,z_2)\in \R^2\mid z_1<0, z_2<z_1+z_2'\}\label{Def-S}\] intersected with a half plane whose boundary is of distance $\gg s^{1/4}$ to the origin.
	
	\begin{lem}
		\label{lem-change-to-S-asymp-neg}
		For all large $s$ we have
		\begin{equation*}
		F^+_{(\a_3,\a_4)}(s)=\frac{\l^{2r}}{s^{1/2}}\int_{0}^{s^{-5/8}}\int_{\mc{W}'}\int_{B(\a_3,\a_4)}I(\mc{L}'\cap (I_1'\times I_2)=\emptyset)\mc{I}_2(h)\,dx_1\,dx_2\,dw\,dz_2'+\mc{O}(s^{-17/8})\end{equation*}
		where $I_2=I_2(w,x_2)$ is as in \Cref{lem-charac-dense-lattice-empty-intersection} and 
		\[I_1'=I_1'(z_2',x_2):=\{x\in \R \mid xb_1h_1\in S(z_2')\}.\] 
	\end{lem}
	
	\begin{proof}
		Assume that for a fixed tuple $(z_2',w,x_1,x_2)$ appearing in the integral \eqref{eqn-contr-1-z>0} we have 
		\[I(\mc{L}'\cap (I_1\times I_2)=\emptyset)\neq I(\mc{L}'\cap (I_1'\times I_2)=\emptyset).\]
		Since $I_1\subset I_1'$ the only possbility is that $\mc{L}'\cap (I_1\times I_2)=\emptyset$ and $\mc{L}'\cap (I_1'\times I_2)\neq\emptyset$. Thus, the intersections of the line $\R b_1h_1$ with $\l^{-2r}T(s)-z^1$ and $S(z_2')$ are distinct. This implies that the intersection of the line $\R b_1h_1$ with $\l^{-2r}T(s)-z^1$ has length $\gg s^{1/4}$. Recalling that $\norm{b_1h_1}\ll s^{1/4}$ (cf.\ \Cref{lem-size-v_i-b_i}) we conclude that $|I_1|\gg s^{1/2}$. 
		Combining this fact with $\mc{L}'\cap (I_1\times I_2)=\emptyset$ and \Cref{lem-L'} we conclude that $|I_2|\ll s^{-1/2}$.
		
		Define again, as in the proof of \Cref{lem-charac-dense-lattice-empty-intersection}, 
		\[\delta(w)=\inf_{\varphi\in \R/2\pi\Z}|\{t\in\R: w+t(\cos\varphi,\sin\varphi)\in \mc{W}\}|\]
		and conclude as in said proof that $\l^{2r}s^{1/4}\delta(w)\ll|I_2|\ll s^{-1/2}$ which gives $\delta(w)\ll s^{-1}$. This implies that the set of  $(z_2',w,x_1,x_2)$ such that \[I(\mc{L}'\cap (I_1\times I_2)=\emptyset)\neq I(\mc{L}'\cap (I_1'\times I_2)=\emptyset)\]
		have $w$ confined to a subset of $\mc{W}$ of total area $\ll s^{-2}$. It follows that the difference of the integral in the statement of the lemma and the integral appearing in \Cref{lem-z_2-large-asymp-neg} is $\ll \frac{\l^{2r}}{s^{1/2}}s^{-5/8}s^{-2}s^{-1}=s^{-31/8}$ which gives the desired result.
	\end{proof}
	
	Define now $x_1',x_2'\in [-C_1,C_1]$ by
	\[x_1=s^{-1/2}x_1'-\a_4/\a_3,\quad x_2=s^{-1/2}x_2'-\s(\a_4/\a_3).\]
	
	\begin{lem}
		\label{lem-mcI3}
		Let $\mc{I}_3=\mc{I}_3(x_1'):=I\left(x_1'>-\frac{c^2\l^{4r}}{2s^{1/2}}\right)\label{Def-mcI3}$. Then for all large $s$ we have
		\begin{align*}
		&F^+_{(\a_3,\a_4)}(s)\\
		&=\frac{\l^{2r}}{s^{1/2}}\int_{0}^{s^{-5/8}}\int_{\mc{W}'}\int_{B(\a_3,\a_4)}I(\mc{L}'\cap (I_1'\times I_2)=\emptyset)\mc{I}_2(h)\mc{I}_3\,dx_1\,dx_2\,dw\,dz_2'+\mc{O}(s^{-17/8}).\end{align*}
	\end{lem}
	
	\begin{proof}
		Note that 
		\[b_1h_1=(\a_3(y_1\cos\theta_1+s^{-1/2}\inv{y_1}x_1'\sin\theta_1),\a_3(-y_1\sin\theta_1+s^{-1/2}\inv{y_1}x_1'\cos\theta_1)).\]
		From this it is seen that $b_1h_1$ is parallel with the $y$-axis if and only if $x_1'=-s^{1/2}y_1^2\cot\theta_1$. Assume that $x_1'<-s^{1/2}y_1^2\cot\theta_1$. Recalling that $\a_3>0$ we see that this implies that the first coordinate of $b_1h_1$ is negative. By the definition of $I_1'$ we find that
		\[I_1'=\{x>0\mid x\a_3(-y_1(\sin\theta_1+\cos\theta_1)+s^{-1/2}x_1'\inv{y_1}(\cos\theta_1-\sin\theta_1))<z_2'\}.\]
		However, from \eqref{eqn-y_i-theta_i-determined-2}, we know that $\theta_1$ is close to $\pi/4$ for large $s$ and also that $y_1\asymp s^{-1/4}$. Thus, $-y_1(\sin\theta_1+\cos\theta_1)+s^{-1/2}x_1'\inv{y_1}(\cos\theta_1-\sin\theta_1)$ is negative for large $s$ which implies that $I_1'=\R_{>0}$. Since $I_2=I_2(w,x_2)$ contains an open interval for almost all $w\in \mc{W}$ we have $\mc{L}'\cap (I_1'\times I_2)\neq \emptyset$ using \Cref{lem-L'}. Thus
		\begin{align*}
		&F^+_{(\a_3,\a_4)}(s)\\
		&=\frac{\l^{2r}}{s^{1/2}}\int_{0}^{s^{-5/8}}\int_{\mc{W}'}\int_{B(\a_3,\a_4)}I(\mc{L}'\cap (I_1'\times I_2)=\emptyset)\mc{I}_2(h)\mc{I}_3'\,dx_1\,dx_2\,dw\,dz_2'+\mc{O}(s^{-17/8}).\end{align*}
		where 
		\[\mc{I}_3'=\mc{I}_3'(z_2',x_1')=I(x_1'>-s^{1/2}y_1^2\cot\theta_1).\]
		Now we note that for $0<z_2'<s^{-5/8}$ we have 
		\[y_1^2\cot\theta_1=\frac{d(s)-z_2'}{d(s)(d(s)^2+(d(s)-z_2')^2)}=\frac{c^2\l^{4r}}{2s}(1+\mc{O}(s^{-7/8})).\]
		in view of \eqref{eqn-y_i-theta_i-determined-2}. It follows that if $\mc{I}_3(x_1')\neq \mc{I}_3'(x_1')$, then $x_1'$ belongs to an interval of length $\ll s^{-7/8}$ around $-\frac{c^2\l^{4r}}{2s^{1/2}}$. Thus, the claim of the lemma follows.
	\end{proof}
	
	\begin{lem}
		\label{lem-I_1-justification}
		For all large $s$ we have
		\begin{align*}
		&F^+_{(\a_3,\a_4)}(s)\\
		&=\frac{\l^{2r}}{s^{1/2}}\int_{0}^{s^{-5/8}}\int_{\mc{W}'}\int_{B(\a_3,\a_4)}\mc{I}_1'(h)\mc{I}_2(h)\mc{I}_3\,dx_1\,dx_2\,dw\,dz_2'+\mc{O}(s^{-17/8}).\end{align*}
		where $\mc{I}_1'(h):=I(\mc{L}'\cap (\widetilde{I_1}'\times \widetilde{I_2})=\emptyset)$, $\widetilde{I_1}'=\widetilde{I_1}'(z_2'):=\left(-\frac{\l^{2r}s^{1/2}z_2'}{c\a_3},0\right)$
		and $\widetilde{I_2}:=\l^{-4r}I_2$.
	\end{lem}
	
	\begin{proof}
		In view of \Cref{lem-mcI3} we may assume that $x_1'>-s^{1/2}y_1^2\cot\theta_1$ so that the first coordinate of \[b_1h_1=(\a_3(y_1\cos\theta_1+s^{-1/2}\inv{y_1}x_1'\sin\theta_1),\a_3(-y_1\sin\theta_1+s^{-1/2}\inv{y_1}x_1'\cos\theta_1))\]
		is positive. Recalling the definition of $I_1'$ (cf.\ \Cref{lem-change-to-S-asymp-neg}) one verifies that 
		\[I_1'=\{x<0\mid x\a_3(-y_1(\sin\theta_1+\cos\theta_1)+s^{-1/2}x_1'\inv{y_1}(\cos\theta_1-\sin\theta_1))<z_2'\}\]
		that is
		\[I_1'=\left(\frac{z_2'}{\a_3(-y_1(\sin\theta_1+\cos\theta_1)+s^{-1/2}x_1'\inv{y_1}(\cos\theta_1-\sin\theta_1))},0\right)\]
		since $-y_1(\sin\theta_1+\cos\theta_1)+s^{-1/2}x_1'\inv{y_1}(\cos\theta_1-\sin\theta_1)<0$ for large $s$. Since $\theta_1$ is close to $\pi/4$ for large $s$ (cf.\ \eqref{eqn-y_i-theta_i-determined-2}) we have that $I_1'$ is approximately equal to $\left(-\frac{z_2'}{\sqrt{2}\a_3y_1},0\right)$ which, by \eqref{eqn-y_i-theta_i-determined-2}, is approximately equal to $\l^{-4r}\widetilde{I_1}'$.
		
		From \eqref{eqn-y_i-theta_i-determined-2} we find that $\cot\theta_1=1+\mc{O}(s^{-1/4}z_2')$ and hence $\theta_1=\frac{\pi}{4}+\mc{O}(s^{-1/4}z_2')$. It follows that $\sin\theta_1-\cos\theta_1=\mc{O}(s^{-1/4}z_2')$ and $\cos\theta_1+\sin\theta_1=\sqrt{2}+\mc{O}(s^{-1/4}z_2')$. Recalling that $y_1\asymp s^{-1/4}$ and $|x_1'|\leq C_1$ we also have 
		\begin{align*}
		-y_1(\sin\theta_1+\cos\theta_1)+s^{-1/2}\inv{y_1}x_1'(\cos\theta_1-\sin\theta_1)&=-\sqrt{2}y_1(1+\mc{O}(s^{-1/4}z_2'))+\mc{O}(s^{-1/2}z_2')\\
		&=-\sqrt{2}y_1(1+\mc{O}(s^{-1/4}z_2')).
		\end{align*}
		Recall that $d(s)=\l^{-2r}s^{1/2}\inv{c}\asymp s^{1/4}$. We have
		\begin{equation}
		\label{eqn-inv-y_1}
		\inv{y_1}=\sqrt{2}d(s)(1+\mc{O}(s^{-1/4}z_2')).
		\end{equation}
		
		By assuming $s$ to be sufficiently large, we have $(1+\mc{O}(s^{-1/4}z_2'))^{-1}=1+\mc{O}(s^{-1/4}z_2')$. Thus, the left end point of $I_1'$ is equal to $-\frac{s^{1/2}z_2'}{c\a_3\l^{2r}}(1+\mc{O}(s^{-1/4}z_2'))$. Assume now that precisely one of $\mc{L}'\cap (I_1'\times I_2)=\emptyset$ and $\mc{L}'\cap (\l^{-4r}\widetilde{I_1}'\times I_2)=\emptyset$ holds. This is equivalent with the existence of $\gamma\in \mc{O}_K$ such that $\gamma\in I_1'\triangle \l^{-4r}\widetilde{I_1}'$ and $\s(\gamma)\in I_2$. Note that $0\in I_2$ and that $|I_2|\ll s^{1/2}$. By using the fact that $\mc{L}'$ is invariant under $\mathrm{diag}(u,\sigma(u))$ for $u\in \mc{O}_K^*$ we see that this is also equivalent with the existence of $\gamma\in \mc{O}_K$ so that $\gamma \in \l^{4r}I_1'\triangle \widetilde{I_1}'$ and $\sigma(\gamma)\in \widetilde{I_2}$, where $|\widetilde{I_2}|\ll 1$. Note that $\gamma<0$ and that there exists a constant $c_1$ so that $|\s(\g)|\leq c_1$. Furthermore, $\gamma$ must belong to the interval between the left end points of $\l^{4r}I_1'$ and $\widetilde{I_1}'$, i.e.\ $\gamma$ must lie in between $-\frac{s^{1/2}\l^{2r}z_2'}{c\a_3}(1+\mc{O}(s^{-1/4}z_2'))$ and $-\frac{s^{1/2}\l^{2r}z_2'}{c\a_3}$. It follows that $|\gamma|\asymp s^{3/4}z_2'$, which, in particular, implies that $|\gamma|\leq c_2s^{1/8}$ for some absolute constant $c_2$ (recall that $z_2'<s^{-5/8}$). We also have that $-\frac{s^{1/2}\l^{2r}z_2'}{c\a_3}$ lies in between $\gamma$ and $\gamma(1+\mc{O}(s^{-1/4}z_2'))=\gamma(1+\mc{O}(s^{-1}|\gamma|))$. We conclude that there is an absolute constant $c_3$ so that
		\[z_2'\in I_\gamma:=-c\a_3s^{-1/2}\l^{-2r}\gamma(1-c_3s^{-1}|\gamma|,1+c_3s^{-1}|\gamma|).\]
		Note that $|I_\gamma|\asymp s^{-7/4}|\g|^2$.
		Let now \[M_s=\{\gamma\in \mc{O}_K\mid \gamma<0, |\gamma|\leq c_2s^{1/8}, |\s(\g)|\leq c_1, I_\gamma \cap (0,s^{-5/8})\neq\emptyset\}.\] Thus, the difference between the integral in the statement of the lemma and the integral appearing in the statement of \Cref{lem-mcI3} is
		\begin{equation}
		\label{eqn-error-1}
		\ll s^{-1}s^{-1/4}\sum_{\g\in M_s}\int_{I_\gamma}\int_{\mc{W}(z_2')}\,dw\,dz_2',
		\end{equation}
		where the factor $s^{-1}$ comes from the measure of $B(\a_3,\a_4)$. 
		
		Note that $|\gamma|^{-1}\asymp s^{-3/4}\inv{(z_2')}$. Recall that
		$\mc{W}(z_2')=\{w\in \mc{W}\mid z_2'\delta(w)\leq C_3\l^{-2r}s^{-1/2}\}$
		hence $\mc{W}(z_2')=\{w\in \mc{W}\mid \delta(w)\ll|\gamma|^{-1}\}$ and thus $\mathrm{area}(\mc{W}(z_2'))\ll |\g|^{-2}$.	
		Therefore, the expression in \eqref{eqn-error-1} is 
		\[\ll s^{-5/4}\sum_{\gamma\in M_s}s^{-7/4}|\g|^2|\g|^{-2}=s^{-3}|M_s|.\]
		Finally, note that if $\gamma\in M_s$, then $(\gamma,\s(\gamma))$ belongs to a rectangle of area $\ll s^{1/8}$. Such a rectangle contains $\ll s^{1/8}$ points from $\mc{L}'$ and thus $|M_s|\ll s^{1/8}$.
	\end{proof}
	
	\subsubsection{The condition $\mc{L}_{v}h\cap \mc{T}(s)=\emptyset\text{ for all } v\in \mc{L}\setminus \mc{L}_{(0,1,0,1)}$}
	
	We continue our study of the expression for $F_{(\a_3,\a_4)}^+(s)$ given in \Cref{lem-I_1-justification} by studying the condition $\mc{L}_{v}h\cap \mc{T}(s)=\emptyset\text{ for all } v\in \mc{L}\setminus \mc{L}_{(0,1,0,1)}$, which is encoded in the characteristic function $\mc{I}_2(h)$.
	
	Given $v\in \mc{L}$, let $\Pi_v=v+\R v_1+\R v_2=v+\R b_1+\R b_2$\label{Def-Pi_v} be the \textit{filled plane} corresponding to $v$ (recall that $\mc{L}_v=v+\Z v_1+\Z v_2$). Given $v\in \mc{L}\setminus \mc{L}_{(0,1,0,1)}$ and $h\in H$, we call $\mc{L}_vh$ an \textit{exceptional subgrid} of $\mc{L}h$ if $\mc{L}_vh\cap \mc{T}(s)=\emptyset$ but $\Pi_vh\cap \mc{T}(s)\neq\emptyset$.
	
	Note that a general vector in $\mc{L}$ can be written as \[v(\b_1,\b_2):=(\b_1,\b_2,\s(\b_1),\s(\b_2))\label{Def-v-b_1-b_2}\] for some $\b_1,\b_2\in \mc{O}_K$.
	
	\begin{lem}
		\label{lem-O_K-determines-sublattice}
		For $\b_1,\b_2,\b_1',\b_2'\in \mc{O}_K$ we have 
		\[\mc{L}_{v(\b_1,\b_2)}=\mc{L}_{v(\b_1',\b_2')}\] if and only if $-\a_4\b_1+\a_3\b_2=-\a_4\b_1'+\a_3\b_2'$. Thus, every $v\in\mc{L}$ corresponds to a unique $\delta\in \mc{O}_K$, namely $\delta=\delta(v)=-\a_4\b_1+\a_3\b_2$, where $(\b_1,\b_2)\in\mc{O}_K^2$ is any pair so that $v(\b_1,\b_2)\in \mc{L}_v$.
		
		In particular, $v(\b_1,\b_2)\in \mc{L}_{(0,1,0,1)}$ if and only if $(\b_1,\b_2)\in (0,1)+\mc{O}_K(\a_3,\a_4)$.
	\end{lem}
	
	\begin{proof}
		Suppose $\mc{L}_{v(\b_1,\b_2)}=\mc{L}_{v(\b_1',\b_2')}$ so that, in particular, $v(\b_1,\b_2)=v(\b_1',\b_2')+k_1v_1+k_2v_2$ for some $k_1,k_2\in \Z$. The pair of the first two coordinates of the right hand side reads $(\b_1'+(k_1+\sqrt{2}k_2)\a_3,\b_2'+(k_1+\sqrt{2}k_2)\a_4)$ and must be equal to $(\b_1,\b_2)$. From this it follows that $-\a_4\b_1+\a_3\b_2=-\a_4\b_1'+\a_3\b_2'$.
		
		Suppose that $-\a_4\b_1+\a_3\b_2=-\a_4\b_1'+\a_3\b_2'$. Then, $\a_4(\b_1'-\b_1)=\a_3(\b_2'-\b_2)$. Recalling that $\a_3,\a_4$ are relatively prime, it follows that $\b_1'-\b_1=\a_3(k_1+\sqrt{2}k_2)$ and $\b_2'-\b_2=\a_4(k_1+\sqrt{2}k_2)$ for some integers $k_1,k_2$. It is then verified that $v(\b_1,\b_2)=v(\b_1',\b_2')+k_1v_1+k_2v_2$, from which the other implication follows.
		
		The final claim follows by noting that $(0,1,0,1)=v(0,1)$. This implies that $-\a_4\b_1+\a_3\b_2=\a_3$ or, equivalently, $\b_1\a_4=(\b_2-1)\a_3$. Using again that $\a_3,\a_4$ are relatively prime, it follows that $\b_2-1=x\a_4$ and $\b_1=x\a_3$ for some $x\in \mc{O}_K$, which gives the last claim.
	\end{proof}

	\begin{lem}
		\label{lem-filled-plane-empty}
		Fix $v\in \mc{L}\setminus\mc{L}_{(0,1,0,1)}$ and let $\delta\in \mc{O}_K$ be the algebraic integer corresponding to $\mc{L}_v$ (cf.\ \Cref{lem-O_K-determines-sublattice}). 
		Given $h=h(x_1,x_2,y_1,y_2,\theta_1,\theta_2)$ corresponding to parameters appearing in the integral in \Cref{lem-I_1-justification}, set
		\[u_1:=\frac{\a_3x_1+\a_4}{\a_3y_1^2}=\frac{x_1'}{s^{1/2}y_1^2} \text{ and } u_2:=\frac{\s(\a_3)x_2+\s(\a_4)}{\s(\a_3)y_1^2}=\frac{x_2'}{s^{1/2}y_2^2}.\]
		Then $|u_1|,|u_2|\ll 1$. 
		
		Let also $B=B(x_1,y_1,\theta_1,x_2,y_2,\theta_2)$ be defined by
		\begin{equation}\label{eqn-def-B}B:=\a_3y_1\pi_2(\l^{-2r}T(s)k(-\theta_1)n(-u_1))\times \s(\a_3)y_2\pi_2(\l^{2r}\mc{W}k(-\theta_2)n(-u_2)),\end{equation}
		where $\pi_2:\R^2\longrightarrow\R$ denotes projection onto the second coordinate. Then, there is an absolute constant $C_4>0$ such that $B\subset [-C_4,C_4]^2$. Finally, $\Pi_vh\cap \mc{T}(s)=\emptyset$ is equivalent with
		$(\delta,\s(\delta))\notin B$.
	\end{lem}
	
	\begin{proof}
		From the definitions of $b_1,b_2$ (cf.\ \eqref{eqn-b_1-b_2}) we have
		\[\R b_1h=(\R\times \{0\})n(u_1)k(\theta_1)\times \{(0,0)\},\quad \R b_2h=\{(0,0)\}\times(\R\times \{0\})n(u_2)k(\theta_2).\]
		We have that $\Pi_vh\cap \mc{T}(s)=\emptyset$ is equivalent with 
		\begin{equation*}
		\label{eqn-ftp-1}
		(vh\,\mathrm{diag}(k(-\theta_1)n(-u_1),k(-\theta_2)n(-u_2))+(\R\times\{0\})^2)\cap B'=\emptyset
		\end{equation*}
		where 
		\[B'=B'(x_1,x_2,y_1,y_2,\theta_1,\theta_2):=\l^{-2r}T(s)k(-\theta_1)n(-u_1)\times \l^{2r}\mc{W}k(-\theta_2)n(-u_2).\]
		Suppose now that $v=v(\b_1,\b_2)$. It can be verified that
		\[\pi_2((\b_1,\b_2)n(x_1)a(y_1)n(-u_1))=\frac{\d}{\a_3y_1},\]
		where $\d=-\a_4\b_1+\a_3\b_2$ is the algebraic integer corresponding to $\mc{L}_v$.
		Similarly, \[\pi_2((\s(\b_1),\s(\b_2))n(x_2)a(y_2)n(-u_2))=\frac{\s(\d)}{\s(\a_3)y_2}.\]	
		From this it follows that $\Pi_vh\cap \mc{T}(s)=\emptyset$ is equivalent with $(\d,\s(\d))\notin B$. Recall that $y_1,y_2\asymp s^{-1/4}$ from \Cref{lem-size-y_1-y_2}. This implies that $|u_1|,|u_2|\ll 1$ and hence that $B$ is contained in a bounded box.
	\end{proof}

	We now note the following consequence of \Cref{lem-filled-plane-empty}. 
	There is a finite set \[\Delta\subset \mc{O}_K\label{Def-Delta}\] such that for all large $s$, and any $h\in H$ appearing in the integral in \Cref{lem-I_1-justification}, every possible exceptional subgrid $\mc{L}_vh$ of a lattice $\mc{L}h$ corresponds to some $\delta\in \Delta$. We now fix such a set $\Delta$ for the remainder of this section. We assume that $\a_3\notin \Delta$ since the lattice corresponding to $\delta$ is equal to $\mc{L}_{(0,1,0,1)}$ if and only if $\delta=\a_3$. Let \[\mc{I}_\delta(h):=I(\mc{L}h\text{ contains an exceptional subgrid corresponding to }\delta)\label{Def-I-delta}.\]
	
	\begin{lem}
		\label{lem-J_1-J_2}
		Fix $v=v(\b_1,\b_2)$ in $\mc{L}\setminus \mc{L}_{(0,1,0,1)}$ and $h=h(x_1,x_2,y_1,y_2,\theta_1,\theta_2)=\mathrm{diag}(h_1,h_2)$. Let 
		\[J_1=J_1(\b_1,\b_2,h):=\{x\in \R\mid (\b_1,\b_2)h_1+xb_1h_1\in \l^{-2r}T(s)\}\label{Def-J_1}\]
		and 
		\[J_2=J_2(\b_1,\b_2,h):=\{x\in \R\mid (\s(\b_1),\s(\b_2))h_2+xb_2h_2\in \l^{2r}\mc{W}\}\label{Def-J_2}.\]
		We have $\mc{L}_vh\cap \mc{T}(s)=\emptyset$ if and only if $\mc{L}'\cap (J_1\times J_2)=\emptyset$, in which case $|J_1|\cdot |J_2|\ll 1$.
	\end{lem}

	\begin{proof}
		Recall from \eqref{eqn-connection-v_i-b_i} that
		\[\mc{L}_v=v+\{\a b_1+\s(\a) b_2\mid \a\in \mc{O}_K\}.\]
		It thus follows from the definitions of $v(\b_1,\b_2)$, $J_1$, $J_2$ and $\mc{T}(s)$ that $\mc{L}_vh\cap \mc{T}(s)=\emptyset$ if and only if $\mc{L}'\cap (J_1\times J_2)=\emptyset$. If $\mc{L}'\cap (J_1\times J_2)=\emptyset$, then we have $|J_1|\cdot |J_2|\ll 1$ by \Cref{lem-L'}.
	\end{proof}
	
	In the following lemma we obtain strong restrictions on those $h$ that admit an exceptional sublattice and contribute to the integral in \Cref{lem-I_1-justification}.
	
	\begin{lem}
		\label{lem-J_1-large}
		
		
		There is an absolute constant $C_5>0$ such that for large $s$ and any $\delta\in \Delta$ we have
		\[\frac{\l^{2r}}{s^{1/2}}\int_{0}^{s^{-5/8}}\int_{\mc{W}'}\int_{B(\a_3,\a_4)}\mc{I}_1'(h)\mc{I}_2(h)\mc{I}_3\cdot I(\mc{I}_\delta(h)=1,|J_1|< C_5s^{1/2})\,dx_1\,dx_2\,dw\,dz_2'\ll s^{-11/4}.\]
	\end{lem}
	
	\begin{proof}
		Fix $v=v(\b_1,\b_2)$ so that $\d=\a_3\b_2-\a_4\b_1$. Assume $h$ is such that the integrand of the integral in the formulation of the lemma is equal to $1$. As usual, write $h=\mathrm{diag}(h_1,h_2)$. We study the distance $d$ between the parallel lines $\inv{y_1}(\sin\theta_1,\cos\theta_1)+\R b_1h_1$ and $(\b_1,\b_2)h_1+\R b_1h_1$. Note that $\inv{y_1}(\sin\theta_1,\cos\theta_1)=(0,\inv{y_1})k(\theta_1)=(0,1)a(y_1)k(\theta_1)$ and that $d$ is the length of the projection of \[\inv{y_1}(\sin\theta_1,\cos\theta_1)-(\b_1,\b_2)h_1\] on any non-zero $v\in\{\R b_1h_1\}^\perp$. We have \[\inv{y_1}(\sin\theta_1,\cos\theta_1)-(\b_1,\b_2)h_1=(-\b_1,1-\b_2)n(x_1)a(y_1)k(\theta_1).\] We now note that $b_1h_1=(\alpha_3,\a_4)n(x_1)a(y_1)k(\theta_1)$ and thus a convenient choice of $v$ is \[(-\a_4,\a_3)n(-x_1)^Ta(1/y_1)^Tk(-\theta_1)^T.\] 
		It follows that
		\[(\inv{y_1}(\sin\theta_1,\cos\theta_1)-(\b_1,\b_2)h_1)\cdot v=(-\b_1,1-\b_2)\cdot (-\a_4,\a_3)=\a_4\b_1-\a_3\b_2+\a_3=\a_3-\d.\]
		Let $x_1=-\frac{\a_4}{\a_3}+\widetilde{x_1}$ so that $|\widetilde{x_1}|\ll s^{-1/2}$, since $x_1\in B(\a_3,\a_4)$. We have
		\begin{align*}
		\norm{v}&=\norm{(-\a_4,\a_3)n(-x_1)^Ta(1/y_1)}=\norm{((-\a_4-\a_3x_1)/y_1,\a_3y_1)}\\
		&=|\a_3|\norm{(-\widetilde{x_1}/y_1,y_1)}=|\a_3|\sqrt{(\widetilde{x_1})^2y_1^{-2}+y_1^2}.
		\end{align*}	
		Since $y_1\asymp s^{-1/4}$, we have $(\widetilde{x_1})^2y_1^{-2}+y_1^2\ll s^{-1/2}$ and hence $\norm{v}\ll s^{-1/4}$, which implies $d=|\a_3-\delta|/\norm{v}\gg s^{1/4}$ (recall that $\a_3\notin\Delta$). Similarly, one calculates that the distance between $(\b_1,\b_2)h_1+\R b_1h_1$ and the origin is $|\d|/\norm{v}$.
		
		Now, since $\mc{I}_2(h)=1$ we have $\mc{L}_{0}h\cap \mc{T}(s)=\emptyset$ in particular, since $0\notin \mc{L}_{(0,1,0,1)}$. As in \Cref{lem-z_2-large-asymp-neg}, we may assume that $b_1h_1$ is not vertical. From \[\mc{L}_0=\Z v_1+\Z v_2=\{\a b_1+\s(\a)b_2\mid \a\in\mc{O}_K\}\] we see that if $(x,y):=b_1h_1$, then $|y|\geq |x|$. In the proof of \Cref{lem-mcI3} it was seen that $y-x<0$ for large $s$. Combining this with $|y|\geq |x|$ it follows that $y\leq -|x|$. 
		
		If $x>0$ (recall that we can assume that $b_1h_1$ is not vertical, i.e.\ that $x\neq 0$), it follows from $d\gg s^{1/4}$ and the fact that the distance from $(\b_1,\b_2)h_1+\R b_1h_1$ to the origin is $\gg s^{1/4}$ that the intersection of the line $(\b_1,\b_2)h_1+\R b_1h_1$ and $\l^{-2r}T(s)$ has length $\gg s^{1/4}$ and since $\norm{b_1h_1}\ll s^{-1/4}$ (cf.\ \Cref{lem-size-v_i-b_i}) it follows that $|J_1|\gg s^{1/2}$, which contradicts our assumption that $|J_1|\leq C_5s^{1/2}$, provided that $C_5$ is chosen sufficiently small.
		
		Hence we can assume that $x<0$. As in \Cref{lem-mcI3}, we see that $x<0$ is equivalent with $x_1'<-s^{1/2}y_1^2\cot\theta_1$. However, we have $\mc{I}_3=1$, and as in the proof of \Cref{lem-mcI3} this implies that $x_1'$ is confined to an interval of length $\ll s^{-7/8}$ around $-\frac{c^2\l^{4r}}{2s^{1/2}}$. This finishes the proof of the lemma.
		
	\end{proof}
	
	\begin{lem}
		\label{lem-J_2} 
		Fix $\delta\in \Delta$. Then for large $s>0$ we have
		\begin{equation}
		\label{eqn-lem-J_2}
		\frac{\l^{2r}}{s^{1/2}}\int_{0}^{s^{-5/8}}\int_{\mc{W}'}\int_{B(\a_3,\a_4)}I(\mc{I}_\delta(h)=1,|J_1|\geq C_5s^{1/2})dx_1dx_2dwdz_2'\ll s^{-17/8}.
		\end{equation}
	\end{lem}
	
	\begin{proof}
		Consider an arbitrary choice of $x_1,x_2,w,z_2'$ appearing in the integral, and assume that the corresponding $h=h(x_1,x_2,y_1,y_2,\theta_1,\theta_2)$, is such that the integrand is equal to $1$, i.e.\ assume that $\mc{I}_\delta(h)=1$ and $|J_1|\geq C_5s^{1/2}$. Let $h_2=n(x_2)a(y_2)k(\theta_2)$ and let $v=v(\b_1,\b_2)$ be a vector in the subgrid corresponding to $\d$, and consider the line \[L=L(\b_1,\b_2,h_2):=(\s(\b_1),\s(\b_2))h_2+\R b_2h_2=((\s(\b_1),\s(\b_2))+\R(\s(\a_3),\s(\a_4)))h_2.\]
		It follows from $|J_1|\geq C_5s^{1/2}$ and \Cref{lem-J_1-J_2} that $|J_2|\ll s^{-1/2}$. Thus, the length of the intersection of $L$ and $\l^{2r}\mc{W}$ must be $\ll s^{-3/4}$ (by the definition of $J_2$ and the fact that $\norm{b_2h_2}\ll s^{-1/4}$). Hence, there is an absolute constant $c_1>0$ and a vertex of $\l^{2r}\mc{W}$ whose distance to $L$ is less than or equal to $c_1 s^{-3/4}$. Let $E$ be the union of the eight balls of common radius $c_1 s^{-3/4}$ centered at the vertices of $\l^{2r}\mc{W}$. Note that we can write $L=L'k(\theta_2)$, where $L'=L'(x_2,y_2)$ is also a line.
		
		Now we study for which $\theta_2'$ close to $\theta_2$ the line $L'k(\theta_2')$ intersects $E$. To this end, we study the length $d$ of the intersection of $L$ and the annulus $A$ traced by $E$ as it is rotated about the origin. Note that the inner radius $r_A$ of $A$ is $\asymp s^{1/4}$ since the inner radius of $\l^{2r}\mc{W}$ is $\asymp s^{1/4}$. We have that $d$ is maximal when $L$ is tangent to the inner circle forming $A$, so we assume this to be the case. We may further assume that the point of tangency is located at the $x$-axis.
		
		Let $r_0=r_A-\eta$ and let $y_0>0$ be the number such that $(r_0,y_0)$ is on the outer boundary of $A(w,x_2)$. Let $\theta_0$ be the angle at the origin of the triangle with vertices $(0,0)$, $(r_0,0)$ and $(r_0,y_0)$. This triangle is right angled with hypotenuse of length $r_A+\eta$. We find that $r_0^2+y_0^2=(r_A+\eta)^2$, or equivalently, $y_0=2\sqrt{r_A\eta}\ll s^{-1/4}$. Thus, $\sin\theta_0=\frac{y_0}{r_A+\eta}\ll s^{-1/2}$ so $\theta_0\ll s^{-1/2}$ for large $s$. 
		
		Thus, we conclude that any line $L=L'k(\theta_2)$ that intersects the ball $E$ no longer does so if $\theta_2$ is changed to $\theta_2'$ with $|\theta_2-\theta_2'|\gg s^{-1/2}$ (with the difference still being small, say, $\ll s^{-1/2}$). This means, that for each fixed $x_2,y_2$ there is a union of intervals $I(x_2,y_2)$ of total length $\ll s^{-1/2}$ such that $\theta_2\in I(x_2,y_2)$ if $|J_1|\geq C_5s^{-1/2}$. 
		
		Now, recall that $\l^{2r}w=\inv{y_2}(\sin\theta_2,\cos\theta_2)$. Make the changes of variables $\theta_2=\pi/2-\theta_2'$ and $y_2'=y_2^{-1}\l^{-2r}$. Then, $w=y_2'(\cos\theta_2',\sin\theta_2')$ and hence $dw=dw_1dw_2=y_2'dy_2'd\theta_2'$ or $dw=y_2^{-3}\l^{-4r}dy_2d\theta_2$.
		
		The corresponding contribution to the integral in \eqref{eqn-lem-J_2} is 
		\[\ll s^{-1/4}\int_{0}^{s^{-5/8}}\int_{|x_1|\leq C_1s^{-1/2}}\int_{y_2\asymp s^{-1/4}}\int_{|x_2|\leq C_1s^{-1/2}}\int_{I(x_2,y_2)}y_2^{-3}\l^{-4r}d\theta_2dx_2dy_2dx_1dz_2'.\]
		Note that $y_2^{-3}\l^{-4r}\asymp s^{1/4}$. Thus, the above expression is $\ll s^{-19/8}\ll s^{-17/8}$ as desired.
	\end{proof}

	\begin{lem}
		\label{lem-summary-no-exceptional-sublattice}
		For all large $s$ we have
		\begin{align*}
		F^+_{(\a_3,\a_4)}(s)=\frac{\l^{2r}}{s^{1/2}}\int_{0}^{s^{-5/8}}\int_{\mc{W}'}\int_{B(\a_3,\a_4)}\mc{I}_1'(h)\mc{I}_2'(h)\mc{I}_3\,dx_1\,dx_2\,dw\,dz_2'+\mc{O}(s^{-17/8}).\end{align*}
		where 
		\[\mc{I}_2'(h):=I(\mc{L}'\cap B\subset \{(\a_3,\s(\a_3))\})\label{Def-mcI2'}\]
		(recall the definition of $B$ in \eqref{eqn-def-B}).
	\end{lem}

	\begin{proof}
		In view of \Cref{lem-J_1-large} and \Cref{lem-J_2} we may assume that $\mc{L}h$ contains no exceptional subgrid, i.e.\ that $\mc{I}_\delta(h)=0$ for all $\delta\in \Delta$, since the error introduced by this assumption is $\mc{O}(s^{-17/8})$. Then, by \Cref{lem-filled-plane-empty} we have $\mc{I}_2(h)=1$ if and only if $(\d,\s(\d))\notin B$ for all $\d\in \mc{O}_K\setminus \{\a_3\}$. In other words, $\mc{I}_2(h)=1$ is equivalent with $\mc{L}'\cap B\subset \{(\a_3,\s(\a_3))\}$.
	\end{proof} 
	
	Let $B_1=B_1(x_1,y_1,\theta_1)$, $B_2=B_2(x_2,y_2,\theta_2)$\label{Def-B_1-B_2} be the intervals so that $B=B_1\times B_2$. Let $T$ be the open triangle with vertices at $(0,0)$, $(1,0)$ and $(0,1)$\label{Def-T}. Let \[B'=B'(x_1,x_2,y_2,\theta_2):=B_1'\times B_2,\] where 
	\[B_1'=B_1'(x_1):=\a_3\pi_2\left(Tn(-2c^{-2}s^{1/2}\l^{-4r}x_1')\right)\label{Def-B'}.\]
	
	\begin{lem}
		\label{lem-B-approximation}
		For all large $s$ we have
		\begin{align*}
		F^+_{(\a_3,\a_4)}(s)
		=\frac{\l^{2r}}{s^{1/2}}\int_{0}^{s^{-5/8}}\int_{\mc{W}'}\int_{B(\a_3,\a_4)}\mc{I}_1'(h)\mc{I}_2''(h)\mc{I}_3dx_1dx_2dwdz_2'+\mc{O}(s^{-17/8}).\end{align*}
		where 
		\[\mc{I}_2''(h):=I(\mc{L}'\cap B'\subset \{(\a_3,\s(\a_3))\})\label{Def-mcI2''}\]
	\end{lem}
	
	\begin{proof}
		Fix $y_1,y_2,\theta_1,\theta_2$ and $x_2$. Hence $x_2'$ is also fixed. For varying $x_1'$, consider the corresponding boxes $B$ and $B'$. We will show that there is an absolute constant $c_1>0$ such that if exactly one of $\mc{L}'\cap B\subset \{(\a_3,\s(\a_3))\}$ and $\mc{L}'\cap B'\subset \{(\a_3,\s(\a_3))\}$ holds, then $x_1'$ belongs to a union of intervals of total length $\leq c_1s^{-1/4}z_2'$.
		
		Recall that $d(s)=\l^{-2r}s^{1/2}\inv{c}\asymp s^{1/4}$ and that $z_2'<s^{-5/8}$. Recall from \eqref{eqn-inv-y_1} that $y_1^{-1}=\sqrt{2}d(s)(1+\mc{O}(s^{-1/4}z_2'))$. It follows that \[y_1=\frac{1}{\sqrt{2}d(s)}(1+\mc{O}(s^{-1/4}z_2'))=\frac{\l^{2r}c}{s^{1/2}\sqrt{2}}(1+\mc{O}(s^{-1/4}z_2')).\] 
		Recall that $\theta_1$ is close to $\pi/4$; in fact \[\sin\theta_1=\frac{1}{\sqrt{2}}(1+\mc{O}(s^{-1/4}z_2')) \text{ and } \cos\theta_1=\frac{1}{\sqrt{2}}(1+\mc{O}(s^{-1/4}z_2'))\] (see the proof of \Cref{lem-I_1-justification}). This implies that $k(-\theta_1)=k(-\pi/4)A(\theta_1)$ where
		\[A(\theta_1)=\begin{pmatrix}
		1+f_1(\theta_1) & f_2(\theta_1)\\ -f_2(\theta_1) & 1+f_1(\theta_1)
		\end{pmatrix}\]
		for some $f_1,f_2\in \mc{O}(s^{-1/4}z_2')$.
		
		Recall that $T(s)$ is the triangle with vertices at $(0,0)$ and $s^{1/2}\inv{c}(1,\pm 1)$. We see that $T(s)k(-\theta_1)=\sqrt{2}s^{1/2}\inv{c}TA(\theta_1)$. Thus, using $y_1=\frac{\l^{2r}c}{s^{1/2}\sqrt{2}}(1+\mc{O}(s^{-1/4}z_2'))$ we have 
		\[y_1\l^{-2r}T(s)k(-\theta_1)=(1+\mc{O}(s^{-1/4}z_2'))TA(\theta_1).\]
		Set \[u_1':=\frac{2s^{1/2}x_1'}{\l^{4r}c^2}\]
		and recall that $u_1=\frac{x_1'}{s^{1/2}y_1^2}$.
		Next, we compare $u_1$ and $u_1'$.
		We have \[y_1^{-2}=\frac{2s}{\l^{4r}c^2}(1+\mc{O}(s^{-1/4}z_2'))\] and hence $u_1=u_1'+\mc{O}(s^{-1/4}z_2')$. Thus
		\[y_1\l^{-2r}T(s)k(-\theta_1)n(-u_1)=(1+\mc{O}(s^{-1/4}z_2'))TA(\theta_1)n(-u_1'+\mc{O}(s^{-1/4}z_2')).\]
		We wish to compare the projection $\pi_2$ of this figure to $\pi_2(Tn(-u_1'))=\inv{\a_3}B_1'$.
		Now, the vertices of $TA(\theta_1)$ are 
		\[(0,0),\quad (1,0)A(\theta_1)=(1+f_1(\theta_1),f_2(\theta_1))\text{ and }(0,1)A(\theta_1)=(-f_2(\theta_1),1+f_1(\theta_1)).\]
		Thus, the vertices of $TA(\theta_1)n(-u_1'+\mc{O}(s^{-1/4}z_2'))$ are $(0,0)$,
		\[(1+f_1(\theta_1),(1+f_1(\theta_1))(-u_1'+\mc{O}(s^{-1/4}z_2'))+f_2(\theta_1))\]
		and \[(-f_2(\theta_1),-f_2(\theta_1)(-u_1'+\mc{O}(s^{-1/4}z_2'))+1+f_1(\theta_1)).\]
		Now, recalling that $|u_1'|\ll 1$ and $f_1,f_2\in \mc{O}(s^{-1/4}z_2')$, we have
		\[(1+f_1(\theta_1))(-u_1'+\mc{O}(s^{-1/4}z_2'))=-u_1'+\mc{O}(s^{-1/4}z_2')\]
		and 
		\[-f_2(\theta_1)(-u_1'+\mc{O}(s^{-1/4}z_2'))+1+f_1(\theta_1)=1+\mc{O}(s^{-1/4}z_2').\]
		Similarly, the $y$-coordinates of points in $\mc{O}(s^{-1/4}z_2')TA(\theta_1)n(-u_1'+\mc{O}(s^{-1/4}z_2'))$ belong to an interval of length $\ll s^{-1/4}z_2'$ centered at the origin. In summary, we conclude that every point in the symmetric difference $B_1\triangle B_1'$ must have distance $\ll s^{-1/4}z_2'$ both to an endpoint of $B_1$ and an endpoint of $B_1'$.
		
		Now, there is an absolute constant $c_1$ such that $B\cup B'\subset [-c_1,c_1]^2$. Let $F'$ be the finite set $\mc{L}'\cap [-c_1,c_1]^2$. Suppose that $\mc{L}'\cap B\subset \{(\a_3,\s(\a_3))\}$ and $\mc{L}'\cap B'\not\subset \{(\a_3,\s(\a_3))\}$. Then, there is some $(\gamma,\s(\gamma))\in F'$, with $\gamma\neq\a_3$, such that $(\gamma,\s(\gamma))\in B'\setminus B$, i.e.\ $\gamma\in B_1'\setminus B_1$. This means that $\gamma$ must be within a distance $\ll s^{-1/4}z_2'$ of an endpoint of $B_1'$. Now, the vertices of $Tn(-u_1')$ are $(0,0)$, $(1,-u_1')$ and $(0,1)$, so $B_1'$ is the interval
		\[\a_3(\min(0,-u_1'),\max(1,-u_1')).\]
		This means that either $\gamma=0$ or that $|-\a_3u_1'-\gamma|\ll s^{-1/4}z_2'$. If $\gamma=0$, then $-u_1'$ has to be negative, since it is only then we can have $(0,0)\in B'$. However, the endpoint $-u_1$ of $B_1'$ must have distance $\ll s^{-1/4}z_2'$ from an endpoint of $B_1$, and $B_1$ cannot contain $0$ since then $(0,0)\in B$, which contradicts $\a_3\neq 0$. Hence, in that case $|\a_3u_1'|\ll s^{-1/4}z_2'$. From the definition of $u_1'$ it follows that $x_1'$ has to belong to a union of intervals of total measure $\ll s^{-1/4}z_2'$.
		
		Suppose now that $\mc{L}'\cap B\not\subset \{(\a_3,\s(\a_3))\}$ and $\mc{L}'\cap B'\subset \{(\a_3,\s(\a_3))\}$. Then, there is some $(\gamma,\s(\gamma))\in F'$, with $\gamma\neq\a_3$, such that $(\gamma,\s(\gamma))\in B\setminus B'$, i.e.\ $\gamma\in B_1\setminus B_1'$. 
		Recall that the vertices of the triangle defining $B_1$ have $y$-coordinates $0$, $-\a_3u_1'+\mc{O}(s^{-1/4}z_2')$ and $\a_3+\mc{O}(s^{-1/4}z_2')$. If $\gamma=0$, then $|\a_3 u_1'|\ll s^{-1/4}z_2'$ again, since $-\a_3u_1+\mc{O}(s^{-1/4}z_2')$ has to be negative, while $-\a_3 u_1$ cannot be (since then $(0,0)\in B_1'$ in view of the above formula for the endpoints of $B_1'$). Otherwise, $|-\a_3u_1'-\gamma|\ll s^{-1/4}z_2'$. We again conclude that $x_1'$ has to belong to a finite union of intervals with total length $\ll s^{-1/4}z_2'$. The claim of the lemma follows.\end{proof}
	
	\subsubsection{Summary and proof of main result}
	
	\begin{lem}
		\label{lem-extend-integration}
		Make the changes of variables\[x_1=s^{-1/2}x_1'-\a_4/\a_3 \text{ and } x_2=s^{-1/2}x_2'-\s(\a_4/\a_3)\]
		in the integral in \Cref{lem-B-approximation}
		Then, for all large $s$ we have
		\begin{equation*}
		F^+_{(\a_3,\a_4)}(s)
		=\frac{\l^{2r}}{s^{3/2}}\int_{0}^{\infty}\int_{\mc{W}'}\int_\R\int_\R\mc{I}_1'(h)\mc{I}_2''(h)\mc{I}_3dx_1'dx_2'dwdz_2'+\mc{O}(s^{-17/8}).\end{equation*}
	\end{lem}

	\begin{proof}
		Recall the definitions of $\mc{I}_1'(h)=I(\mc{L}'\cap (\widetilde{I_1}'\times \widetilde{I_2})=\emptyset)$ and its components from \Cref{lem-I_1-justification} and consider 
		\[\frac{\l^{2r}}{s^{1/2}}\int_{s^{-5/8}}^{\infty}\int_{\mc{W}'}\int_{B(\a_3,\a_4)}\mc{I}_1'(h)dx_1dx_2dwdz_2'.\]	
		Recall also that $\widetilde{I_1}'=\widetilde{I_1}'(z_2')=\left(-\frac{\l^{2r}s^{1/2}z_2'}{c\a_3},0\right)$. Define $\delta(w)$ as in \eqref{eqn-delta-w}. We have $|\widetilde{I_1}'|\asymp s^{3/4}z_2'$ and $|\widetilde{I_2}|\gg \l^{-2r}s^{1/4}\delta(w)$. Hence, there is a constant $C_6>0$ such that if $\mc{I}_1'(h)=1$, then $z_2'\delta(w)\leq C_6\l^{2r}s^{-1}$. Let now
		\[\mc{W}(z_2'):=\{w\in \mc{W}\mid z_2'\delta(w)\leq C_6\l^{2r}s^{-1}\}.\]
		Then we have that $\mc{I}_1'(h)=1$ only if $w\in \mc{W}(z_2')$. From this it follows that 
		\[\frac{\l^{2r}}{s^{1/2}}\int_{s^{-5/8}}^{\infty}\int_{\mc{W}'}\int_{B(\a_3,\a_4)}\mc{I}_1'(h)dx_1dx_2dwdz_2'\in \mc{O}(s^{-17/8})\]
		and hence that
		\begin{equation}
		F^+_{(\a_3,\a_4)}(s)
		=\frac{\l^{2r}}{s^{3/2}}\int_{0}^{\infty}\int_{\mc{W}'}\int_{-C_1}^{C_1}\int_{-C_1}^{C_1}\mc{I}_1'(h)\mc{I}_2''(h)\mc{I}_3dx_1'dx_2'dwdz_2'+\mc{O}(s^{-17/8})\label{eqn-lem-extend-integration-1}
		\end{equation}
		where $C_1$ originates from the definition of $B(\a_3,\a_4)$.
		
		Now recall that \[\mc{I}_2''(h)=I(\mc{L}'\cap (B_1'\times B_2)\subset \{(\a_3,\s(\a_3))\})\] with
		\[B_1'=\a_3\pi_2\left(Tn(-2c^{-2}s^{1/2}\l^{-4r}x_1')\right) \text{ and } B_2=\s(\a_3)\pi_2(y_2\l^{2r}\mc{W}k(-\theta_2)n(-x_2'/(s^{1/2}y_2^2))).\]
		Note that $|B_1'|\gg 1$ and $|B_2|\gg 1$. Indeed, the first claim follows by noting that
		\[B_1'=\a_3(\min(0,-u_1'),\max(1,-u_1'))\]
		where $u_1'=2c^{-2}s^{1/2}\l^{-4r}x_1'$. This also implies that $|B_1'|\asymp |x_1'|$ for large $|x_1'|$. The second claim follows by noting that $\norm{w}=(y_2\l^{2r})^{-1}$, so that 
		\[B_2=\s(\a_3)\pi_2(\inv{\norm{w}}\mc{W}k(-\theta_2)n(-x_2'\norm{w}^2\l^{4r}/s^{1/2})).\] As $\inv{\norm{w}}\mc{W}k(-\theta_2)$ always contains a disc of some absolute radius $r_1>0$ centered at the origin, it contains the line segment between $(0,r_1)$ and $(0,-r_1)$ in particular. It follows that $B_2$ contains the interval $(-|\s(\a_3)|r_1,|\s(\a_3)|r_1)$, so $|B_2|\gg 1$. Using the fact that $\inv{\norm{w}}\mc{W}k(-\theta_2)$ also contains the line segment between $(r_1,0)$ and $(-r_1,0)$ we have \[|\s(\a_3)|(-r_1|x_2'|\norm{w}^2\l^{4r}/s^{1/2},r_1|x_2'|\norm{w}^2\l^{4r}/s^{1/2})\subset B_2,\] which implies that $|B_2|\gg |x_2'|$. We conclude that $\mc{L}'\cap B'\subset \{(\a_3,\s(\a_3))\}$ necessarily fails if either $|x_1'|$ \textit{or} $|x_2'|$ is larger than some absolute constant $C_7$. Now, we are free to choose the constant $C_{1}$ defining $B(\a_3,\a_4)$ as being larger than $C_7$. Thus, we may integrate over all $x_1',x_2'\in \R$ in \eqref{eqn-lem-extend-integration-1} without affecting the value of the integral, therefore the claim of the lemma is established.
	\end{proof}
	
	We can now prove:
	
	\begin{prop}
		\label{prop-a_3-a_4-contr-decay}
		Fix $(\a_3,\a_4)\in A$. Then there exists $c_{(\a_3,\a_4)}\in \R_{>0}$ such that  \[F_{(\a_3,\a_4)}(s)=c_{(\a_3,\a_4)}s^{-2}+\mc{O}(s^{-17/8})\] as $s\to\infty$.
	\end{prop}
	
	\begin{proof}
		By \Cref{lem-wlog-z_2>0} we have $F_{(\a_3,\a_4)}(s)=F_{(\a_3,\a_4)}^+(s)+F_{(\a_3,-\a_4)}^+(s)$. Hence, it suffices to show that for each $(\a_3,\a_4)\in A'$ there exists a constant $c_{(\a_3,\a_4)}^+>0$ so that $F_{(\a_3,\a_4)}^+(s)=c_{(\a_3,\a_4)}^+s^{-2}+\mc{O}(s^{-17/8})$.
		By \Cref{lem-extend-integration} we have 
		\begin{equation}
		\label{eqn-prop-a_3-a_4-contr-decay-1}
			F^+_{(\a_3,\a_4)}(s)
			=\frac{\l^{2r}}{s^{3/2}}\int_{0}^{\infty}\int_{\mc{W}'}\int_\R\int_\R\mc{I}_1'(h)\mc{I}_2''(h)\mc{I}_3dx_1'dx_2'dwdz_2'+\mc{O}(s^{-17/8}).
		\end{equation}
		Recall that $\mc{I}_1'(h)=I(\mc{L}'\cap (\widetilde{I_1}'\times \widetilde{I_2})=\emptyset)$, where \[\widetilde{I_1}'=\widetilde{I_1}'(z_2'):=\left(-\frac{\l^{2r}s^{1/2}z_2'}{c\a_3},0\right)\] and 
		\[\widetilde{I_2}=\l^{-4r}\{x\in \R\mid z^2+xb_2h_2\in\l^{2r}\mc{W}\}.\]
		Recall also that \[\mc{I}_2''(h)=I(\mc{L}'\cap B'\subset \{(\a_3,\s(\a_3))\})\]
		where $B'=B_1'\times B_2$, with
		\[B_1'=\a_3\pi_2\left(Tn(-2c^{-2}s^{1/2}\l^{-4r}x_1')\right) \text{ and } B_2=\s(\a_3)\pi_2(y_2\l^{2r}\mc{W}k(-\theta_2)n(-x_2'/(s^{1/2}y_2^2)))\]
		and that $\mc{I}_3=I\left(x_1'>-\frac{c^2\l^{4r}}{2s^{1/2}}\right)$.
		
		Now note that we can rewrite $\widetilde{I_2}$ as 
		$\{x\in \R\mid w+x\l^{2r}b_2h_2\in\mc{W}\}$. Since $\norm{w}=\l^{-2r}\inv{y_2}$ we have
		\begin{align*}
		\l^{2r}b_2h_2&=\l^{2r}(y_2\s(\a_3),y_2^{-1}(\s(\a_3)x_2+\s(\a_4)))k(\theta_2)=\l^{2r}\s(\a_3)(y_2,\inv{y_2}s^{-1/2}x_2')k(\theta_2)\\
		&=\s(\a_3)(\inv{\norm{w}},s^{-1/2}\l^{4r}x_2'\norm{w})k(\theta_2).
		\end{align*} 	
		Then make the changes of variables $z_2=s^{1/2}\l^{2r}z_2'$, $x_1=s^{1/2}\l^{-4r}x_1'$ and $x_2=s^{-1/2}\l^{4r}x_2'$ to conclude that \eqref{eqn-prop-a_3-a_4-contr-decay-1} is equal to
		\begin{equation}
		\label{eqn-prop-a_3-a_4-contr-decay-2}
		F^+_{(\a_3,\a_4)}(s)
		=\frac{1}{s^{2}}\int_{0}^{\infty}\int_{\mc{W}'}\int_\R\int_\R\mc{I}_1'(h)\mc{I}_2''(h)\mc{I}_3dx_1dx_2dwdz_2+\mc{O}(s^{-17/8}),
		\end{equation}
		where the integral is independent of $s$. We now verify that this integral is finite, whence the claim of the proposition follows.
		
		To this end, let us write $\mc{I}_1'(h)$ and $\mc{I}_2''(h)$ explicitly in terms of the new variables of integration. 
		Firstly, $\mc{I}_2''(h)=I(\mc{L}'\cap B'\subset \{(\a_3,\s(\a_3))\})$ where $B'=B_1'\times B_2$ with 
		\[B_1'=\a_3(\min(0,-2c^{-2}x_1),\max(1,-2c^{-2}x_1))\] and $B_2=\s(\a_3)\pi_2(\inv{\norm{w}}\mc{W}k(-\theta_2)n(-\norm{w}^2x_2))$. As in the proof of \Cref{lem-extend-integration}, we see that $\mc{I}_2''(h)$ is zero if either $|x_1|$ or $|x_2|$ is larger than some absolute constant (recall that $\norm{w}$ is bounded away from $0$).
		
		Secondly, we have $\mc{I}_1'(h)=I(\mc{L}'\cap (\widetilde{I_1}'\times \widetilde{I_2})=\emptyset)$, where \[\widetilde{I_1}'=\left(-\frac{z_2}{c\a_3},0\right)\] and 
		\[\widetilde{I_2}=\{x\in \R\mid w+x\s(\a_3)(\inv{\norm{w}},x_2\norm{w})k(\theta_2)\in\mc{W}\}.\] 
		
		Now we note that if $\l^{2n+1}<R<\l^{2n+3}$ for some integer $n>0$, then 
		\begin{align*}
		(-\l^{2n},\s(\l^{2n})),(-\l^{2n+1},\s(\l^{2n+1}))&\in (-R,0)\times [-\l^{-2n},\l^{-2n}]\cap \mc{L}'\\
		\subset & (-R,0)\times\l^4(-\inv{R},\inv{R})\cap \mc{L}'
		\end{align*}
		Thus, if $R>0$ is a large number, $z_2>R$ and $\mc{I}_1'(h)=1$, then $|\widetilde{I_2}|\ll R^{-1}$ which implies that $w$ must be within distance $\ll \inv{R}$ from some vertex of $\mc{W}$, that is, has to belong to a set of area $\ll R^{-2}$. It follows that the integral is indeed finite.
	\end{proof}

	\newpage
	
	We are now ready to prove the main result of the present paper.
	
	\begin{thm}
		\label{thm-F(s)-asymp-decay}
		We have $F(s)=C_\mc{P}s^{-2}+\mc{O}(s^{-17/8})$ as $s\to\infty$.
	\end{thm}
	
	\begin{proof}
		In \eqref{eqn-F(s)-2} we saw that $F(s)$ is a finite sum over $(\a_3,\a_4)\in A$ of terms that according to \Cref{prop-a_3-a_4-contr-decay} decay as $c_{(\a_3,\a_4)}s^{-2}+\mc{O}(s^{-17/8})$ as $s\to\infty$. If follows that $F(s)=Cs^{-2}+\mc{O}(s^{-17/8})$ for some constant $C>0$, but \Cref{prop-G(s)-asympt} implies that $C=C_{\mc{P}}$.
	\end{proof}

	\vspace{1cm}

	\noindent\textbf{Acknowledgements}. The author wishes to thank Jens Marklof and Andreas Strömbergsson for sharing their notes in \cite{marklofKineticPrivate}. The author also wishes to thank Andreas Strömbergsson for his involvement in the creation of this paper: for advice, guidance, patience, suggestions and for reading and commenting numerous preliminary versions of the present manuscript. 
	
	The research leading up to the results in the present paper has been supported by the Knut and Alice Wallenberg Foundation.
	
	\newpage
	
	\bibliographystyle{siam}

	\newpage
	
	\section*{Index of notation}
	
	The following table contains descriptions of most of the recurring notation used throughout the article. The rightmost column contains a reference to the page of the first occurrence of each notation, respectively.
	
	\begin{center}
		\begin{longtable}{llr}
			$\ll$ & if $f,g\geq 0$ then $f\ll g$ if there is $C>0$ so that $f\leq Cg$ & \pageref{Def-ll}\\
			
			$\asymp$ & $f\asymp g$ if $f\ll g$ and $g\ll f$ & \pageref{Def-ll}\\
			
			$a(y)$ & $\begin{pmatrix}y & 0\\ 0 & \inv{y}\end{pmatrix}$ & \pageref{Def-nak}\\
			
			$b_1,b_2$ & a particular basis of $\R v_1+\R v_2$ & \pageref{eqn-b_1-b_2}\\
			
			$b_1,b_2,b_3,b_4$ & a particular basis of $\R^4$ & \pageref{Def-b_i-1}\\
			
			$B$ & $B(x_1',y_1,\theta_1,x_2',y_2,\theta_2)$, a specific box defined in \eqref{eqn-def-B} & \pageref{eqn-def-B}\\
			
			$B_1,B_2$ & $B=B_1\times B_2$ & \pageref{Def-B_1-B_2}\\
			
			$B',B_1'$ & $B'=B_1'\times B_2$ & \pageref{Def-B'}\\
			
			$B(\a_3,\a_4)$ & $\{(x_1,x_2):|x_1+\a_4/\a_3|,|x_2+\s(\a_4/\a_3)|\leq C_1s^{-1/2}\}$ & \pageref{Def-B(a3,a4)}\\ 
			
			$c$ & $\theta(\widehat{\mc{P}})^{1/2}$ & \pageref{Def-c}\\
			
			$\mf{C}(s)$ & $\{(x_1,x_2)\in\R^2: 0<x_1<1, |x_2|<s/\theta(\widehat{\mc{P}})\}$ & \pageref{Def-mfC}\\
			
			$\mf{D}_t$ & & \pageref{Def-mfDt}\\
			
			$\mf{D}_t'$ & & \pageref{Def-mfD}\\
			
			$d(s)$ & $c^{-1}\l^{-2r}s^{1/2}$ & \pageref{Def-d(s)}\\
			
			$\mc{E}(s)$ & $\{x\in X\mid \mc{L}(x)\cap \mc{T}'(s)=\emptyset\}$ & \pageref{def-mc-E}\\
			
			$\mf{F}$ & a bounded fundamental domain of $\mc{L}'$ which is also a  subset of $\R_{>0}^2$ & \pageref{Def-mfF}\\
			
			$F(s)$ & $-\frac{d}{ds}G(s)$ & \pageref{eqn-F(s)}\\
			
			$F_1(s,\eta)$ &  & \pageref{Def-F_1}\\
			
			$F_2(z_1)$ &  & \pageref{Def-F_2}\\
			
			$g$ & a matrix in $\SLR{4}$ such that $\mc{L}=\delta^{1/4}\Z^4g$ for some $\delta>0$ & \pageref{Def-g}\\
			
			$G$ & $\SL{n}{\R}$ & \pageref{Def-G}\\
			
			$G(s)$ & $\mu(\{x\in X\mid \widehat{\mc{P}^x}\cap \mf{C}(s)=\emptyset\})$ & \pageref{Def-G(s)}\\
			
			$h$ & $h=h(x_1,x_2,y_1,y_2,\theta_1,\theta_2)\in H$;  & \pageref{Def-h}\\
			& parameters $x_i,y_i,\theta_i$ from Iwasawa decomposition & \\
			
			$h_1,h_2$ & $h=\mathrm{diag}(h_1,h_2)\in H$ & \\
			
			$H$ & $\SLR{2}^2\subset \SLR{4}$ & \pageref{Def-H}\\
			
			$H_g$ & $gH\inv{g}\subset\SLR{4}$ & \pageref{Def-H_g}\\
			
			$(H_g)_z$ & $\{h\in H_g\mid zh=z\}\subset H_g$ & \pageref{Def-Hgz}\\
			
			$I(\cdot)$ & indicator function & \pageref{Def-indicator}\\
			
			$I_1$ & $\{x\in \R\mid z^1+xb_1h_1\in \l^{-2r} T(s)\}$ &\pageref{Def-I_1-1}\\
			
			$I_2$ & $\{x\in \R\mid z^2+xb_2h_2\in\l^{2r}\mc{W}\}$ & \pageref{Def-I_2}\\
			
			$\mc{I}_1(h)$ & & \pageref{Def-mcI1}\\
			
			$\mc{I}_2(h)$ & & \pageref{Def-mcI2}\\
			
			$\mc{I}_2'(h)$ & & \pageref{Def-mcI2'}\\
			
			$\mc{I}_2''(h)$ & & \pageref{Def-mcI2''}\\
			
			$\mc{I}_3(h)$ & & \pageref{Def-mcI3}\\
			
			$\mc{I}_\delta(h)$ & $I(\mc{L}h\text{ contains an exceptional subgrid corresponding to }\delta)$ & \pageref{Def-I-delta}\\
			
			$J_1$ & $\{x\in \R\mid (\b_1,\b_2)+xb_1h_1\in \l^{-2r}T(s)\}$ & \pageref{Def-J_1}\\
			
			$J_2$ & $\{x\in \R\mid (\s(\b_1),\s(\b_2))+xb_2h_2\in \l^{2r}\mc{W}\}$ & \pageref{Def-J_2}\\
			
			$K$ & $\Q(\sqrt{2})$ & \pageref{Def-K}\\
			
			$k(\theta)$ & $\begin{pmatrix}\cos\theta & -\sin\theta\\ \sin\theta & \cos \theta\end{pmatrix}$ & \pageref{Def-nak}\\
			
			$L$ & a fixed list of representatives of $(\mc{O}_K\smpt{0})/\mc{O}_K^*$ & \pageref{Def-L}\\
			
			$\mc{L}',\mc{L}$ & Minkowski embeddings of $\mc{O}_K$ in $\R^2$ and $\R^4$, respectively & \pageref{Def-L'-L}\\
			
			$\mc{L}(x)$ & $\mc{L}(x)=\mc{L}h$ for $x=gh\inv{g}\in X$ & \pageref{Def-L(x)}\\
			
			$\mc{L}_v$ & $v+\Z v_1+\Z v_2$ & \pageref{Def-L_v}\\
			
			$n(x)$ & $\begin{pmatrix}1 & x\\ 0 & 1\end{pmatrix}$ & \pageref{Def-nak}\\
			
			$\mc{O}_K$ & $\Z[\sqrt{2}]$ & \pageref{Def-O_K}\\
			
			$\mc{P}(\mc{W},\mc{L})$ & cut-and-project set from a lattice $\mc{L}$ and a window $\mc{W}$ & \pageref{DefP(W,L)}\\
			
			$\widehat{\mc{P}}$ & subset of visible points of a point set $\mc{P}$ & \pageref{Def-vis-points}\\
			
			$\mc{P}^x$ & $\mc{P}(\mc{W},\mc{L}(x))$ & \pageref{Def-P^x}\\
			
			$R$ & a specific set of representatives of $\Z^4\cap zH_g$ under the action of $\Gamma_g$ & \pageref{Def-R}\\
			
			$S$ & $\{(z_1,z_2)\in \R^2\mid z_1<0, z_2<z_1+z_2'\}$ & \pageref{Def-S}\\
			
			$T$ & the open triangle with vertices at $(0,0)$, $(1,0)$ and $(0,1)$ & \pageref{Def-T}\\
			
			$T(s)$ & the open triangle with vertices at $(0,0)$ and $(s/\theta(\widehat{\mc{P}}))^{1/2}(1,\pm 1)$ & \pageref{Def-T(s)}\\
			
			$\mc{T}(s)$ & & \pageref{Def-mcT(s)}\\
			
			$\mc{T}'(s)$ & $T(s)\times \mc{W}$ & \pageref{Def-T'(s)}\\ 
			
			$v_1,v_2$ & vectors depending on $\a_3,\a_4$ & \pageref{Def-v1-v2}\\
			
			$v_{\b_1,\b_2}$ & $(\b_1,\b_2,\s(\b_1),\s(\b_2))h$ & \pageref{Def-v-b_1-b_2}\\
			
			$\mc{W}$ & a particular octagon in $\R^2$ & \pageref{Def-W}\\
			
			$\mc{W}'$ & $\mc{W}\setminus |\s(\l)|\mc{W}$ & \pageref{Def-W-2}\\
			
			$\mc{W}(z_2')$ & $\{w\in \mc{W}\mid z_2'\delta(w)\leq C_3\l^{-2r}s^{-1/2}\}$ & \pageref{Def-W(z)}\\
			
			$X$ & $\Gamma\bs \Gamma H_g$ & \pageref{Def-X}\\
			
			$X(z)$ & $\{\Gamma h\in X\mid h\in H_g,z\in \Z^4h\}\subset X$ & \pageref{Def-X(z)}\\
			
			$X(k,z)$ & $\{\Gamma h\in X\mid h\in H_g,kh=z\}\subset X$ & \pageref{Def-X(k,z)}\\
			
			$\overline{z}$ & $\delta^{-1/4}z\inv{g}$ for $z\in \R^4$ & \pageref{Def-z-bar}\\
			
			$zH_g$ & $\{zh\mid h\in H_g\}\subset \R^4$ & \pageref{Def-zH_g}\\
			
			$\Z^4_*$ & $\Z^4\smpt{0}$ & \pageref{Def-Z^4_*}\\		
			
			$\Gamma$ & $\SLZ{4}$ & \pageref{Def-Gamma}\\
			
			$\Gamma_g$ & $\Gamma\cap H_g$ & \pageref{Def-Gamma_g}\\
			
			$\Gamma_K$ & the set of all matrices of the form $\mathrm{diag}(A,\s(A))$, $A\in \SL{2}{\mc{O}_K}$ & \pageref{Def-Gamma_K}\\
			
			$\Delta$ & a finite subset of $\mc{O}_K$ corresponding to exceptional sublattices & \pageref{Def-Delta}\\
			
			$\delta(w)$ & $\inf_{\varphi\in \R/2\pi\Z}|\{t\in\R: w+t(\cos\varphi,\sin\varphi)\in \mc{W}\}|$ & \pageref{eqn-delta-w}\\
			
			$\theta(\mc{P})$ & density of the point set $\mc{P}$ & \pageref{Def-density}\\
			
			$\l$ & $1+\sqrt{2}$, the fundamental unit of $\mc{O}_K$ & \pageref{Def-lambda}\\
			
			$\mu$ & the unique right $H_g$-invariant probability measure on $X$& \pageref{Def-mu}\\
			
			$\mu_z$ & Haar measure on $(H_g)_z$ & \pageref{Def-mu_z}\\
			
			$\nu_z$ & a measure on $X(k,z)$; a measure on $X(z)$ & \pageref{Def-nu_z}\\
			
			$\overline{\nu}_z$ & $\nu_{\overline{z}}$ & \pageref{Def-nu-bar}\\
			
			$\Pi_v$ & $v+\R b_1+\R b_2$ & \pageref{Def-Pi_v}\\
			
			$\s$ & The non-trivial automorphism of $\bb{Q}(\sqrt{2})$ & \pageref{Def-sigma}\\
		\end{longtable}
	\end{center}	
	
\end{document}